\documentclass[10 pt]{amsart}

\usepackage{amsmath}
\usepackage{amssymb}
\usepackage{mathrsfs}
\usepackage[all,cmtip]{xy}
\usepackage{pigpen}
\usepackage{graphicx}
\usepackage{verbatim}
\usepackage{enumitem}\usepackage{xcolor}
\numberwithin{equation}{subsubsection}
\newtheorem{introthm}[subsection]{Theorem}

\newtheorem{prop}[subsubsection]{Proposition}
\newtheorem{thm}[subsubsection]{Theorem}

\newtheorem*{thm*}{Theorem}
\newtheorem{lemma}[subsubsection]{Lemma}
\newtheorem*{lem*}{Lemma}
\newtheorem{cor}[subsubsection]{Corollary}

\theoremstyle{definition}

\theoremstyle{remark}

\newtheorem*{assump*}{Assumption}

\theoremstyle{definition}
\newtheorem{definition}[subsubsection]{Definition}

\theoremstyle{remark}
\newtheorem{remark}[subsubsection]{Remark}
\newtheorem{remarks}[subsubsection]{Remarks}

\newcommand{\R}{\mathbb{R}}  

\newcommand{\C}{\ensuremath{\mathbb{C}}}
\newcommand{\Z}{\ensuremath{\mathbb{Z}}}

\newcommand{\D}{\ensuremath{\mathbb{D}}}

\newcommand{\A}{\ensuremath{\mathbb{A}}}

\newcommand{\F}{\ensuremath{\mathbb{F}}}

\newcommand{\der}{\ensuremath{\mathrm{der}}}
\newcommand{\ab}{\ensuremath{\mathrm{ab}}}
\newcommand{\ad}{\ensuremath{\mathrm{ad}}}

\newcommand{\Sh}{\ensuremath{\mathrm{Sh}}}
\newcommand{\Q}{\ensuremath{\mathbb{Q}}}
\newcommand{\xbar}{\ensuremath{{\overline{x}}}}

\newcommand{\Ok}{\ensuremath{\mathcal{O}}}
\renewcommand{\O}{\ensuremath{\mathcal{O}}}

\newcommand{\G}{\ensuremath{\mathcal{G}}}
\newcommand{\GG}{\ensuremath{\mathbb{G}}}
\newcommand{\uG}{\ensuremath{\underline{G}}}
\newcommand{\ucalG}{\ensuremath{\underline{\mathcal{G}}}}
\newcommand{\pdiv}{\ensuremath{\mathscr{G}}}

\newcommand{\Spf}{\ensuremath{\mathrm{Spf}}}
\newcommand{\Res}{\ensuremath{\mathrm{Res}}}
\newcommand{\GL}{\ensuremath{\mathrm{GL}}}
\newcommand{\bGL}{\ensuremath{\mathbf{GL}}}
\newcommand{\GSp}{\ensuremath{\mathrm{GSp}}}
\newcommand{\bGSp}{\ensuremath{\mathbf{GSp}}}
\newcommand{\SL}{\ensuremath{\mathrm{SL}}}
\newcommand{\SU}{\ensuremath{\mathrm{SU}}}

\newcommand{\Gr}{\ensuremath{\mathrm{Gr}}}

\newcommand{\ur}{\ensuremath{\mathrm{ur}}}

\newcommand{\hW}{\ensuremath{\widehat{W}}}
\newcommand{\Adm}{\ensuremath{\mathrm{Adm}}}
\newcommand{\Gal}{\ensuremath{\mathrm{Gal}}}
\newcommand{\fka}{\ensuremath{\mathfrak{a}}}

\newcommand{\fkf}{\ensuremath{\mathfrak{f}}}

\newcommand{\fkm}{\ensuremath{\mathfrak{m}}}

\newcommand{\fks}{\ensuremath{\mathfrak{s}}}

\newcommand{\fkM}{\ensuremath{\mathfrak{M}}}

\newcommand{\fkS}{\ensuremath{\mathfrak{S}}}

\newcommand{\bbA}{\ensuremath{\mathbb{A}}}

\newcommand{\bbC}{\ensuremath{\mathbb{C}}}
\newcommand{\bbD}{\ensuremath{\mathbb{D}}}

\newcommand{\bbF}{\ensuremath{\mathbb{F}}}
\newcommand{\bbG}{\ensuremath{\mathbb{G}}}
\newcommand{\bbH}{\ensuremath{\mathbb{H}}}

\newcommand{\bbL}{\ensuremath{\mathbb{L}}}
\newcommand{\bbM}{\ensuremath{\mathbb{M}}}
\newcommand{\bbN}{\ensuremath{\mathbb{N}}}

\newcommand{\bbQ}{\ensuremath{\mathbb{Q}}}
\newcommand{\bbR}{\ensuremath{\mathbb{R}}}
\newcommand{\bbS}{\ensuremath{\mathbb{S}}}

\newcommand{\bbZ}{\ensuremath{\mathbb{Z}}}

\newcommand{\bfC}{\ensuremath{\mathbf{C}}}

\newcommand{\bfE}{\ensuremath{\mathbf{E}}}

\newcommand{\bfG}{\ensuremath{\mathbf{G}}}
\newcommand{\bfH}{\ensuremath{\mathbf{H}}}

\newcommand{\bfL}{\ensuremath{\mathbf{L}}}

\newcommand{\bfT}{\ensuremath{\mathbf{T}}}

\newcommand{\bfZ}{\ensuremath{\mathbf{Z}}}
\newcommand{\bfu}{\ensuremath{\mathbf{u}}}

\newcommand{\tG}{\ensuremath{\widetilde{\mathcal{G}}}}
\newcommand{\ta}{\ensuremath{\widetilde{\alpha}}}

\newcommand{\rmB}{\ensuremath{\mathrm{B}}}

\newcommand{\rmE}{\ensuremath{\mathrm{E}}}
\newcommand{\rmF}{\ensuremath{\mathrm{F}}}

\newcommand{\rmH}{\ensuremath{\mathrm{H}}}

\newcommand{\rmK}{\ensuremath{\mathrm{K}}}

\newcommand{\rmM}{\ensuremath{\mathrm{M}}}

\newcommand{\scrA}{\ensuremath{\mathscr{A}}}

\newcommand{\scrG}{\ensuremath{\mathscr{G}}}

\newcommand{\scrP}{\ensuremath{\mathscr{P}}}

\newcommand{\scrS}{\ensuremath{\mathscr{S}}}

\newcommand{\calA}{\ensuremath{\mathcal{A}}}
\newcommand{\calB}{\ensuremath{\mathcal{B}}}
\newcommand{\calC}{\ensuremath{\mathcal{C}}}

\newcommand{\calF}{\ensuremath{\mathcal{F}}}
\newcommand{\calG}{\ensuremath{\mathcal{G}}}
\newcommand{\calH}{\ensuremath{\mathcal{H}}}
\newcommand{\calI}{\ensuremath{\mathcal{I}}}

\newcommand{\calK}{\ensuremath{\mathcal{K}}}
\newcommand{\calL}{\ensuremath{\mathcal{L}}}
\newcommand{\calM}{\ensuremath{\mathcal{M}}}

\newcommand{\calO}{\ensuremath{\mathcal{O}}}

\newcommand{\calS}{\ensuremath{\mathcal{S}}}
\newcommand{\calT}{\ensuremath{\mathcal{T}}}
\newcommand{\calU}{\ensuremath{\mathcal{U}}}
\newcommand{\calV}{\ensuremath{\mathcal{V}}}

\newcommand{\calY}{\ensuremath{\mathcal{Y}}}
\newcommand{\calZ}{\ensuremath{\mathcal{Z}}}
\newcommand{\Mloc}{\ensuremath{\mathbb{M}^{\mathrm{loc}}}}

\newcommand{\brE}{\ensuremath{\breve{E}}}
\newcommand{\brF}{\ensuremath{\breve{F}}}

\newcommand{\brQ}{\ensuremath{\breve{\mathbb{Q}}_p}}

\newcommand{\brZ}{\ensuremath{\breve{\mathbb{Z}}_p}}

\newcommand{\et}{\ensuremath{\mathrm{\acute{e}t}}}

\newcommand{\Hom}{\ensuremath{\mbox{Hom}}}

\newcommand{\Fr}{\ensuremath{\mathrm{Fr}}}
\newcommand{\Conj}{\ensuremath{\mathrm{Conj}}}

\newcommand{\sa}{\ensuremath{s_{\alpha,0}}}

\newcommand{\po}{\ar@{}[dr]|{\text{\pigpenfont R}}}
\newcommand{\pb}{\ar@{}[dr]|{\text{\pigpenfont J}}}
\begin{document}



\title[Independence of $\ell$ for Frobenius conjugacy classes]{Independence of $\ell$ for Frobenius conjugacy classes attached to Abelian varieties}
\author{Mark Kisin and Rong Zhou}
\begin{abstract}
Let $A$ be an abelian variety over a number field $\rmE\subset \bbC$ and let $\bfG$ denote the Mumford--Tate group of $A$. 
After replacing $\rmE$ by a finite extension, the action of the absolute Galois group $\mathrm{Gal}(\overline{\rmE}/\rmE)$ on the $\ell$-adic cohomology $\rmH^1_{\mathrm{\acute{e}t}}(A_{\overline{\rmE}},\bbQ_\ell)$  factors through $\bfG(\bbQ_\ell).$ 
We show that for $v$ an odd prime of $\rmE$ where $A$ has good reduction, the conjugacy class of Frobenius $\mathrm{Frob}_v$ in 
$\bfG(\bbQ_\ell)$ is independent of $\ell$. Along the way, we prove that under certain hypotheses, every point in the $\mu$-ordinary locus of the special fiber of Shimura varieties has a special point lifting it.
\end{abstract}

\address{Department of Mathematics, Harvard University, Cambridge, USA, MA 02138}
\email{kisin@math.harvard.edu}

\address{Department of Pure Mathematics and Mathematical Statistics, University of Cambridge, Cambridge, UK, CB3 0WA}
\email{rz240@dpmms.cam.ac.uk}

\maketitle
\tableofcontents
\section{Introduction}

Let $A$ be an abelian variety over a number field $\rmE\subset \bbC$ and $\overline{\rmE}$ the algebraic closure of $\rmE$ in $\bbC$.
We fix a prime $p$  and $v|p$ a place of $\rmE$ where $A$ has good reduction. Then for any prime $\ell\neq p$, the action of $\mathrm{Gal}(\overline{\rmE}/\rmE)$ 
on the $\ell$-adic cohomology $\rmH^1_{\et}(A_{\overline{\rmE}},\bbQ_\ell)$ is unramified at $v$, and the characteristic polynomial $P_{v,\ell}(t)$ of a geometric Frobenius $\mathrm{Frob}_v\in \mathrm{Gal}(\overline{\rmE}/\rmE)$ has coefficients in $\mathbb Z,$ and is independent of $\ell.$ The aim of this paper is to 
prove a refinement of this statement for the image of $\mathrm{Frob}_v$ in the {\em Mumford--Tate} group of $A.$

Recall that the Mumford--Tate group $\bf G$ of $A$ is a reductive group over $\bbQ$, defined as the Tannakian group of the  $\bbQ$-Hodge structure 
given by the Betti cohomology $V_B:=\rmH^1_B(A(\bbC),\bbQ).$ It may also be defined as the stabilizer in $\mathbf{GL}(V_B)$ of all Hodge cycles of type (0,0)  on the tensor spaces $V_B^{\otimes r}\otimes (V_B^\vee)^{\otimes r}$ for $r\in \bbZ_{\geq0}$. A fundamental result of Deligne \cite{De1} asserts that there exists a finite extension $\rmE'/\rmE$ in $\overline{\rmE}$ such that for any prime $\ell,$ the action of $\mathrm{Gal}(\overline{\rmE}/\rmE')$ on $\rmH^1_\et(A_{\overline{\rmE}},\bbQ_{\ell})$ is induced by a representation 
$$\rho^{\bfG}_\ell:\mathrm{Gal}(\overline{\rmE}/\rmE')\rightarrow \bfG(\bbQ_\ell).$$
It is not hard to see that for any finite extension $\rmE'/\rmE$, if $\rho^{\bfG}_\ell$ exists for one $\ell,$ then it exists for all $\ell.$ 
Moreover there is a minimal such extension $\rmE'.$ 
The existence of $\rho^{\bfG}_{\ell}$ is in fact predicted by the (in general still unproved) Hodge conjecture for $A$; Deligne's result on absolute Hodge cycles \cite{De1} provides a reasonable substitute in this case so that the existence is unconditional. Upon replacing $\rmE$ by $\rmE'$, we assume there is a map $\rho^{\bfG}_\ell:\mathrm{Gal}(\overline{\rmE}/\rmE)\rightarrow \bfG(\bbQ_\ell)$.

For any reductive group $\bfH$ over $\bbQ$ we write $\mathrm{Conj}_{\bfH}$ for the variety of semisimple 
conjugacy classes of $\bfH$ (cf. \S\ref{subsec: Conj_H})  and $\chi_{\bfH}:\bfH\rightarrow \mathrm{Conj}_{\bfH}$ for the natural projection map which sends an element of $\bfH$ to the associated conjugacy class of its semisimple part. We thus obtain a well-defined element $$\gamma_\ell = \gamma_{\ell}(v) : = \chi_{\bfG}(\rho_\ell^{\bfG}(\mathrm{Frob}_v))\in\Conj_{\bfG}(\bbQ_\ell),$$ the conjugacy class of $\ell$-adic Frobenius at $v$. Our main theorem is the following.

\begin{introthm}\label{introthm: main}
Let $p>2$ and $v|p$  a prime of $\rmE$ where $A$ has good reduction. Then there exists $\gamma\in \mathrm{Conj}_{\bfG}(\bbQ)$ such that $$\gamma=\gamma_\ell\in\mathrm{Conj}_{\bfG}(\bbQ_\ell), \ \forall\ell\neq p.$$
\end{introthm}
	Explicitly, $\gamma$ is a $\mathrm{Gal}(\overline\bbQ/\bbQ)$-stable $\bfG(\overline\bbQ)$-conjugacy class whose associated $\bfG(\overline{\bbQ}_\ell)$-conjugacy class contains $\rho_\ell^\bfG(\mathrm{Frob}_v)$ for all $\ell\neq p$. 
 Since $P_{v,\ell}(t)$ is independent of $\ell,$ the image of $\gamma_\ell$ in $\mathrm{Conj}_{\bfG\bfL(V_B)}(\bbQ_\ell)$ is defined over $ \bbQ$ and independent of $\ell$. However, in general the map $\mathrm{Conj}_{\bfG}(\bbQ)\rightarrow \mathrm{Conj}_{\bfG\bfL(V_B)}(\bbQ)$ is not injective, so the theorem gives more information than the $\ell$-independence of $P_{v,\ell}(t)$.
We remark that $\bfG$ depends on the chosen embedding $\rmE \subset \C$; in general changing the embedding has the effect of replacing $\bfG$ by an inner form. However, the variety $\mathrm{Conj}_{\bfG}$ and the elements $\gamma_{\ell}$ do not 
depend on the choice of this embedding, and so neither does the statement of the theorem. 
When $\bfG$ is quasi-split at $p$ it is expected that $\gamma$ can be lifted to an element $\gamma_0 \in \bfG(\Q),$ which is 
$\bfG(\Q_{\ell})$-conjugate to $\rho_\ell^{\bfG}(\mathrm{Frob}_v).$ We prove a slightly weaker version of this result, when $\bfG^{\der}$ 
is simply connected; see \S\ref{sec: liftability} below for this and other potential refinements of the theorem.

An analogue of the above theorem for any algebraic variety (or more generally motive) over a number field was conjectured by Serre in \cite[12.6]{Serre}, but in general one does not even know the analogue of Deligne's theorem on the existence of $\rho^{\bfG}_{\ell}.$

Previously proved cases of our theorem include a result of Noot who showed a version of this theorem where $\mathrm{Conj}_{\bfG}$ is replaced by a certain quotient $\mathrm{Conj}_{\bfG_{A}}'$ and under the additional assumption that the Frobenius element $\gamma_\ell$ is weakly neat \cite{Noot}.  
More recently, one of us \cite[Corollary 2.3.1]{Ki3} proved the Theorem when $\bfG$ is unramified at $p$.  In fact, \cite{Ki3} proves the stronger result that $\gamma$ lifts to $\gamma_0\in\bfG(\bbQ)$ and is $\bfG(\bbQ_\ell)$-conjugate to $\rho_\ell^{\bfG}(\mathrm{Frob}_v)$.
Noot's argument uses the independence of $\ell$ of $P_{v,\ell}(t),$ together with group theoretic arguments to analyze the map 
$\mathrm{Conj}_{\bfG}\rightarrow \mathrm{Conj}_{\bfG\bfL(V_B)}.$ The result of \cite{Ki3} is proved by showing that, on the Shimura variety associated to $\bfG,$ the isogeny class corresponding to $A$ contains a point which admits a CM lift.

For the rest of the introduction we assume $p>2$. Our proof of Theorem \ref{introthm: main} makes use of families of abelian varieties with Mumford--Tate group contained in $\bfG$, 
and especially the structure of their mod $p$ reductions. These families are parameterized by a Shimura variety $\mathrm{Sh}_{\rmK}(\bfG, X)$ associated to $\bfG,$ and defined over a number field (its reflex field) $\bfE\subset \bbC$ which is contained in $\rmE$. We take $\rmK = \rmK_p\rmK^p$ with $\rmK_p \subset \bfG(\Q_p)$ a parahoric subgroup and $\rmK^p \subset \bfG(\bbA_f^p)$ a compact open subgroup. 
Let $w$ be the restriction of $v$ to  $\bfE.$ Write $\bfE_w$  for the completion of $\bfE$ at $w,$ $\O_{\bfE_w}$ for the ring of integers of $\bfE_w$ and $\kappa(w)$ for its residue field. 
Under some mild conditions, $\mathrm{Sh}_{\rmK}(\bfG,X)$ has an integral model $\scrS_{\rmK}(\bfG,X)$ over 
$\O_{\bfE_w},$ which is smoothly equivalent to a ``local model", defined as the closure of an orbit of $\bfG$ acting on a certain Grassmannian. The construction of these models, which extends results of the first author and Pappas \cite{KP}, originally appeared in a previous version of this article, but is now contained in \cite{KPZ}. 

For each prime $\ell \neq p,$ $\scrS_{\rmK}(\bfG,X)$ is equipped with a  $\bfG(\bbQ_\ell)$-torsor $\bbL_{\ell}.$ 
In particular, for any finite extension $\kappa/\kappa(w)$ and 
$x \in \scrS_{\rmK}(\bfG,X)(\kappa),$ the $q = |\kappa|$-Frobenius acting on the geometric fiber of $\bbL_{\ell}$ at $x,$ 
gives rise to an element $\gamma_{x,\ell}\in\mathrm{Conj}_{\bfG}(\bbQ_\ell).$ 
We say $x$ has the  property ($\ell$-ind), or the \textit{$\ell$-independence property}, if there exists an element $\gamma\in \Conj_{\bfG}(\bbQ)$  such that $$\gamma=\gamma_{x,\ell}\in\Conj_{\bfG}(\bbQ_\ell), \forall   \ell\neq p.$$ 

Now suppose that $(\bfG,X)$ satisfies the conditions needed to guarantee the existence of  $\scrS_{\rmK}(\bfG,X)$; the general case of Theorem \ref{introthm: main} is eventually reduced to this one. 
Then for a suitable choice of $\rmK$, our abelian variety $A$ corresponds to a point $\tilde{x}_A\in \mathrm{Sh}_{\rmK}(\bfG,X)(\rmE)$ and its mod $v$ reduction is a point $x_{A}$ of the special fiber $\calS_{\rmK} := \scrS_{\rmK}(\bfG,X)\otimes_{\calO_{\bfE_w}}\kappa(w).$ Moreover there is an equality $\gamma_\ell(v)=\gamma_{x_A,\ell}$ as elements of $\mathrm{Conj}_{\bfG}(\bbQ_\ell)$. 
Thus in order to show Theorem
\ref{introthm: main}, it suffices to prove 
\begin{equation}\tag{$\dagger$} 
\text{If  $\kappa/\kappa(w)$ is finite and $x\in \calS_{\rmK}(\kappa),$ then $x$ satisfies ($\ell$-ind). }
\end{equation}

By considering $A$ as a point on a larger Shimura variety related to a group of the form $\mathrm{Res}_{\rmF/\bbQ}\bfG$ where $\rmF$ is a suitably chosen totally real field, one can show that Theorem \ref{introthm: main} follows from the following special case of $(\dagger)$. 

\begin{introthm}\label{introthm: l indep full}Let $(\bfG,X,\rmK_p)$ be a strongly  acceptable triple. Then for any  $\kappa/\kappa(w)$ finite and $x\in \calS_{\rmK}(\kappa),$ $x$ satisfies ($\ell$-ind).
\end{introthm}
The condition of strong acceptability of the triple $(\bfG,X,\rmK_p)$ is a technical one, and we refer the reader to \S\ref{sec: integral models abelian type} for the definition. We only mention here that the condition implies that $\rmK_p$ is a certain type of maximal compact subgroup of $\bfG(\bbQ_p)$ known as a very special parahoric. These conditions are needed to ensure the integral models satisfy some desirable properties as we explain below.

As a first step towards Theorem \ref{introthm: l indep full}, we prove the following analogue of Serre--Tate theory, which  allows us to show that under the assumptions of Theorem \ref{introthm: l indep full}, ($\ell$-ind) holds on a dense, Zariski open subset of $\calS_{\rmK}$.

\begin{introthm}\label{introthm: can lift for Shimura var} Assume  the triple  $(\bfG,X,\rmK_p)$ is strongly acceptable. Then
\begin{enumerate} 
	\item The $\mu$-ordinary locus  $\calS_{\rmK,[b]_{\mu}}\subset \calS_{\rmK}$  is Zariski open and dense in $\calS_{\rmK}$.
	\item Any closed point  $x$ lying in $\calS_{\rmK,[b]_{\mu}}$ admits a lifting to a special point 
	$\widetilde{x}\in\mathrm{Sh}_{\rmK}(\bfG,X).$

\end{enumerate}
\end{introthm}
The $\mu$-odinary locus in  (1) is the group theoretic generalization of the ordinary locus in the moduli space of principally polarized abelian varieties, and the density follows from results about the local structure of $\scrS_{\rmK}(\bfG,X)$ and \cite[Corollary 1.3.16]{KPS}. 
The lifting constructed in (2) is then the analogue in our setting of  the canonical lift for ordinary abelian varieties and had been considered for Shimura varieties with good reduction in previous work of Moonen \cite{Mo} and Shankar and the second author \cite{SZ}. For these points, the Frobenius lifts to an automorphism of the associated CM abelian variety, and we obtain the desired element $\gamma\in \mathrm{Conj}_{\bfG}(\bbQ)$ by considering the induced action on Betti cohomology.  

To prove Theorem \ref{introthm: l indep full}, one considers a smooth curve $\calC$ with a map $\pi:\calC \rightarrow \calS_{\rmK}.$ Using a theorem of L. Lafforgue \cite[Th\'eor\`eme VII.6]{Laf} on the  existence of compatible local systems on smooth curves, we show that if the property ($\ell$-ind) holds for a dense open subset of points on $\calC$ then it holds for all points of $\calC.$ The results in \cite[\S7]{KPZ} on the structure of the integral models $\scrS_{\rmK}(\bfG,X)$ imply that $\calS_{\rmK}$ is equipped with a certain combinatorially described stratification, 
the Kottwitz--Rapoport stratification.  The stratum of maximal dimension is the smooth locus of $\calS_{\rmK}.$ A theorem of 
Poonen \cite{Poonen} shows that $\pi$ can be chosen so that its image intersects  $\calS_{\rmK,[b]_{\mu}}$ and 
any point $x$ of the maximal stratum. The $\mu$-ordinary case explained above then implies that any such $x$ satisfies ($\ell$-ind). We now argue by induction on the codimension of the strata; for a closed point $x$ in some stratum of $\calS_{\rmK},$ we show that $\pi$ can be chosen so that its image contains $x,$ and also meets some higher dimensional stratum. 

In fact, using general arguments with ampleness, it is not hard to construct a map $\pi$ from a smooth curve whose image contains any closed point $x\in \calS_{\rmK},$  and meets the 
$\mu$-ordinary locus. This would appear to avoid the induction on strata above. However, this argument would only allow us to prove  the $\ell$-independence result for some power of the Frobenius. To prove Theorem \ref{introthm: l indep full} in full, one needs the existence of a $y \in \calC,$ with $\pi(y) = x,$ such that $\pi$ induces an isomorphism of residue fields $\kappa(x) \simeq \kappa(y).$ To construct such curves, we first construct a sequence of smooth curves which are {\em subschemes} of the local model associated to $\scrS_{\rmK}(\bfG,X),$ using the explicit group theoretic description of this local model.  These are then pulled back to $\scrS_{\rmK}(\bfG,X)$ via the local model diagram. The assumption that $\rmK_p$ is very special is key to our argument, as this not only guarantees the density of $\calS_{\rmK,[b]_{\mu}}$, but also that the Kottwitz--Rapoport stratification on the local model has a particularly simple description (cf. \S\ref{sec: KR stratification v special}) which is used in the construction of $\pi$. 

The induction argument would also be unnecessary if one could show a conjecture of Deligne \cite[Conjecture 1.2.10]{De3} on the existence of compatible local systems on a normal variety. For smooth schemes Deligne's conjecture has been proved by Drinfeld 
\cite{Drinfeld}, but the special fiber $\calS_{\rmK}$ is not smooth, so Drinfeld's theorem does not suffice for our purposes. 

We now give some details about the geometric properties of the integral models $\scrS_{\rmK}(\bfG,X)$  that we use.
The two main results we need about these models are the existence of a local model diagram as predicted in \cite{Ra}, 
which relates the models to an orbit closure on a Grassmannian, and the analogue of Serre--Tate theory at $\mu$-ordinary points, 
already mentioned in Theorem \ref{introthm: can lift for Shimura var} above.  

Under some mild assumptions, \cite{KPZ} proves the existence of a version of the local model diagram for abelian type Shimura varieties where the torsor is for the parahoric of the adjoint group. The strategy follows that of \cite{KP} which proved the result under an additional tameness hypothesis on the group $\bfG$. In \cite{KPZ}, the local model diagram is first constructed in some special Hodge type cases; the general case is eventually reduced to this via Deligne's formalism of connected Shimura varieties. Both the reduction step and the construction in the Hodge-type case make crucial use of the notion of $R$-smoothness for tori introduced in \S2.4. This is related to the failure of a closed immersion of tori  to extend to a closed immersion of lft N\'eron models, a phenomenon which does not occur in the tamely ramified case. For us, this notion is needed to prove certain functoriality properties of our integral models in \S4.3.

In the special Hodge type case, the construction of the local model diagram is intertwined with a result (Proposition \ref{prop: formal nbd Shimura}) which gives  a description of the formal neighborhood of the Shimura variety in terms of the deformation theory of a $p$-divisible group equipped with a collection of tensors in its crystalline cohomology. For this we use the constructions in \cite[\S3]{KP}, as generalized in \cite{KPZ}. This result is used as a key input in proving the existence of canonical liftings in these cases,

The special Hodge-type cases  considered above are actually not enough for applications to proving $\ell$-independence for abelian varieties. This is due to pathologies in the local models when $p||\pi_1(\bfG^{\der})|$. In order to prove the results concerning  integral models in the required level of generality, we consider the Hodge-type Shimura datum of interest as a datum of abelian type. The results concerning the local model diagram and canonical liftings are then transferred to the integral model using a different, auxiliary, Hodge-type datum which does satisfy the required properties.  The price of this indirect approach is that we have to do some work to prove that the integral model constructed in this way maps to an appropriate moduli space of abelian varieties.
This is needed in order to define the $\mu$-ordinary locus, and prove Theorem \ref{introthm: can lift for Shimura var}.

We now explain the organization of the paper. In \S 2-4 we prove the geometric result concerning integral models of Shimura varieties we need. 
These are then used to prove Theorem \ref{introthm: main} in \S 5,6. In \S3, we recall results concerning local models and review the deformation theoretic results of \cite[\S3]{KP} and \cite{KPZ}, which are then used to show the existence of canonical deformations for $\mu$-ordinary $p$-divisible groups in \S3.4. The latter uses a generalization to general parahorics of a result of Wortmann on $\mu$-ordinary $\sigma$-conjugacy classes, 
which is proved in \S 2.3. In \S 4, we recall the construction of the local model diagram and prove Theorem \ref{introthm: can lift for Shimura var},
first in some special Hodge-type cases in \S4.1, then in general in \S4.2--4.4. 
 
 In \S 5,  we prove Theorem \ref{introthm: l indep full} following the strategy outlined above and in \S 6 we prove Theorem \ref{introthm: main} using  Theorem \ref{introthm: l indep full}. Finally we remark that for technical reasons related to the level structure on $A$, we actually  work with Shimura stacks (i.e. Shimura varieties where the level structure is not neat) in \S 4-6.

\textit{Acknowledgments:} M.K. was supported by NSF grant  DMS-1902158. 
R.Z. was supported by NSF grant  DMS-1638352  through membership of the Institute for Advanced Study, and  by the European Research Council (ERC) under the European Union’s Horizon 2020 research and innovation programme (grant agreement No. 804176). We would like to thank the referees for many suggestions which greatly improved the paper. We also thank 
Yihang Zhu for useful discussions.

\section{Group theoretic results}\label{sec: group theoretic}

In this section, we prove some group theoretic results which will be used  in \S\ref{sec: integral models for Shimura + canonical lifts}. \S\ref{subsec: sigma straight elements}--\S\ref{sec: parahoric group schemes} contains the results needed for the construction of canonical liftings in \S\ref{sec: M-adapted} and \S\ref{sec: canonical liftings}. In \S\ref{sec: Neron models} we study properties of N\'eron models of tori needed  to study embeddings of parahorics.

\subsection{$\sigma$-straight elements} \label{subsec: sigma straight elements}

\subsubsection{}

Let $F$ be a non-archimedean local field with ring of integers $\Ok_F$. We fix a uniformizer $\varpi_F\in\calO_F$ and we let $k_F$ denote the residue field of $\calO_F$. We let $\brF$ denote the completion of the maximal unramified extension of $F$ and $\Ok_{\brF}$ its ring of integers, and we fix $\overline{F}$ an algebraic closure of $F$. We let $k$ be the residue field of $\calO_{\breve F}$ which is an algebraic closure of $k_F$. We write $\Gamma$ for the absolute Galois group $\mathrm{Gal}(\overline{F}/F)$ of $F$ and $I$ for the inertia subgroup, which is identified with $\text{Gal}(\overline{\brF}/\brF)$. We let $\sigma$ denote the Frobenius element of $\text{Aut}(\brF /F)$. 

Let $S$ be a scheme. If $X$ is a scheme over $S$ and $S'\rightarrow S$ is a morphism of schemes,  we write $X_{S'}$ for the base change of $X$ along $S'\rightarrow S$.

\subsubsection{}Let $G$ be a reductive group\footnote{Our convention is that all reductive groups are connected.}  over $F$. Let $S$ be a maximal $\brF$-split torus of $G$ defined over $F$ and $T$ its centralizer (cf. \cite[1.10]{Ti1}  for the existence of $S$). By Steinberg's Theorem, $G$ is quasi-split over $\brF$ and $T$ is a maximal torus of $G$.  We let $\calB(G,F)$ (resp. $\calB(G,\brF)$) denote the (extended) Bruhat--Tits building of $G$ over $F$ (resp. $\brF$). Let $\mathfrak{a}$ denote a $\sigma$-invariant alcove in the apartment $V:=\calA(G,S,\brF)$ over $\brF$ associated to $S$; we write $\calI$ for the corresponding Iwahori group scheme over $\calO_F$.
The relative Weyl group $W_0$ and the Iwahori Weyl group are defined as \begin{equation}\label{eqn: exact sequence Iwahori Weyl}W_0=N(\brF)/T(\brF),\ \ \ W=N(\brF)/\mathcal{T}_0(\mathcal{O}_{\brF}),\end{equation}
where $N$ is the normalizer of $T$ and $\mathcal{T}_0$ is the connected  N\'eron model for $T$. These are related by an exact sequence 
\[\xymatrix{0\ar[r]& X_*(T)_I\ar[r]& W\ar[r]& W_0\ar[r]& 0.}\]

For an element $\lambda\in X_*(T)_I$ we write $t_\lambda$ for the corresponding element in $W$; such elements will be called translation elements. We will sometimes write $W_G$ or $W_{G_{\brF}}$ for $W$ if we want to specify the group that we are working with.

\subsubsection{}We also fix a special vertex $\fks$ lying in the closure of $\fka$. Such a vertex induces a splitting of the exact sequence (\ref{eqn: exact sequence Iwahori Weyl}) and gives an identification \begin{equation}
\label{eqn: id apartment}V\cong X_*(T)_{I}\otimes_{\bbZ}\bbR.\end{equation} 

Let $\mathrm{Aff}(V)$ denote the group of affine transformations of $V$. Then we have an identification $\mathrm{Aff}(V)\cong V\rtimes \GL(V).$ The Frobenius $\sigma$ acts on $V$ via affine transformations and we write $\varsigma\in \GL(V)$ for the linear part of this action.
The identification (\ref{eqn: id apartment}) also determines a dominant chamber $C_+\subset X_*(T)_{I}\otimes_{\bbZ}\bbR$; namely by taking the one containing $\fka$, and we write $B$ for the corresponding Borel subgroup defined over $\brF$. We write $\sigma_0$ 
for the automorphism of $X_*(T)_{I}\otimes_{\bbZ}\bbR$ defined by $\sigma_0:=w_0\circ\varsigma$ where $w_0\in W_0$ is the unique element such that $w_0\circ \varsigma (C_+)=C_+$. We call this the $L$-action on $X_*(T)_I\otimes_{\bbZ}\bbR$; by definition it preserves $C_+$.

\subsubsection{}\label{subsubsec:bruhatorder}
Let $\mathbb{S}$ denote the set of simple reflections about the walls of $\mathfrak{a}$. We let $W_a$ denote the affine Weyl group; it is the subgroup of $W$ generated by the reflections in $\mathbb{S}$. Then $(W_a,\bbS)$ has the structure of a Coxeter group, and  hence we have a notion of length and Bruhat order. The Iwahori Weyl group and affine Weyl group are related via the following exact sequence
\begin{equation}\label{eqn: exact sequence affine Weyl}\xymatrix{0\ar[r]& W_a\ar[r]&W\ar[r]& \pi_1(G)_I\ar[r]& 0.}\end{equation}
The choice of $\mathfrak{a}$ induces a splitting of this exact sequence and $\pi_1(G)_I$ can be identified with the subgroup $\Omega\subset W$ consisting of elements which preserve $\fka$. The length function $\ell$ and Bruhat order  $\leq$ extend to $W$ via this choice of splitting and $\Omega$ is identified with the set of length $0$ elements.

We let $\widetilde{\kappa}_G(w)$ denote the image of $w \in W$ in $\pi_1(G)_I$ and $\kappa_G(w)$ its projection to $\pi_1(G)_\Gamma$. For $w\in W$, there is an integer $n$ such that $\sigma^n$ acts trivially on $W$ and $w\sigma(w)\dotsc\sigma^{n-1}(w)=t_\lambda$ for some $\lambda\in X_*(T)_I$. We define the (non-dominant) Newton cocharacter $\nu_w\in X_*(T)_{I,\bbQ}\cong X_*(T)^I_{\bbQ}$ to be $\frac{1}{n}\lambda$, which is easily seen to be independent of $n$. We let $\overline{\nu}_w\in X_*(T)^{I,+}_{\bbQ}$ be the dominant representative of $\nu_w$.

\subsubsection{}\label{sec: affine Weyl group}
Let $G^{\der}$ be the derived group of $G$ and let $T^{\mathrm{sc}}$ denote the preimage of $T$ in the simply connected covering $G^{\mathrm{sc}}$ of $G^{\der}$. Then $W_a$ is the Iwahori Weyl group for $G^{\mathrm{sc}}$  and we have the following exact sequence \[\xymatrix{0\ar[r]& X_*(T^{\mathrm{sc}})_I\ar[r]& W_a\ar[r] &W_0\ar[r]& 0.}\]Since the action of $I$ permutes the set of absolute coroots, $X_*(T^{\mathrm{sc}})_I$  is torsion free and there is an inclusion $X_*(T^{\mathrm{sc}})_I\hookrightarrow X_*(T)_I$. By \cite{HaRa}, there exists a reduced root system $\Sigma$ such that 
$$W_a \simeq Q^\vee(\Sigma)\rtimes W(\Sigma)$$
where $Q^\vee(\Sigma)$ and $W(\Sigma)$ denotes the coroot lattice and Weyl group of $\Sigma$, respectively, and there is a canonical isomorphism $W(\Sigma)\cong W_0$. The roots of $\Sigma$ are proportional to the roots of the relative root system for $G_{\brF}$; however the root systems themselves may not be proportional.

 As explained in \cite[p7]{HaRa}, we may consider elements of $\Sigma$ as functions on $X_*(T)_I\otimes_{\bbZ}\bbR$, and we write $\langle\ ,\ \rangle$ for the induced pairing between $X_*(T)_I\otimes_{\bbZ}\bbR$ and the root lattice associated to $\Sigma$.  We let $\rho$ denote the half sum of all positive roots in $\Sigma$. Then for any $\lambda\in X_*(T)_I$ we have the equality \begin{equation}\label{eqn: length of translation element}\ell(t_\lambda)=\langle\overline{\lambda},2\rho\rangle, \end{equation} where $\overline{\lambda}\in W_0\cdot\lambda$ is the dominant representative of $\lambda$, i.e. \!the image of $\overline{\lambda}$ in $X_*(T)_I\otimes_{\bbZ}\bbR$ lies in $C_+$.

\subsubsection{} We say that an element $w\in W$ is $\sigma$-\textit{straight} if for any $n\in \bbN$,  we have $$\ell(w\sigma(w)\dotsc\sigma^{n-1}(w))=n\ell(w).$$
It is straightforward to check that this is equivalent to the condition $\ell(w)=\langle \overline{\nu}_w,2\rho\rangle$.

In this paper, we are particularly interested in translation elements $t_{\mu'}$ which are also $\sigma$-straight; the key property of these elements that we will need is that they are central for some Levi subgroup of $G$ defined over $F$.

For any $v\in X_*(T)_I\otimes_{\bbZ}\bbR$, we  let $\Phi_{v,0}$  be the set of relative roots $\alpha$ for $G_{\brF}$ such that $\langle v,\alpha\rangle=0$. We  may then associate to $v$ the semi-standard Levi subgroup $M_v\subset G_{\brF}$  generated by $T$ and the  root subgroups $U_\alpha$ corresponding to $\alpha\in\Phi_{v,0}$. If in addition $v$ is fixed by $\varsigma$, then $M_v$ is defined over $F$. We say $\lambda\in X_*(T)_I$ is central in $G$ if it pairs with any relative root (equivalently any root in $\Sigma$) to give $0$.

\begin{lemma}\label{lemma: cochar central}Let $\mu'\in X_*(T)_I$ such that $t_{\mu'}$ is a $\sigma$-straight element and let $M:=M_{\nu_{t_{\mu'}}}$ be the semi-standard Levi subgroup of $G$ associated to the Newton cocharacter $\nu_{t_{\mu'}}$. Then $M$ is defined over $F$ and ${\mu'}$ is central in $M$.	

\end{lemma}
\begin{proof}For any $\lambda\in X_*(T)_I$, and for sufficiently divisible $n$ we have $$n\nu_{\sigma(t_\lambda)}=\sigma(t_\lambda)\dotsc\sigma^n(t_\lambda)=t_{\lambda}^{-1}n\nu_{t_\lambda} t_\lambda=n\nu_{t_{\lambda}}.$$	Note that $\sigma(t_{\lambda})=t_{\varsigma(\lambda)}$; it follows that $\nu_{\sigma(t_{\lambda})}=\varsigma(\nu_{t_{\lambda}})$ and hence $\nu_{t_{\lambda}}$ is fixed by $\varsigma$. Therefore $M$ is defined over $F$.
	
 We let $u\in W_0$ be such that $u(\nu_{t_{\mu'}})=\overline{\nu}_{t_{\mu'}}$. 
	For a sufficiently divisible $n$, we  have 
	$$\ell(t_{\mu'})=\langle\overline{\nu}_{t_{\mu'}},2\rho\rangle=\frac{1}{n}\sum_{i=0}^{n-1}\langle u\varsigma^i(\mu'),2\rho\rangle$$
where the first equality follows from the $\sigma$-straightness of $t_{\mu'}$.	Now $\langle u\varsigma^i(\mu'),2\rho\rangle\leq \ell(t_{\mu'})$ with equality if and only if $u\varsigma^i(\mu')$ is dominant. Therefore $u\varsigma^i(\mu')$ is  dominant for all $i$ and hence $\varsigma^i(\mu')$ is contained in the translate $C'$ of the dominant chamber $C_+$ by $u^{-1}$ for all $i$.
	
	Now $M$ corresponds to a sub-root system $\Sigma_M$ of $\Sigma$ consisting of the roots $\alpha\in \Sigma$ such that $\langle \nu_{t_{\mu'}},\alpha\rangle=0$. Then $\Sigma_M$ is also the reduced root system associated to the affine Weyl group for $M$ as in \S\ref{sec: affine Weyl group}. We must show for all $\alpha\in \Sigma_M$, we have $\langle{\mu'},\alpha\rangle=0$.
	Let $\alpha\in \Sigma_M$ be a root, then since $\varsigma^i(\mu')$ is contained in a single Weyl chamber for all $i$, it follows that $\langle  \varsigma^i(\mu'),\alpha\rangle$ have the same sign for all $i$. 
	
	Without loss of generality, assume $\langle \varsigma^i(\mu'),\alpha\rangle\geq 0, \forall i.$
	Then we have \begin{equation}
	\begin{split}0&=\langle \nu_{t_{\mu'}},\alpha\rangle =\frac{1}{n}\sum_{i=0}^{n-1}\langle\varsigma^i(\mu'),\alpha\rangle
	\end{split}.
	\end{equation}
	Since all  terms in the sum are non-negative, they must be $0$. Hence $\mu'$ is central in $M$.
\end{proof}

\subsubsection{}Now let $\{\mu\}$ be a geometric conjugacy class of cocharacters of $G$. Let $\mu\in X_*(T)_I$ denote the image of a dominant (with respect to the choice of Borel $B$ defined above) representative $\widetilde{\mu}\in X_*(T)$ of $\{\mu\}$.

\begin{lemma}\label{lemma: cochar central 2}
		Let  $w\in W_0$ such that  for $\mu':=w({\mu})$, $t_{\mu'}$ is a $\sigma$-straight element. Let $\widetilde{\lambda}:=w(\widetilde{\mu})\in X_*(T)$. Then $\widetilde{\lambda}$ is central in $M:=M_{\nu_{t_{\mu'}}}$. Here, we consider $W_0$ as a subgroup of the absolute Weyl group for $G$.
\end{lemma}
\begin{proof}	
		Let $w(C_+)\subset X_*(T)_I\otimes_{\bbZ}\bbR$ be the translate of the dominant chamber by $w$. Then $w(C_+)$ determines a chamber $C_M$ for $M$ (it is the unique chamber for $M$ such that $w(C_+)\subset C_M$) and we have $\mu'\in C_M$. The chamber $C_M$ determines an ordering of the root system $\Sigma_M$. Let ${\alpha}$ be a positive root for $\Sigma_M$ and $\widetilde{\alpha}\in X^*(T)$ an (absolute) root lifting ${\alpha}$; such a lift exists by the construction of $\Sigma$, see eg. \cite[VI, 2.1]{Bourbaki}. We let $(\ ,\ ):X_*(T)\times X^*(T)\rightarrow \bbZ$ denote the natural pairing.
	
	Let $K/\brF$ be a finite Galois extension over which $T$ splits. We have by definition of $\Sigma_M$ $$0=\langle \mu',{\alpha}\rangle=c\sum_{\tau\in\mathrm{Gal}(K/\brF)}( \widetilde{\lambda},\tau(\widetilde{\alpha}))$$for some positive $c\in \mathbb{R}$, where the first equality follows since $\mu'$ is central in $M$.
	For any $\tau\in \mathrm{Gal}(K/\brF)$, $C_M$ is preserved by $\tau$ and hence $\tau(\widetilde{\alpha})$ is a positive root for $M$. Therefore $(\widetilde{\lambda},\tau(\widetilde{\alpha}))\geq 0$, and hence $(\widetilde{\lambda},\tau(\widetilde{\alpha}))=0$ for all $\tau$. Applying this to every relative root $\alpha$ for $M$, we see that $\widetilde\lambda$ is central in $M$.
	\end{proof}
\subsection{$\mu$-ordinary $\sigma$-conjugacy classes}

\subsubsection{}\label{sec: mu admissible set}Let $\{\mu\}$ be a geometric conjugacy class of cocharacters of $G$; we let $\widetilde{\mu}\in X_*(T)$ and $\mu\in X_*(T)_I$ as above. The $\mu$-admissible set is defined to be $$\Adm(\{\mu\})=\{w\in W|w\leq t_{x({\mu})} \text{ for some }x\in W_0\}.$$
It has a unique minimal element denoted  $\tau_{\{\mu\}},$ which is also the unique element of $\Adm(\{\mu\})\cap\Omega$. 

For $b\in G({\brF})$, we let $[b]$ denote the set $\{g^{-1}b\sigma(g)|g\in G({\brF})\}$, the $\sigma$-conjugacy class of $b$. The set of $\sigma$-conjugacy classes $B(G)$ has been classified by Kottwitz in \cite{Ko2} and \cite{Ko1}. For $b\in G({\brF})$, we let $\nu_b:\bbD\rightarrow G_{\brF}$ denote its Newton cocharacter and $$\overline{\nu}_b\in X_*(T)_{I,\Q}^+\cong X_*(T)_{\Q}^{I,+}$$  the dominant representative for $\nu_b$; it is known that $\overline{\nu}_b$ is invariant under the action of $\sigma_0$. We let $\widetilde{\kappa}_G:G(\brF)\rightarrow \pi_1(G)_I$ denote the Kottwitz homomorphism and we write $$\kappa_G:G({\brF})\rightarrow \pi_1(G)_\Gamma$$  for the composition of $\widetilde{\kappa}_G$ and the projection map $\pi_1(G)_I\rightarrow \pi_1(G)_\Gamma$. This induces a well-defined map $B(G)\rightarrow \pi_1(G)_\Gamma$, also denoted $\kappa_G$. Then there is an injective map \begin{equation}\label{eqn: Kottwitz classification}B(G)\xrightarrow{(\kappa_G,b\mapsto \overline{\nu}_b)}\pi_1(G)_\Gamma\times (X_*(T)_{\Q}^{I,+})^{\sigma_0}.\end{equation}

\subsubsection{}
There is a more explicit description of this map using the Iwahori Weyl group $W$. For $w\in W$, its $\sigma$-conjugacy class is the set $\{u^{-1}w\sigma(w)| u\in W\}$. We let $B(W,\sigma)$ denote the set of $\sigma$-conjugacy  classes in $W$. For $w\in W$, we let $\dot{w}\in N(\brF)$ be a lift of $w$. Then to $w\in W$, we associate the $\sigma$-conjugacy class of $\dot{w}$; by Lang's theorem this does not depend on the choice of representative $\dot{w}$. We write $$\Psi:B(W,\sigma)\rightarrow B(G)$$ for the map induced by $w\mapsto [\dot{w}]$.

By \cite[Theorem 3.7]{He1}, $\Psi$ is surjective and we have a commutative diagram
\begin{equation}
\label{eqn: B(G) Iwahori Weyl}
\xymatrix{B(W, \sigma) \ar@{->>}[rr]^{\Psi} \ar[dr]_{(\overline\nu,\kappa_G)} & & B(G) \ar@{^{(}->}[ld]^{(\overline\nu,\kappa_G)} \\ & (X_*(T)^{I,+}_ \bbQ) \times \pi_1(G)_\Gamma &}.\end{equation}
The map $\Psi$ is not injective in general, however it is proved in \cite[Theorem 3.7]{He1} that  its restriction to the set of $\sigma$-straight $\sigma$-conjugacy classes is a bijection. Here, a $\sigma$-conjugacy class in $W$ is said to be $\sigma$-straight if it contains a $\sigma$-straight element.

\subsubsection{} Note that there is a partial order on the set  $X_*(T)_\bbQ^+$; for $\lambda,\lambda'\in
X_*(T)_\bbQ^+$, we write $\lambda\leq\lambda'$ if $\lambda'-\lambda$ is a non-negative rational linear combination of positive  roots. For $\{\mu\}$ as above, we write $\mu^\natural$ for the common image of an element of $\{\mu\}$ in $\pi_1(G)_\Gamma$ and we define $$\mu^\diamond=\frac{1}{N}\sum_{i=1}^N\sigma_0^i(\mu)\in X_*(T)^+_{I,\bbQ}\cong X_*(T)^{I,+}_\bbQ.$$
where $N$ is the order of the element $\sigma_0$ giving rise to the $L$-action on $X_*(T)_I\otimes_{\bbZ}\bbQ$. We set $$B(G,\{\mu\})=\{[b]\in B(G):\kappa_G(b)=\mu^\natural,\overline{\nu}_b\leq\mu^\diamond\}.$$

Note that for $[b_1],[b_2]\in B(G,\{\mu\})$, we have $[b_1]=[b_2]$  if and only if $\overline{\nu}_{[b_1]}=\overline{\nu}_{[b_2]}$, since $[b_1]$ and $[b_2]$ have common image $\mu^\natural$ under $\kappa_G$.

\begin{definition}\label{def: mu ordinary}
	Suppose there  exists a class $[b]\in B(G,\{\mu\})$ such that $\overline{\nu}_{[b]}=\mu^\diamond$ (such a class is necessarily unique if it exists by the above remark). We write $[b]_{\mu}$ for this class; it is called the $\mu$-ordinary $\sigma$-conjugacy class.
\end{definition}

\begin{remark}	\cite[Theorem 1.1]{HeNie2} have shown that $B(G,\{\mu\})$ always contains a  unique
maximal element with respect to the partial order $\leq$. When $G$ is quasi-split, this class is just $[b]_{\mu}$. However if $G$ is not quasi-split, there may be no $[b]\in B(G,\{\mu\})$ such that $\overline{\nu}_{[b]}=\mu^\diamond$.
\end{remark}

\begin{lemma}\label{lemma: ord class rep by straight elt}Assume there exists $[b]_{\mu}\in B(G,\{\mu\})$ with $\overline{\nu}_{[b]_{\mu}}=\mu^\diamond$. There exists $\mu'\in W_0\cdot{\mu}$ such that $t_{\mu'}$  is $\sigma$-straight and $\dot{t}_{\mu'}\in [b]_{\mu}.$ 
\end{lemma}
\begin{proof}Since $[b]_{\mu}\in B(G,\{\mu\})$, there exists a $\sigma$-straight element $w\in \Adm(\{\mu\})$  such that $\dot{w}\in [b]_{\mu}$ by \cite[Theorem 4.1]{He3}. The commutativity of diagram (\ref{eqn: B(G) Iwahori Weyl}) implies that $\overline{\nu}_w=\mu^\diamond$. Since $w$ is $\sigma$-straight, we have
	$$\ell(w)=\langle\overline{\nu}_w,2\rho\rangle=\langle\mu^\diamond,2\rho\rangle=\langle{\mu},2\rho\rangle=\ell({t_\mu}),$$
	where the third equality follows from the fact $\rho$ is invariant under $\sigma_0$, and the final equality uses (\ref{eqn: length of translation element}) and the fact that $\mu$ is dominant. 
		Since $w\in\Adm(\{\mu\})$, $\ell(w)\leq \ell(t_{{\mu}})$ with equality if and only if $w=t_{\mu'}$ for some $\mu'\in W_0\cdot {\mu}$.
	\end{proof}
	
	\subsubsection{} Now let $G'$ be another reductive group over $F$ and $f: G\rightarrow G'$ a group scheme morphism which induces an isogeny $G^{\der}\rightarrow G'^{\der}$. We write $\{\mu'\}$ for the $G'$-conjugacy class of cocharacters induced by $\{\mu\}$.
We have the following relationship between $\mu$-ordinary $\sigma$-conjugacy classes for $G$ and $G'$.

\begin{lemma}\label{lemma: mu-ordinary class change of groups}\begin{enumerate}\item There exists $[b]_{\mu}\in B(G,\{\mu\})$ with $\overline{\nu}_{[b]_{\mu}}=\mu^\diamond$ if and only if there exists $[b']_{\mu'}\in B(G',\{\mu'\})$ with $\overline{\nu}_{[b']_{\mu'}}=\mu'^\diamond$.

	\item	Let $[b]\in B(G,{\mu})$ and $[b']:=[f(b)]\in B(G',\{\mu'\})$. Then $[b]=[b]_{\mu}$ if and only if $[b']= [b']_{\mu'}$.
	\end{enumerate}
\end{lemma}
\begin{proof} (1) Note that we have a commutative diagram \[\xymatrix{  B(G)\ar[r]\ar[d]&(X_*(T)^{I,+}_ \bbQ) \times \pi_1(G)_\Gamma\ar[d]\\ B(G')\ar[r]& (X_*(T')^{I,+}_ \bbQ) \times \pi_1(G')_\Gamma}\]
	where $T'$ is the centralizer of a maximal $\brF$-split torus of $G'$ containing $f(T)$.  Thus one direction of (1) is clear.
	
For the converse, suppose there exists $[b']_{\mu'}\in B(G',\{\mu'\})$. Note that by assumption, there is an identification of relative Weyl groups for $G$ and $G'$. Then by Lemma \ref{lemma: ord class rep by straight elt}, there exists $w_0\in W_0$ such that $t_{w_0(\mu')}$ is a $\sigma$-straight element of the Iwahori Weyl group for $G'$ and $\dot{t}_{w_0(\mu')}\in [b']_{\mu}$. Then it is easy to check that  $t_{w_0(\mu)}$ is a $\sigma$-straight element of the Iwahori Weyl group for $G$ and that $\overline{\nu}_{t_{w_0(\mu)}}=\mu^\diamond$. It follows that $[\dot{t}_{w_0(\mu)}]=[b]_{\mu}\in B(G,\{\mu\})$. 
 
(2) One direction is clear. Suppose then that $[b']=[b']_{\mu'}.$ It follows that $\overline{\nu}_{[b]}=\mu^\diamond +\alpha$ for some $\alpha\in X_*(\ker(G\rightarrow G'))^I$. But $[b]\in B(G,\{\mu\})$ and hence $\mu^\diamond-\overline{\nu}_{[b]}$ is a rational linear combination of positive coroots. Thus $\alpha=0$ and $[b]=[b]_{\mu}$.
\end{proof}

\subsection{Parahoric group schemes}\label{sec: parahoric group schemes}
\subsubsection{}

Recall that $\calB(G,F)$ and $\calB(G,\brF)$ denote the extended Bruhat--Tits buildings  associated to $G$. For a non-empty bounded subset $\Xi\subset \calB(G,F)$ which is contained in an apartment, we let $G(F)_{\Xi}$  (resp. $G(\brF)_\Xi$) denote the subgroup of $G(F)$  (resp. $G(\brF)$) which fixes $\Xi$ pointwise. By the main result of \cite{BT2}, there exists a smooth affine group scheme $\widetilde{\G}_{\Xi}$ over $\Ok_F$ with generic fiber $G$ which is uniquely characterized by the property $\widetilde{\G}_{\Xi}(\Ok_{\brF})=G(\brF)_\Xi$. As in \cite[\S 1.1.2]{KP}, we  will call such a group scheme  the Bruhat--Tits stabilizer scheme associated to $\Xi$. If $\Xi=\{x\}$ is a point we write $G(F)_x$  (resp. $\widetilde{\G}_x$) for $G(F)_{\{x\}}$ (resp. $\widetilde{\G}_{\{x\}}$). 

For $
\Xi\subset \calB(G,F)$, we write $\calG_\Xi$ for the “connected stabilizer” $\Xi$ (cf. \cite[\S4]{BT2}). We caution the reader that our convention differs from \cite{KP}, where $\calG_{\Xi}$ is used for the Bruhat--Tits stabilizer scheme and $\calG_{\Xi}^\circ$ for the connected stabilizer.
We are mainly interested in the cases where $\Xi$ is a point $x$ or an open facet $\fkf$. In this case, $\calG_x$  (resp. $\calG_{\fkf}$) is  the parahoric group scheme associated to $x$ (resp. $\fkf$). By \cite{HaRa}, $\G_{\Xi}(\Ok_{\brF})= \widetilde\G_{\Xi}(\Ok_{\brF})\cap\ker\widetilde{\kappa}_G.$ It follows that $\calG_\Xi(\calO_F)=\widetilde{\calG}_\Xi(\calO_F)\cap\ker \widetilde{\kappa}_G$. 

We may also consider the corresponding objects over $\brF$ and we use the same notation in this case. When it is understood which  point of $\calB(G,F)$ or $\calB(G,\brF)$ we are referring to, we  simply write  $\widetilde{\G}$ and $\G$ for the corresponding group schemes.

An important case that is needed for applications is when $\calG_x=\tG_x$, i.e. the parahoric is equal to the Bruhat--Tits stabilizer. When this happens, we necessarily  have that  $\tG_x=\tG_{\fkf}$, where $\fkf$ is the facet which contains $x$, and $x\in \fkf$ is a point which is ``in general position.'' We say that $\calG$ is a \emph{connected parahoric} if there exists a point $x\in \calB(G,F)$ such that $\tG_x=\calG$.

Let $G'$ be another reductive group and assume there is an identification $G^{\ad}\cong G'^{\ad}$  between their respective adjoint groups. Then there are surjective maps of buildings $\calB(G,F)\rightarrow \calB(G^{\ad},F)$ and $\calB(G',F)\rightarrow \calB(G'^{\ad},F)$ which are equivariant for $G(F)$ and $G'(F)$ respectively. Let $\calG=\calG_x$ be a parahoric group scheme of $G$ corresponding to $x\in \calB(G,F)$, and let $x^{\ad}\in\calB(G^{\ad},F)$ denote the image of $x$.  Then for any $x'\in \calB(G',F)$  lifting $x^{\ad}$, the parahoric $\calG'=\calG'_{x'}$ of $G'$  is independent of the choice of $x'$ lifting $x^{\ad}$. Thus $\calG
$ determines a parahoric  $\calG'$ of $G'$; in this case we say that $\G$ and $\G'$ are \emph{associated}.
\subsubsection{} 
Now let $J\subset \mathbb{S}$ be a subset  and we write $W_J$ for the subgroup of $W$ generated by $J$. If $W_J$ is finite, $J$ corresponds to a parahoric group scheme $\calG$ over $\calO_{\brF}$; such parahorics are called \emph{standard} (with respect to $\fka$).  We let $W^J$ (resp. $^JW$) denote the set of minimal length representatives of the cosets $W/W_J$ (resp  $W_J\backslash W)$.

We recall the  Iwahori decomposition. For $w\in W$, the map $w\mapsto \dot{w}$ induces a bijection
$$W_J\backslash W / W_J \cong \G(\Ok_{\brF})\backslash G({\brF})/\G(\Ok_{\brF}).$$

We now assume $J$ is $\sigma$-stable; in this case the parahoric group scheme $\calG$ is defined over $\calO_F$. For the rest of \S\ref{sec: parahoric group schemes}, we fix a geometric conjugacy class of cocharacters $\{\mu\}$ of $G$ and assume the existence of $[b]_{\mu}\in B(G,\{\mu\})$.
We define $\Adm(\{\mu\})_J$ to be the image of $\Adm(\{\mu\})$ in $W_J\backslash W/W_J.$  We sometimes write $\Adm_G(\{\mu\})_J$ if we want to specify the group $G$ we are working with. The following is the key group theoretic result needed to prove the existence of canonical liftings in \S\ref{sec: canonical liftings}.

\begin{prop}\label{prop: F-crystal basis}
	Let $b\in\left( \bigcup_{w\in\Adm(\{\mu\})_J}\G(\Ok_{\brF})\dot{w}\G(\Ok_{\brF})\right)\cap [b]_{\mu}$.  Then 
	
	\begin{enumerate}\item $b\in \mathcal{G}(\Ok_{\brF})\dot{t}_{\mu'}\G(\Ok_{\brF})$ for some $\sigma$-straight element $t_{\mu'}$ .
	
	\item There exists $g\in \G(\Ok_{\brF})$ such that $g^{-1}b\sigma(g)=\dot{t}_{\mu'}$, for $t_{\mu'}$ as in (1).
\end{enumerate}

 \end{prop}

\begin{proof}By \cite[Theorem 6.1 (b)]{HR}, there exists $h\in \G(\Ok_{\brF})$ such that $h^{-1}b\sigma(h)\in \mathcal{I}(\Ok_{\brF})\dot{w}\mathcal{I}(\Ok_{\brF})$ for some $w\in {^JW}$. 
Thus $w\in {^JW}\cap W_J\Adm(\{\mu\})W_J$ and hence lies in ${^J}W\cap\Adm(\{\mu\})$ by \cite[Theorem 6.1]{He3}. Thus upon replacing $b$ by $h^{-1}b\sigma(h)$, we may assume $b\in\mathcal{I}(\Ok_{\brF})\dot{w}\mathcal{I}(\Ok_{\brF})$. By \cite[Theorem 4.1]{HZ}, there exists a $\sigma$-straight element $x\leq w$ such that $[b]_{\mu}\cap\calI(\Ok_{\brF})\dot{x}\calI(\Ok_{\brF})\neq\emptyset$ (the Theorem in {\em loc.~cit.} proves the non-emptiness of the affine Deligne--Lusztig variety $X_x(b)$, which is equivalent to this statement). By the proof of \cite[Theorem 3.5]{He1}, we have $\dot{x}\in [b]_{\mu}$ and  by the same argument as in Lemma \ref{lemma: ord class rep by straight elt} we have $x=t_{\mu'}$ for some $\mu'\in W_0\cdot{\mu}$. Since $x\leq w$ and $w\in \Adm(\{\mu\})$, we have $w=t_{\mu'}$. This proves (1).
	
	For  (2), the above argument shows that we may assume $b\in \calI(\Ok_{\brF})\dot{t}_{\mu'}\calI(\Ok_{\brF})$ for $t_{\mu'}$ a $\sigma$-straight element. By \cite[Proposition 4.5]{He1}, there exists $i\in\calI(\Ok_{\brF})$ such that $i^{-1}b\sigma(i)=\dot{t}_{\mu'}$; the result follows.
\end{proof}

\begin{remark}
	This result is a generalization to general parahorics of \cite[Proposition 2.5]{SZ} which is due to Wortmann. In the case when $\G$ is a hyperspecial parahoric,  this result is the group theoretic analogue of the fact that there is exactly one isomorphism class of ordinary $F$-crystal over $\calO_{\brF}$.
\end{remark}

\subsection{N\'eron models of tori}\label{sec: Neron models}

\subsubsection{} In this subsection, we introduce the notion of $R$-smooth tori and discuss some consequences for Bruhat--Tits group schemes. 

Let $T$ be a torus over a non-archimedean  local field $F$; recall we have defined $\calT_0$ to be the connected N\'eron model of $T$. We let $\calT$ (resp. $\calT_{\mathrm{ft}}$) denote the lft N\'eron model (resp. finite type N\'eron model) for $T$. Then we have $\calT(\calO_{\brF})=T(\brF)$ and  $\calT_{\mathrm{ft}}$ is characterized by the condition  $\calT_{\mathrm{ft}}(
	\calO_{\brF})=\{t\in T(\brF)|\widetilde\kappa_T(t)\in X_*(T)_{I,\mathrm{tors}}\}$ where $X_*(T)_{I,\mathrm{tors}}$ is the torsion subgroup of $X_*(T)_I$. Alternatively, by \cite[n$^\circ$1]{Ra} the connected components of the special fiber of $\calT$ are parameterized by $X_*(T)_I$ and $\calT_{\mathrm{ft}}$ is the unique smooth subgroup scheme of $\calT$ whose special fiber is given by the set of connected components corresponding to the subgroup $X_*(T)_{I,\mathrm{tors}}$ of $X_*(T)_{I}$.

\subsubsection{}\label{sssec: R-smooth}Let $\widetilde{F}/F$ be a finite Galois extension over which $T$ splits and 
$\calT_{\calO_{\widetilde{F}}}$ denote the lft N\'eron model of $T_{\widetilde{F}}.$ \footnote{We are abusing notation here since $\calT_{\calO_{\widetilde{F}}}$ is not necessarily the base change to $\calO_{\widetilde{F}}$ of the N\'eron model $\calT$ of $T$ over $\calO_{F}$.} By \cite[\S7.6, Proposition 6]{BLR}, $\mathrm{Res}_{\calO_{\widetilde{F}}/\calO_F}\calT_{\calO_{\widetilde{F}}}$ is  the lft  N\'eron model over $\calO_{F}$ for $\mathrm{Res}_{\widetilde{F}/F}T_{\widetilde{F}}$. There is a natural map $T\rightarrow \mathrm{Res}_{\widetilde{F}/F}T_{\widetilde{F}}$ and we define $\calT^c$ to be the Zariski closure of $T$ inside $\mathrm{Res}_{\calO_{\widetilde{F}}/\calO_{F}}\calT_{\calO_{\widetilde{F}}}$. As in \cite[\S4.4.8]{BT2}, $\calT^c$ does not depend on the choice of splitting field  $\widetilde{F}$. 

\begin{definition}\label{def: R-smooth}
	We say a torus  $T$ is \textit{R-smooth} if $\calT^c$ is smooth.
\end{definition} 

Since $\calT^c$ satisfies the N\'eron mapping property (see \cite[Proof of Theorem 4.2]{Ed} for the proof in the case of abelian varieties which also works for tori), we have $\calT\cong\calT^c$ if $T$ is $R$-smooth. 

We can similarly define a notion of $R$-smoothness for tori over $\brF$. It is easy to see using compatibility of N\'eron models with base change along $\calO_F\rightarrow \calO_{\brF}$ that a torus over $F$ is $R$-smooth if and only if $T_{\brF}$ is $R$-smooth.

\begin{lemma}\label{lem: R smooth torus property}
	Suppose we have a closed immersion $f:T_1\rightarrow T_2$ between tori where $T_1$ is $R$-smooth. Then \begin{enumerate}\item $f$ extends to a closed immersion $\calT_1\rightarrow \calT_2$ of lft N\'eron models.
		\item $f$ extends to a closed immersion $\calT_{1,\mathrm{ft}}\rightarrow \calT_{2,\mathrm{ft}}$ of finite type N\'eron models.

		\end{enumerate} 
\end{lemma}
\begin{proof}(1) Let $\widetilde{F}/F$ be a finite  Galois splitting field for both $T_1$ and $T_2$. Then since $T_{1,\widetilde{F}}$ and $T_{2,\widetilde{F}}$ are products of multiplicative group schemes, the map $T_{1,\widetilde{F}}\rightarrow T_{2,
	\widetilde{F}}$ extends to a closed immersion  of lft N\'eron models $\calT_{\calO_{\widetilde{F}}}\rightarrow \calT_{2,\calO_{\widetilde{F}}}$ over $\calO_{\widetilde{F}}$. We obtain a diagram \[\xymatrix{\calT_1\ar[r]^f\ar[d]_g&\calT_2\ar[d]^h\\ \mathrm{Res}_{\calO_{\widetilde{F}}/\calO_F}\calT_{1,\calO_{\widetilde{F}}}\ar[r]^i&\mathrm{Res}_{\calO_{\widetilde{F}}/\calO_F}\calT_{2,\calO_{\widetilde{F}}}}\]
where $i$ is a closed immersion since it is obtained via restriction of scalars of a closed immersion, and $g$ is a closed immersion since $T_1$ is $R$-smooth.  It follows that $h\circ f=i\circ g$ is a closed immersion, and hence $f$ is a closed immersion.

(2) By (1), we have a closed immersion $\calT_1\rightarrow \calT_2$ of lft N\'eron models. We let $\phi:X_*(T_1)_I\rightarrow X_*(T_2)_I$ denote the  morphism on the targets of the Kottwitz homomorphism. Using that $\ker(\phi)$ is torsion, one sees that $$\phi^{-1}(X_*(T_2)_{I,\mathrm{tors}})=X_*(T_1)_{I,\mathrm{tors}}.$$
As the finite type N\'eron models $\calT_{1,\mathrm{ft}}$ and $\calT_{2,\mathrm{ft}}$ correspond to the subschemes of $\calT_1$ and $\calT_2$ whose special fibers are  given by the connected components parameterized by $X_*(T_1)_{I,\mathrm{tors}}$ and $X_*(T_2)_{I,\mathrm{tors}}$ respectively, it follows that $\calT_1\rightarrow \calT_2$ induces a closed immersion $\calT_{1,\mathrm{ft}}\rightarrow \mathrm{\calT}_{2,\mathrm{ft}}$ as desired.
\end{proof}

\subsubsection{}The proof of \cite[Theorem 4.2]{Ed} shows that if $T$ splits over a tamely ramified extension of $F$, then $T$ is $R$-smooth. In addition, the main examples of $R$-smooth tori that we will consider are given by the following proposition.

\begin{prop}\label{prop: examples of R smooth torus}\begin{enumerate}\item Let $T=\prod_{i=1}^s\mathrm{Res}_{K_i/F}S_i$, where $K_i$ are finite separable extensions of $F$ and $S_i$ are $K_i$-tori which split over a tamely ramified extension of $K_i$. Then $T$ is $R$-smooth.

		\item Let $T$ be a torus which is an extension of an $R$-smooth torus by an $R$-smooth torus. Then $T$ is $R$-smooth.
		
	\end{enumerate} 
\end{prop}

\begin{proof}(1) We will make use of the following result which follows from  \cite[\S7.6 Proposition 6]{BLR}: If $S$ is a torus over a finite separable extension $K$ of $F$ with lft N\'eron model $\calS$ over $\calO_K$, then $\mathrm{Res}_{\calO_K/\calO_F} \calS$ is the lft N\'eron model for $\mathrm{Res}_{K/F}S$.
	
	We may reduce to the  case $s=1$, in which case we write $T=\mathrm{Res}_{K/F}S$ for $S$ a tamely ramified torus over $K$. Let $\widetilde{F}/F$ be a finite Galois splitting field of $T$ which necessarily contains $K$. For any $F$-morphism $\tau:K\rightarrow \widetilde{F}$, the base change of $S$ along $\tau$ is split. Since $S$ is $R$-smooth, it follows that we have a closed immersion  of $\calO_K$-group schemes $$\calS\rightarrow\mathrm{Res}_{\calO_{\widetilde{F}}/\calO_K}\calS_{\calO_{\widetilde{F}}},$$ where $\calS$ (resp. $\calS_{\calO_{\widetilde{F}}}$) is the lft N\'eron model for $S$ (resp. $S_{\widetilde{F}}$).
	
	Applying $\mathrm{Res}_{\calO_K/\calO_F}$ we obtain a closed immersion
	$$\mathrm{Res}_{\calO_{K}/\calO_F}\calS\rightarrow\mathrm{Res}_{\calO_{\widetilde{F}}/\calO_F}\calS_{\calO_{\widetilde{F}}}.$$ 
	Taking the product over all $\tau:K\rightarrow \widetilde{F}$ we obtain a closed immersion 
	$$\mathrm{Res}_{\calO_{K}/\calO_F}\calS\rightarrow\prod_{\tau:K\rightarrow \widetilde{F}}\mathrm{Res}_{\calO_{\widetilde{F}}/\calO_F}\calS_{\calO_{\widetilde{F}}}\cong\mathrm{Res}_{\calO_{\widetilde{F}}/\calO_F}\calT_{\calO_{\widetilde{F}}}.$$
	
	Since $\mathrm{Res}_{\calO_K/\calO_F}\calS$ is the lft N\'eron model $\calT$ for $T$, it follows that $\calT$ is the closure of its generic fiber inside $\mathrm{Res}_{\calO_{\widetilde{F}}/\calO_F}\calT_{\calO_{\widetilde{F}}}$ and hence $T$ is $R$-smooth.
	
	(2) We may assume $F=\brF$.  Assume we have an exact sequence 
		\[\xymatrix{1\ar[r]& T_1\ar[r]^f &T\ar[r]^g &T_2\ar[r]& 1}\] 
		where $T_1$ and $T_2$ are $R$-smooth. Since $T_1$ is $R$-smooth, $f$ extends to a closed immersion of lft N\'eron models $\calT_1\rightarrow \calT,$ by Lemma \ref{lem: R smooth torus property}.  The quotient $\calT/\calT_1$ is a smooth group scheme with generic fiber $T_2$, and by Steinberg's theorem  it has the same $\calO_{\breve F}$-points as the lft N\'eron model $\calT_2$ for $T_2$. Thus by \cite[Proposition 1.7.6]{BT2}, $\calT/\calT_1\cong \calT_2$ and we have an exact sequence of group schemes 
			\[\xymatrix{1\ar[r]  &\calT_1\ar[r] &\calT\ar[r] &\calT_2\ar[r] &1}.\]

	Let ${\widetilde{F}}/F$ be a finite Galois extension over which $T_1,T_2$ and $T$ split. We obtain a commutative diagram with exact rows:
	\[\xymatrix{1\ar[r]&\calT_1\ar[r]\ar[d]&\calT\ar[r]\ar[d]^j&\calT_2\ar[d]\ar[r] & 1\\1\ar[r]&\mathrm{Res}_{\calO_{\widetilde{F}}/\calO_F}\calT_{1,\calO_{\widetilde{F}}}\ar[r]& \mathrm{Res}_{\calO_{\widetilde{F}}/\calO_F}\calT_{\calO_{\widetilde{F}}}\ar[r]&\mathrm{Res}_{\calO_{\widetilde{F}}/\calO_F}\calT_{2,\calO_{\widetilde{F}}}\ar[r]& 1}\]
	
Let $\calT^c$ be the Zariski closure of $T$ in $\mathrm{Res}_{\calO_{\widetilde{F}}/\calO_F}\calT_{\calO_{\widetilde{F}}}.$ 
By $R$-smoothness of $T_2$ and  $T_1$, the two outer vertical maps in the diagram above are closed immersions. 
Hence, $\calT_1$ is closed in $\calT^c,$ and the composite 	
$$ \calT_2 \simeq \calT/\calT_1 \rightarrow \calT^c/\calT_1 \rightarrow \mathrm{Res}_{\calO_{\widetilde{F}}/\calO_F}\calT_{2,\calO_{\widetilde{F}}}$$
is a closed immersion. 
Thus, $\calT/\calT_1$ is closed in $\calT^c/\calT_1.$ As these are two $\O_F$-flat schemes with the same generic fiber, it follows 
that $ \calT/\calT_1 \simeq \calT^c/\calT_1,$ and hence $\calT \simeq \calT^c.$ 
Hence $j$ is a closed immersion and $T$ is $R$-smooth.
\end{proof}

\subsubsection{}
	
The previous results have the following consequences for Bruhat--Tits group schemes. Let $G$ be a reductive group over $F$ and $\widetilde{\calG}$ a Bruhat--Tits stabilizer scheme corresponding to $x\in \calB(G,F)$.
Let 
$\beta:G\hookrightarrow G'$ be a closed immersion of reductive groups over $F,$ which induces an isomorphism on derived groups. As in \cite[\S1.1.3]{KP},  $x$  determines a point $x'\in \calB(G',F)$ and we write $\widetilde{\calG}'$ for the corresponding stabilizer scheme of $G'$; then $\beta$ extends to a group scheme morphism $\beta:\widetilde{\calG}\rightarrow \widetilde{\calG}'$.

\begin{prop}\label{prop: closed immersion of BT schemes}
	Assume that there exists a maximal $\brF$-split torus in $G$ whose centralizer is an $R$-smooth torus.  Then $\beta:\widetilde{\calG}\rightarrow \widetilde{\calG}'$ is a closed immersion.

\end{prop}

\begin{proof}	As all maximal $\brF$-split tori are $\brF$-conjugate, the centralizer of any maximal $\brF$-split torus is $R$-smooth if there exists one such centralizer which is $R$-smooth. Therefore we may assume that all such centralizers are $R$-smooth. Moreover, since the  construction of Bruhat--Tits stabilizer schemes is compatible with unramified base extensions, it is enough to prove the result in the case $F=\brF$. 
	
Let $S$ be a maximal $\brF$-split torus in $G$ such that  $x$ lies in $\calA(G,S,\brF)$. Let $T$ be the centralizer of $S$ which by assumption is an $R$-smooth torus.  
		Let $S'$ be a maximal $\brF$-split torus of $G'$ such that $S'\cap G=S$ and $T'$  the centralizer of $S'$.  
		
		 By the construction of Bruhat--Tits stabilizer schemes in \cite[\S4.6]{BT2}, the Zariski closure of $T$ (resp. $T'$) inside $\widetilde{\calG}$ (resp. $\widetilde{\calG}'$) can be identified with the finite type N\'eron model $\calT_{\mathrm{ft}}$ (resp. $\calT'_{\mathrm{ft}}$). 
By Lemma \ref{lem: R smooth torus property}, the natural map $T\rightarrow T'$ extends to a closed immersion $\calT_{\mathrm{ft}}\rightarrow \calT'_{\mathrm{ft}}$ of finite type N\'eron models.

	 For any relative root $\alpha$, the map $G\rightarrow G'$ induces an isomorphism between the root subgroups $U_\alpha$ and $U'_{\alpha}$. If we let $\calU_{\alpha}$ and $\calU_{\alpha}'$ denote the corresponding schematic closures, then by the construction of $\widetilde{\calG}$ and $\widetilde{\calG}'$  in \cite[\S4.6]{BT2}, the map $G\rightarrow G'$ also induces an isomorphism $\calU_{\alpha}\rightarrow\calU_{\alpha}'$. Thus as in \cite[Theorem 2.2.3]{BT2} the schematic closure $\widehat{\calG}$ of $G$ in $\widetilde{\calG}'$ contains the smooth big open cell
$$\prod_{\alpha}\calU_{-\alpha}\times{\calT}_{\mathrm{ft}}\times\prod_{\alpha}\calU_\alpha;$$ 
hence by \cite[Corollary 2.2.5]{BT2}, $\widehat{\calG}$ is smooth. Since $\widehat{\calG}(\calO_{\brF})=G(\brF)\cap\widetilde{\calG}'(\calO_{\brF})$, it follows that $\widehat{\calG}\cong \widetilde{\calG}$,  and hence we obtain a closed immersion $\widetilde{\calG}\hookrightarrow \widetilde{\calG}'$ as desired.
	\end{proof}

\subsubsection{}\label{sec: embedding Weil restricted form} Now let $K/F$ be a finite separable extension. There is a natural embedding of buildings $\calB(G,F)\rightarrow \calB(G,K)$ and the image of $x$ in $\calB(G,K)$ determines a  Bruhat--Tits stabilizer scheme $\widetilde{\calG}_0$ over $\calO_K$.  Then by \cite[p. 172]{Prasad}, there is an identification of buildings $\calB(G,K)\cong \calB(\mathrm{Res}_{K/F}G_K,F)$ and the stabilizer scheme for $\mathrm{Res}_{K/F}G_K$ corresponding to $x$ can be identified with $\mathrm{Res}_{\calO_K/\calO_F}\tG_0$ (see eg. \cite[\S4.2]{HaRi}). By \cite[\S1.7.6]{BT2}, we obtain a natural morphism $i:\widetilde{\calG}\rightarrow \mathrm{Res}_{\calO_K/\calO_F}\widetilde{\calG}_0$ of $\calO_F$-group schemes.
A similar argument to Proposition \ref{prop: closed immersion of BT schemes} gives the following proposition which generalizes \cite[Prop. 1.3.9]{KP} (cf. \cite[Cor. 5.26]{FHLR}).

\begin{prop}\label{prop: closed immersion BT schemes Weil Res}
Assume $p>2$ and that the centralizer of a maximal $\breve F$-split torus in $G$ is $R$-smooth. Then  $i:\widetilde{\calG}\rightarrow \mathrm{Res}_{\calO_K/\calO_F}\widetilde{\calG}_0$ is a closed immersion.
\end{prop}
\begin{proof}We may assume $F=\breve F$. It suffices to prove the result for $K$ a field over which $G$ splits. Indeed, if $K'/K$ is an extension over which $G$ splits and $\tG_0'$ is the Bruhat--Tits stabilizer scheme over $\calO_{K'}$ corresponding to $x$, then $\tG\rightarrow \mathrm{Res}_{\calO_K/\calO_F}\tG_0$ is a closed immersion if the composition  $\tG\rightarrow \mathrm{Res}_{\calO_K/\calO_F}\tG_0\rightarrow \mathrm{Res}_{\calO_{K'}/\calO_F}\tG_0'$ is a closed immersion.

The same argument as in Proposition \ref{prop: closed immersion of BT schemes} shows that we can reduce to proving the following two statements:

	\begin{enumerate}\item $i|_{\calT_{\mathrm{ft}}}$ is a closed immersion, where $T$ is the centralizer of a maximal $\brF$-split torus $S$ whose apartment contains $x$.
			
			\item $i|_{\calU_{\alpha}}$ is a closed immersion, where $\alpha$ is a relative root for $G$ and $\calU_{\alpha}$ is the schematic closure of the root subgroup $U_{\alpha}$ inside $\tG$.
		\end{enumerate}
The first  follows from Lemma \ref{lem: R smooth torus property} (2) applied to the map $T\rightarrow T'$, where $T'$ is the centralizer in $ \mathrm{Res}_{K/F} G_K$ of a maximal $\breve F$-split torus containing $S$.

For the second statement, let $\alpha$ be a relative root and let $G_\alpha$ denote the simply-connected cover of the subgroup of $G$ generated by the root subgroups corresponding to relative roots which are proportional to $\alpha$. Then $G_{\alpha}$ is isomorphic to either

\begin{enumerate} \item $\mathrm{Res}_{L/F}\SL_2$ for $L/F$ a  finite separable extension.
	
	\item $\mathrm{Res}_{L/F}\mathrm{SU}_3$, where $\mathrm{SU}_3$ is the special unitary group over $L$ associated to a hermitian  space over a separable quadratic extension $L'/L$.	\end{enumerate}

Let $G_\alpha'$ denote the subgroup of $G$ generated by the image of $G_\alpha$ and $T$; then $G_\alpha'$ contains the maximal $\breve F$-split torus $S$ of $G$. By the main result of \cite{Landvogt}, the inclusion $G'_\alpha\rightarrow G$ induces a $G_{\alpha}'(\breve F)$-equivariant embedding of buildings, which restricts to  an  identification of apartments $\calA(G_\alpha',S,\breve F)\cong \calA(G,S,\breve F)$. The point $x\in\calA(G,S,\breve F)$ corresponding to $\tG$ determines a Bruhat--Tits stabilizer scheme of $G_\alpha'$, and  since $G_\alpha $ and $G_\alpha'$ have the same adjoint group, we obtain a stabilizer scheme $\tG_\alpha$ of $G_\alpha$ via the choice of a lift $x_\alpha\in \calB(G_\alpha,\brF)$ of the image of $x$ in $\calB(G_\alpha'^{\ad},\brF)$.

We have a commutative diagram \[\xymatrix{\tG_{\alpha}\ar[r]\ar[d]&\mathrm{Res}_{\calO_K/\calO_F}\tG_{\alpha,0}\ar[d]\\
\tG\ar[r] &\mathrm{Res}_{\calO_K/\calO_F}\tG_0,}\] where  $\tG_{\alpha,0}$ denotes the parahoric for $G_{\alpha,K}$ corresponding to $x_\alpha\in \calB(G_\alpha,K)$. The natural morphism $\tG_\alpha\rightarrow \tG$ induces an isomorphism on the integral root subgroups $\calU_\alpha$ and similarly for  the morphism $\tG_{\alpha,0}\rightarrow \tG_0$. It therefore suffices to prove the result for $G=G_\alpha$. Note that since we have assumed $p>2$, $G_\alpha$ is the Weil-restriction of a tamely ramified group. Thus it suffices to prove the proposition in this case, which we now do.

We first consider the case that  $G$ itself splits over a tamely ramified extension $K^t/F$. 
We may assume $K$ contains $K^t$. Let $\tG_0^t$  denote the Bruhat--Tits stabilizer scheme of $G_{K^t}$ corresponding to $\tG$. Then $i$ factors as $$\tG\rightarrow \mathrm{Res}_{\calO_{K^t}/\calO_F}\tG_0^t\rightarrow \mathrm{Res}_{\calO_K/\calO_F}\tG_0.$$ The first morphism is a closed immersion by \cite[Proposition 1.3.9]{KP}. The second morphism is obtained from 
$\widetilde{\calG}_0^t\rightarrow \mathrm{Res}_{\calO_K/\calO_{K^t}}\widetilde{\calG}_0$ 
by applying Weil-restriction. Since $G_{K^t}$ is split, we can reduce to the case $G_{K^t}=\SL_2$, as above, and this 
follows from Lemma \ref{lem: SL2 closed immersion} below.

Now assume $G=\mathrm{Res}_{L/F}H$ where $H$ is a group which splits over a tame extension of $L$ and that $K$ contains $L$. Then $G\rightarrow \mathrm{Res}_{K/F}G_K$ 
arises from Weil-restriction of a morphism $H\rightarrow \mathrm{Res}_{K/L} G_K$, which is given by a product of the diagonal morphisms $H\rightarrow \mathrm{Res}_{K/L} H_K$. Hence the result in this case follows from the tame case proved in the previous paragraph. The proposition follows.
\end{proof}

\begin{lemma}\label{lem: SL2 closed immersion}
	Let $G=\SL_2$.  Then the morphism $i:\tG\rightarrow \mathrm{Res}_{\calO_K/\calO_F}\tG_0$ is a closed immersion.
\end{lemma}

\begin{proof}We may assume $\tG$ corresponds to a point in the apartment for the diagonal torus $T$; let $U$ be a root subgroup for $T$. Since $T$  is split,  hence $R$-smooth, it suffices as above to show $U\rightarrow\mathrm{Res}_{K/F} U_K$  extends to  a closed immersion 
$\calU\rightarrow \mathrm{Res}_{\calO_K/\calO_F}\calU_0$, where $\calU$ (resp. $\calU_0$) is the Zariski closure of $U$  in $\tG$ (resp. $U_K$ in $\tG_0$). 
	The morphism $U\rightarrow \mathrm{Res}_{K/F}U_K$ can be identified with the diagonal morphism $\bbG_a\rightarrow \mathrm{Res}_{K/F}\bbG_a$. 
	
	Let $\varpi_F$ (resp. $\varpi_K$) be a uniformizer for $F$ (resp. $K$), and let $e$ denote the ramification index of $K/F$. By the construction of the stabilizer schemes in \cite{BT2}, $\calU_0$ is the $\calO_K$-group scheme cordeoresponding to the  $\calO_K$-submodule $\varpi_K^{ne-k}\calO_K$ of $K=\bbG_a(K)$, for some $n\in \bbZ$ and $k\in \{0,\dotsc,e-1\},$ which depend on the choice of $x\in \calB(G,F)$. Then $\calU$ corresponds to the $\calO_F$-submodule $\varpi_F^n\calO_F$  of $F=\bbG_a(F)$. We can extend $\varpi_F^n$ to an $\calO_F$-basis for $\varpi_K^{ne-k}\calO_K$ considered as an $\calO_F$-module, and this induces an identification $\mathrm{Res}_{\calO_K/\calO_F}\calU_0\cong \bbA^m$ where $m=[K:F]$.  The map $\calU\rightarrow \mathrm{Res}_{\calO_K/\calO_F}\calU_0$ is then identified with the closed immersion $\bbA^1\rightarrow \bbA^m$ taking $a$ to $(a,0,\dotsc,0)$.
	\end{proof}
	
\subsubsection{}Now let $\beta:G\rightarrow G'$ be a central extension between reductive groups with kernel $Z$ and $\calG$ the parahoric group scheme associated to some $x\in \calB(G,F)$. We let $\calG'$ denote the parahoric of $G'$ corresponding to $\calG.$ As above, 
$\beta$ extends to a group scheme homomorphism $\calG\rightarrow \calG'$.

\begin{prop}\label{prop: exact sequence parahorics}Assume $Z$ is an $R$-smooth torus. Then the Zariski closure $\widetilde{\calZ}$ of $Z$ inside $\calG$ is smooth and there is an (fppf) exact sequence \begin{equation}\label{eqn: exact sequence of parahorics}\xymatrix{0\ar[r]&\widetilde{\calZ}\ar[r]&\calG\ar[r]^\beta&\calG'\ar[r]&0}\end{equation}of group schemes over $\calO_F$. 
	
\end{prop}

\begin{proof}As before, it suffices to prove the proposition when $F=\brF$. Let $S$ be a maximal $\brF$-split torus of $G$ such that $x$ lies in $\calA(G,S,\brF)$. Let $T$ be the centralizer of $S$ and we let $T'$ be the corresponding maximal  torus of $G'$. 
	
	Assume  there exists an fppf exact sequence \begin{equation}\label{eqn: exact sequence of tori}\xymatrix{1\ar[r]&\widetilde\calZ\ar[r]&\calT_{0}\ar[r]&\calT'_{0}\ar[r]&1}\end{equation}where $\calT_{0}$ and $\calT'_0$ are the connected N\'eron models of $T$ and $T'$ respectively. Then we may argue as in \cite[Proposition 1.1.4]{KP} to obtain the desired exact sequence (\ref{eqn: exact sequence of parahorics}). 

	It remains to exhibit the exact sequence (\ref{eqn: exact sequence of tori}); we follow the argument of \cite[Lemma 6.7]{PR}.	By Lemma \ref{lem: R smooth torus property} we obtain a closed immersion between lft N\'eron models $\calZ\rightarrow \calT$. We let $\widetilde{\calZ}'$ denote the subgroup scheme of $\calZ$ with generic fiber $Z,$ and special fiber corresponding to the connected  components of the special fiber of $\calZ$ parameterized by $\ker(X_*(Z)_I\rightarrow X_*(T)_I)$. Then $\widetilde{\calZ}'$ is smooth and we have a closed immersion $\widetilde{\calZ}'\rightarrow \calT_0$. It follows that $\widetilde{\calZ}'$ coincides with $\widetilde\calZ$ and we obtain a closed immersion $\widetilde\calZ\rightarrow \calT_0$. As in \cite[Lemma 6.7]{PR} we have an exact sequence:
	\[\xymatrix{1\ar[r]&\widetilde{\calZ}(\calO_{\brF})\ar[r]&\calT_0(\calO_{\brF})\ar[r]&\calT'_0(\calO_{\brF})\ar[r]&1}\]
	The quotient $\calT_0/\widetilde\calZ$ is a smooth affine commutative group scheme with the same generic fiber as $\calT_0'$ and with the same $\calO_{\brF}$-points; hence by \cite[Proposition 1.7.6]{BT2} we have $\calT_0'\cong\calT_0/\widetilde{\calZ}$. The result follows.
\end{proof}

\section{Deformation theory of $p$-divisible groups}\label{sec: deformation theory}
In this section we prove  the deformation theoretic results needed for the study of integral models of Shimura varieties in \S\ref{sec: integral models for Shimura + canonical lifts}. In \S\ref{sec: embedding into Grassmannian}, we discuss properties of local models and their embeddings in Grassmannians. \S\ref{sec: adapted liftings}-\S\ref{sec: etale tensors}  contains the results needed to describe the formal neighborhood of Shimura varieties, and in \S\ref{sec: M-adapted}, we apply this to the case  of $\mu$-ordinary $p$-divisible groups to construct an analogue of the Serre--Tate canonical lift.
\subsection{Local models and good embeddings}\label{sec: embedding into Grassmannian}

\subsubsection{}\label{sec: local model triple} Let $(G,\{\mu\},\G)$ be a local model triple over $F$ as in \cite[\S2.1]{HPR}. Thus 
\begin{itemize}
	\item $G$ is a reductive group scheme over $F$.
	\item $\{\mu\}$ is a geometric conjugacy class of minuscule cocharacters of $G$.
	\item $\G=\calG_x$ for some $x\in \calB(G,F)$ is a parahoric group scheme.
\end{itemize}

A morphism of local model triples $(G,\{\mu\},\G) \rightarrow (G',\{\mu'\},\G')$ is a morphism $\G \rightarrow \G'$ 
taking $\{\mu\}$ to $\{\mu'\}.$  
We denote by $E$ the reflex field of the pair $(G, \{\mu\})$. It is a subfield of $\bar F$ containing $F$.  

We consider local model triples which satisfy the following property.
\begin{definition}\label{def: acceptable group}
	A  reductive group $G$ over $F$ is said to be \emph{acceptable} if $G^{\ad}\cong\prod_{i=1}^r\mathrm{Res}_{{K_i}/F}H^{\ad}_i$  where $K_i/F$ is a finite extension and $H^{\ad}_i$ is an adjoint group over $K_i$ which splits over a tame extension of $K_i$.\footnote{In \cite{PRshtukas} and \cite{KPZ}, these groups are called essentially tamely tamified.}
	
		A local model triple $(G,\{\mu\},\calG\})$ is said to be acceptable if $G$ is acceptable.
\end{definition}

\begin{remark}If $p> 3$, there are no automorphisms of a connected Dynkin diagram of order divisible by $p$, hence  any such reductive group is acceptable. Moreover, for $p=3$, any reductive group which has no factors of type $D_4$ is acceptable, as this is the only connected Dynkin diagram with an automorphism of order 3.

\end{remark}

	Let $(G,\{\mu\},\calG)$ be a local model triple. 
An embedding $\rho: G\rightarrow \GL(V)$ is called a \emph{local Hodge embedding} if the following conditions are satisfied:
		\begin{itemize}
			\item $\rho(G)$ contains the scalars.
			
			\item $\rho$ is a minuscule representation.
			
			\item $\rho\circ\mu $ is conjugate to a standard minuscule cocharacter of $\GL(V)$.
		\end{itemize}
We say that $(G,\{\mu\},\calG)$ is of \emph{local Hodge type} if it admits a local Hodge embedding.

As explained in \cite[Remark 3.1.5]{KPZ}, if  $(G,\{\mu\},\calG)$ arises from completion at $p>2$ of a global Shimura datum $(\bfG,X)$ of Hodge type, then $(G,\{\mu\},\calG)$ will be acceptable of local Hodge type.

\subsubsection{}\label{sec: construction local model}
In the rest of this section we assume $p>2$. We write $\Mloc_{\calG,\{\mu\}}$ for the local model associated to the local model triple $(G,\{\mu\},\calG)$. By definition, $\Mloc_{\calG,\{\mu\}}$ is the unique, up to unique isomorphism, proper flat $\O_E$-scheme with $\calG$-action, with generic fiber $G/P_\mu$ and reduced special fiber, which represents the $v$-sheaf ${\rm M}^v_{\calG,\mu}$ over ${\rm Spd}(\O_E)$ defined in \cite{SW2} (this is denoted by ${\rm Gr}_{\calG,{\rm Spd}(\O_E),\mu}$ in \cite[Lect. 21]{SW2}). 

The existence of $\Mloc_{\calG,\{\mu\}}$ was conjectured by Scholze-Weinstein \cite[Conj. 21.4.1]{SW2} and is shown in \cite{AGLR} under mild assumptions, and in general in \cite{GL}. If $(G,\{\mu\},\calG)$ is acceptable and of local Hodge type,  which is our main case of interest, a simpler proof of the existence of  is given in \cite[Theorem 3.2.15]{KPZ}. Under these assumptions, the construction of $\Mloc_{\calG,\{\mu\}}$ in \cite[\S3.2.12]{KPZ} shows that it is identified with the scheme $\rmM_{\calG',\{\mu'\}}$ constructed in \cite{Levin} for an auxiliary local model triple $(G',\{\mu'\},\calG')$ with $p\nmid|\pi_1(G'^{\der})|$ and $G'^{\ad}\cong G^{\ad}$. In particular, $\Mloc_{\calG,\{\mu\}}\otimes_{\calO_E}k$ admits a stratification indexed by the $\mu$-admissible set $\Adm(\{\mu\})_J$, where $J$ is the set of affine reflections corresponding to $\calG$, and its irreducible components are normal and Cohen--Macaulay; see \cite[Theorem 3.2.9]{KPZ}.

\subsubsection{}The following  notion will be needed for applications in  \S\ref{sec: integral models for Shimura + canonical lifts}.

\begin{definition}\label{def: good embedding}
	
	Let $(G,\{\mu\},\calG)$ be an acceptable local model triple of local Hodge  type. \begin{enumerate}
		\item 
		An integral Hodge embedding for $(G,\{\mu\},\calG)$ is a closed immersion $\rho:\tG\rightarrow \GL(\Lambda)$, where $\tilde{\calG}$ is a stabilizer scheme for $G$ with associated parahoric $\calG$ and  $\Lambda\subset V$ is an $\calO_F$-module in a finite dimensional vector space $V$, such that the induced map $G\rightarrow \GL(V)$ on generic fibers is a local Hodge embedding.
		
		\item	An integral Hodge embedding $\rho:\tG\rightarrow \GL(\Lambda)$ is said to be \emph{good} if there is a closed immersion of local models 
		$$\bbM^{\mathrm{loc}}_{\G,\{\mu\}}\hookrightarrow \mathrm{Gr}(\Lambda)\otimes_{\calO_F}\calO_E$$  
		extending the natural map on the generic fiber. Here $\mathrm{Gr}(\Lambda)$ is the smooth Grassmannian of 
		subspaces $\mathcal{F}\subset \Lambda$ of rank $d$, where $d\in \bbZ_{\geq 0}$ is such that $\{\rho\circ\mu\}$ is the conjugacy class of  $a\mapsto\mathrm{diag}(1^{(n-d)},(a^{-1})^{(d)})$.
		
	\end{enumerate}
	
	A local Hodge embedding $G\rightarrow \GL(V)$ is \emph{good}, if it extends to a good integral local Hodge embedding $\tG\rightarrow \GL(\Lambda)$ for some $\tG$ and $\Lambda\subset V$.
\end{definition}

\begin{remark}
	\label{rem: SILHE}
	
	If we assume in addition that 
	$\widetilde{\calG}=\calG$, then Definition \ref{def: good embedding} (2) recovers the definition of a strongly integral local Hodge embedding for  $(\calG,\bbM^{\mathrm{loc}}_{\calG,\{\mu\}})$ as in \cite[\S3.1.4]{Pappascanonical}.
\end{remark}

\subsubsection{} The following result gives a strengthening of \cite[Theorem 3.3.25]{KPZ}.

\begin{prop}\label{prop: good embedding GL}Let $p>2$ and $(G,\{\mu\},\calG)$  an acceptable local model triple of local Hodge type. Assume $p\nmid|\pi_1({G^{\der}})|$ and that the centralizer of a maximal $\breve F$-split torus in $G$ is $R$-smooth.
	Then $(G,\{\mu\},\calG)$ admits a good local Hodge embedding $\rho:G\rightarrow \mathrm{GL}(W)$. 	\end{prop}
\begin{proof}   Let $K/F$ be a finite extension over which $G$ splits and let $\rho':G\rightarrow \GL(V)$ be a local Hodge embedding. We let $W$ denote the underlying $F$ vector space of $V_K:=V\otimes_F K$. 
Then we have a faithful embedding  $\rho:G\rightarrow \GL(W)$ given by the composition
	\begin{equation}\label{eqn: composition of morphism}G\rightarrow \mathrm{Res}_{K/F}G_{K}\rightarrow\Res_{K/F}\GL(V_K)\rightarrow \GL(W)\end{equation}
	where the morphism $\mathrm{Res}_{K/F}G_{K}\rightarrow\Res_{K/F}\GL(V_K)$ arises from Weil restriction of the morphism $\rho'_K:G_K\rightarrow \GL(V_K)$ given by the  base change of $\rho'$ to $K$ and the last morphism is given by restriction of structure.
	Then $\rho$ is (geometrically) isomorphic to a direct sum of the representations $\rho'$ and hence is also local Hodge embedding.
	
		Let $\{\mu_{K/F}\}$ denote the conjugacy class of cocharacters of $\mathrm{Res}_{K/F}G_K$ induced by $\{\mu\}$. Let $\G_0$  (resp. $\tG_0$) denote the parahoric (resp. stabilizer) group scheme over $\calO_K$ corresponding to the image of $x$ in $\calB(G,K)$. We set $$\calG_{K/F}=\Res_{\calO_K/\calO_F}\calG_0, \ \ \  \tG_{K/F}=\Res_{\calO_K/\calO_F}\tG_0.$$ Then we obtain an acceptable local model triple $(\mathrm{Res}_{K/F}G_K, \{\mu_{K/F}\},\calG_{K/F})$ and we  let $E'$ be its local reflex field. The morphism $$\mathrm{Res}_{K/F}G_{K}\rightarrow\Res_{K/F}\GL(V_K)\rightarrow \GL(W)$$ obtained above is  a local Hodge embedding for  $(\mathrm{Res}_{K/F}G_K, \{\mu_{K/F}\},\calG_{K/F})$ and satisfies the assumptions in \cite[Theorem 3.3.25]{KPZ}. Thus by \emph{loc. cit.}, upon possibly replacing $x$ by a different point and $\rho$ by a direct sum, we obtain a good integral local Hodge embedding $\tG_{K/F}\rightarrow \GL(\Lambda)$ for $\Lambda\subset W$ an $\calO_F$-module. Note that our assumption that $p\nmid|\pi_1(G^{\der})|$ implies that the scheme $\rmM_{\calG_{K/F},\mu_{K/F}}$ in \cite[Theoremn 3.3.25]{KPZ} is isomorphic to $\Mloc_{\calG_{K/F},\{\mu_{K/F}\}}$, cf. \cite[Proof of Theorem 3.2.15]{KPZ}. We obtain closed immersions: 
		$$\tG_{K/F}\rightarrow \GL(\Lambda), \ \ \ \  \Mloc_{\calG_{K/F},\{\mu_{K/F}\}}\rightarrow \Gr(\Lambda)\otimes_{\calO_F}\calO_{E'}.$$
	
	By Proposition \ref{prop: closed immersion BT schemes Weil Res} and Lemma \ref{lem:FHLR} below, we have closed immersions $$\tG\rightarrow \tG_{K/F}, \ \ \ \ \bbM^{\mathrm{loc}}_{\calG,\{\mu\}}\rightarrow \bbM^{\mathrm{loc}}_{\calG_{K/F},\{\mu_{K/F}\}}\otimes_{\calO_{E'}}\calO_E,$$
	and hence composing with the above, we obtain a good integral local Hodge embedding $\tG\rightarrow \GL(\Lambda)$.
	\end{proof}

\begin{lemma}\label{lem:FHLR}With the notation and assumptions of the previous proposition, there is a closed immersion \begin{equation}\label{eqn: embedding local models Weil restriction}\bbM^{\mathrm{loc}}_{\calG,\{\mu\}}\rightarrow \bbM^{\mathrm{loc}}_{\calG_{K/F},\{\mu_{K/F}\}}\otimes_{\calO_{E'}}\calO_E.\end{equation}
\end{lemma}
\begin{proof}This follows from \cite[Lemma 5.27]{FHLR} and \cite[Theorem 7.21]{AGLR}. More precisely, \cite[Theorem 7.21]{AGLR} shows that the models constructed in \cite[Lemma 5.27]{FHLR} agree with our $\bbM^{\mathrm{loc}}_{\calG,\{\mu\}}$. \cite[Lemma 5.27]{FHLR}  then shows the existence of the closed immersion noting that Hypotheses 2.1 and 5.24 of \emph{loc. cit.} are satisfied by our assumptions of acceptability and that $p>2$.
\end{proof}

\subsubsection{}We prove  a slight variant of Proposition \ref{prop: good embedding GL} in the presence of an alternating form. Let $\rho:G\rightarrow \mathrm{GSp}(V)$ be a faithful symplectic representation where $V$ is a $2n$-dimensional  vector space over $F$ equipped with a perfect alternating bilinear form $\Psi$. We assume that $\rho\circ \mu$ is conjugate to the standard minuscule coweight $a\mapsto \mathrm{diag} (1^{(n)},(a^{-1})^{(n)})$ and that $\rho(G)$ contains the scalars. We call such an embedding a local symplectic Hodge embedding.  We say $\rho$ is good if the corresponding representation $G\rightarrow \GL(V)$ is good.

\begin{prop}\label{prop: regular triples}Let $p>2$ and $(G,\{\mu\},\calG)$  an acceptable local model triple of local Hodge type. Assume $p\nmid|\pi_1({G^{\der}})|$, the centralizer of a maximal $\breve F$-split torus in $G$ is $R$-smooth and that $G$ admits a local symplectic  Hodge embedding $\rho:G\rightarrow \mathrm{GSp}(V)$.
	Then $(G,\{\mu\},\calG)$ admits a good local symplectic Hodge embedding $\rho':G\rightarrow \mathrm{GSp}(W)$. 	\end{prop}
\begin{proof} We apply the construction in Proposition \ref{prop: good embedding GL} to obtain a good Hodge embedding $G\rightarrow \GL(W)$. By construction, there is a finite extension $K/F$ such that $W\cong V_K^r$ considered as an $F$-vector space. Let $\Psi:V\times V\rightarrow F$ be the alternating form  on $V$. We define $\Psi':V_K^r\times V_K^r\rightarrow F$ to be the alternating form given by $\Psi'=\sum_{i=1}^r\mathrm{Tr}_{K/F}\circ( \Psi\otimes_FK)$. Then  $G\rightarrow \GL(W)$ factors through a good local symplectic Hodge embedding $\rho':G\rightarrow \GSp(W)$ as desired.
	\end{proof}

\subsubsection{}\label{subsec: admissible set mixed char}Now let $(G,\{\mu\},\calG)$ be an acceptable local model triple of local Hodge type and  $\rho:\tG\rightarrow \GL(\Lambda)$  a good integral local Hodge embedding extending $\rho:G\rightarrow \GL(W)$. We assume that $\tG=\calG$. We finish this section by giving a more explicit description of the embedding $\bbM^{\mathrm{loc}}_{\calG,\{\mu\}}\rightarrow \Gr(\Lambda)\otimes_{\calO_F}\calO_E$ on the level of $k$-points which will be needed in \S\ref{sec: M-adapted}. We assume $\calG$ is a standard parahoric corresponding to a subset $J\subset \bbS$.

As explained in \cite[\S3.6]{Z}, we may identify the $k$-points of $\mathrm{Gr}(\Lambda)$ with a subset of $\GL_{W}(\brF)/\mathcal{GL}_{W}(\calO_{\brF})$, where $\calG\calL_{W}:=\GL(\Lambda)$. The convention in \emph{loc. cit.} is that $g\in\GL_{W}(\brF)/\mathcal{GL}_{W}(\calO_{\brF})\cap \mathrm{Gr}(\Lambda)(k)$ corresponds to the subspace of $\Lambda\otimes_{\calO_F}k$ induced by the reduction mod $\varpi_F$ of the lattice $\varpi_F g\Lambda$.   We thus obtain an inclusion $\bbM^{\mathrm{loc}}_{\G,\{\mu\}}(k)\subset \GL_{W}(\brF)/\mathcal{GL}_{W}(\calO_{\brF})$.

\begin{prop}\label{cor: mu admissible mixed char}Assume $p>2$ and  $\tilde{\calG}=\calG$. Let $g\in G(\brF)$ with $$g\in \G(\calO_{\brF})\dot{w}\G(\calO_{\brF})$$ for some $w\in W_J\backslash W/W_J$.  Then the image of $\rho(g)$ in $\GL_{W}(\brF)/\mathcal{GL}_{W}(\calO_{\brF})$ lies in $\bbM^{\mathrm{loc}}_{\G,\{\mu\}}(k)$ if and only if $w\in \Adm(\{\mu\})_J$.
\end{prop}
\begin{proof}By \cite[Theorem 7.23]{AGLR}, the inclusion $\bbM^{\mathrm{loc}}_{\G,\{\mu\}}(k)\subset \GL_{W}(\brF)/\mathcal{GL}_{W}(\calO_{\brF})$ lifts to an inclusion $\bbM^{\mathrm{loc}}_{\G,\{\mu\}}(k)\subset G(\brF)/\mathcal{G}(\calO_{\brF})$ which identifies $\bbM^{\mathrm{loc}}_{\calG,\{\mu\}}(k)$ with the $\mu$-admissible locus in $G(\breve F)/\G(\calO_{\breve F})$ (i.e. elements of the form $\G(\calO_{\brF})\dot{w}\G(\calO_{\brF})/\G(\calO_{\brF})$ for $w\in \Adm(\{\mu\})_J$). 
	By our assumption that $\G=\tG$,  the morphism $G(\breve F)/\calG(\calO_{\breve F})\rightarrow \GL_{W}(\brF)/\mathcal{GL}_{W}(\calO_{\brF})$ induced by $\rho$  is injective, and the result follows. 
\end{proof}
\begin{remark}
	The reason for the convention  in \cite[\S3.6]{Z} is as follows. Let $\mu$ be the standard minuscule cocharacter of $\GL_n$ given by $a\mapsto\mathrm{diag}(1^{(n-d)},(a^{-1})^{(d)})$. Then on the generic fiber, $\mu$ corresponds to the subspace of $W$ where it acts by weight $-1$. The specialization of this point in $\Gr(\Lambda)(k)$ is the subspace of $\Lambda\otimes_{\calO_F}k$ given by the reduction mod $\varpi_F$ of $\varpi_F \mu(\varpi_F)\Lambda$. Thus with this convention, $\Gr(\Lambda)(k)$ is identified with the $\mu$-admissible locus of $\GL_{W}(\brF)/\mathcal{GL}_{W}(\calO_{\brF})$.
\end{remark} 

\subsection{The versal deformation space with tensors}\label{sec: adapted liftings}\subsubsection{}

We assume $p>2$ and we work over the base field $\bbQ_p$ so that $\brQ=W(k)[\frac{1}{p}]$, where $W(k)$ denotes the Witt vectors of $k$. For any ring $R$ and an $R$-module $M$, we let $M^\otimes$ denote the direct sum of all $R$-modules obtained from $M$ by taking duals, tensor products, symmetric and exterior products. If $R$ is a complete local ring with residue field of positive characteristic and $\scrG$ is a $p$-divisible group over $R$, we write $\bbD(\scrG)$ for its (contravariant) Dieudonn\'e crystal.

\subsubsection{}\label{subsec: deformation space with crystalline tensors} Let $\scrG_0$ be a $p$-divisible group over $k$ and set $\bbD:=\bbD(\scrG_0)(\brZ)$. We write $\varphi$ for the Frobenius on $\bbD$ and $\bbD_1\subset \bbD$ the preimage of the filtration on $\bbD(\scrG_0)(k)$.
Let $(s_{\alpha,0})\subset\bbD^\otimes$ be  a collection of $\varphi$-invariant tensors whose image in $\bbD(\scrG_0)(k)^\otimes$ lie in $\mathrm{Fil}^0$. We assume that there exists a $\bbZ_p$-module $U$ and an isomorphism \begin{equation}
\label{eqn: triv Dieudonne}
U\otimes_{\bbZ_p}\brZ\cong\bbD\end{equation} such that $s_{\alpha,0}\in U^\otimes$. Write $\widetilde{\G}\subset \GL(U)$ for the pointwise stabilizer of $\{s_{\alpha,0}\}_{\alpha}$ so that $\widetilde{\G}_{\brZ}$ can be identified with the stabilizer of $s_{\alpha,0}$ in $\GL(\bbD)$. 
We assume  that the generic fiber $G:=\widetilde{\calG}_{\otimes_{\bbZ_p}}\bbQ_p$ is a reductive group and that $\widetilde{\G}=\widetilde{\calG}_x$ for some $x\in \calB(G,\bbQ_p)$. We write $\G$ for the associated parahoric group scheme. 

 Let $P\subset \GL(\D)$ be a parabolic subgroup lifting the parabolic $P_0$ corresponding to the filtration on $\D(\pdiv_0)(k)$. Write $\bbM^{\mathrm{loc}}=\GL(\D)/P$ and $\mbox{Spf}A=\widehat{\bbM}^{\mathrm{loc}}$ the completion 
of $\bbM^{\mathrm{loc}}$ at the identity;  then $A$ is isomorphic to a power series ring over $\brZ$. Let $K'/\brQ$ be a finite extension and $y:A\rightarrow K'$ a continuous map such that $\sa\in \mathrm{Fil}^0\D^\otimes\otimes_{\brZ}K'$ for the filtration induced by $y$ on $\D^\otimes\otimes_{\brZ}K'$. By \cite[Lemma 1.4.5]{Ki2}, the filtration corresponding to $y$ is induced by a $G$-valued cocharacter $\mu_y$ (by convention $\mu_y$ has weights $(0,1)$).  Let $G.y$ be the orbit of $y$ in $\bbM^{\mathrm{loc}}\otimes_{\brZ}K'$ which is defined over a finite extension $\brE/\brQ$, and we write $\bbM^{\mathrm{loc}}_{\calG}$ for the closure of this orbit in $\bbM^{\mathrm{loc}}$.

\subsubsection{}  Let $R$ be a complete local ring with maximal ideal $\mathfrak{m}$ and residue field $k$. We let $W(R)$ denote the Witt vectors of $R$. Recall \cite{Zi1} we have a subring $$\widehat{W}(R)=W(k)\oplus\mathbb{W}(\mathfrak{m})\subset W(R),$$ where $\mathbb{W}(\mathfrak{m})\subset W(R)$ consists of Witt vectors $(w_i)_{i\geq1}$ with $w_i\in\mathfrak{m}$ and $w_i\rightarrow 0$ in the $\mathfrak{m}$-adic topology. The Frobenius of $W(R)$ induces a map $\varphi:\widehat{W}(R)\rightarrow \widehat{W}(R)$, and we write $I_R$ for the kernel of the projection $\widehat{W}(R)\rightarrow R$.
We recall the following definition, which is \cite[Definition 4.6]{Z} in the case that $G$ splits over a tamely ramified extension of $\bbQ_p$.  

\begin{definition}\label{def: G-adapted}
	Let $K/\brQ$ be a finite extension. Let $\pdiv$ be a $p$-divisible group over $\Ok_K$ whose special fiber is isomorphic to $\pdiv_0$. We say $\pdiv$ is $(\widetilde{\G},\mu_y)$-adapted if the tensors $\sa$ extend to Frobenius invariant tensors $\widetilde{s}_\alpha\in\D(\pdiv)(\hW(\Ok_K))^\otimes$ such that the following two conditions hold:
	
	\begin{enumerate}
		\item 	There is an isomorphism $\D(\pdiv)(\hW(\Ok_K))\cong \D\otimes_{\brZ}\hW(\Ok_K)$ taking $\widetilde{s}_\alpha$ to $\sa$.
		
		\item Under the canonical identification $$\D(\pdiv)(\calO_K)\otimes_{\calO_K}K\cong \D\otimes_{\brZ}K$$ given by \cite[Lemma 3.1.17]{KP}, the filtration on $\D\otimes_{\brZ}K$ is induced by a $G$-valued cocharacter conjugate to $\mu_y$. 
	\end{enumerate}
\end{definition}

\subsubsection{}\label{sssec: very good}   Consider the local model triple $(G,\{\mu_y^{-1}\},\calG)$ with reflex field $E$.
We  assume in addition that the following  conditions are satisfied:
\begin{enumerate}
	\item[(A)]$(G,\{\mu_y^{-1}\},\calG)$ is acceptable and of local Hodge type.
	\item[(B)]The embedding $\rho:\tG\rightarrow \GL(U)$ is  a good local integral Hodge embedding. 
\end{enumerate}

Under assumption (B),  property (3) of Definition \ref{def: good embedding} implies that  the definition of $\bbM^{\mathrm{loc}}_{\G}$ above agrees with the  local model $\bbM^{\mathrm{loc}}_{\G,\{\mu_y^{-1}\}}\otimes_{\calO_E}\calO_{\breve E}$.
We write  $\widehat{\bbM}^{\mathrm{loc}}_{\G}\cong \mathrm{Spf}A_{\widetilde{\G}}$ for the completion of $\bbM^{\mathrm{loc}}_{\G}$ at the point $\overline{y}\in \Mloc_{\calG}(k)$ corresponding to the identity element. Then $A_{\tG}$ is normal and we have a natural surjective map $A\otimes_{\brZ}\calO_{\breve E}\rightarrow A_{\tG}$ corresponding to the closed immersion $\widehat{\mathbb{M}}^{\mathrm{loc}}_{\G}\subset \widehat{\bbM}^{\mathrm{loc}}\otimes_{\brZ}\calO_{\breve E}$.

We also make the following assumption \begin{enumerate}\item[(C)] The embedding $\tG\rightarrow \GL(U)$ is very good at the point $\overline{y}\in \Mloc_{\calG}(k)$ in the sense of \cite[Definition 5.2.5]{KPZ}.
	\end{enumerate}
We briefly recall this notion, which was erroneously omitted in \cite{KP} and previous versions of this manuscript.	We set $M=U\otimes_{\bbZ_p}\widehat{W}(A)$. Let $\overline{M}_1\subset M/I_AM$ be the universal direct summand of $\widehat{\bbM}_{\calG}^{\mathrm{loc}}$ and $M_1\subset M$ the preimage of $\overline{M}_1$.  Let $\widetilde{M}_1$ denote the image of the map $$M_1\otimes_{\widehat{W}(A),\varphi}\widehat{W}(A)\rightarrow M\otimes_{\widehat{W}(A),\varphi}\widehat{W}(A).$$ 
	
	By the argument of \cite[Corollary 3.2.11]{KP}, we have $s_{\alpha,0}\in \widetilde{M}_1^\otimes\otimes_{\widehat{W}(A)} \widehat{W}(A_{\tG})$, and the scheme  \begin{equation}\label{eqn: Dieudonne torsor}\calT=\underline{\mathrm{Isom}}_{s_{\alpha,0}}(\widetilde{M}_1\otimes_{\widehat{W}(A)}\widehat{W}(A_{\tG}),M\otimes_{\widehat{W}(A)}\widehat{W}(A_{\tG}))\end{equation} of isomorphisms which preserve the tensors $s_{\alpha,0}$ is a trivial $\widetilde\calG$-torsor. 
	
	Let $\fkm_{A_{\tG}}$ denote the maximal ideal in $A_{\tG}$ and set $\fka_{\tG}:=\fkm_{A_E}^2+\pi_EA_E$, where $\pi_E\in E$ is a unifomizer. We also let  $U_1\subset U\otimes_{\bbZ_p}\brZ$ denote the preimage of the filtration corresponding to $\overline{y}$ (this corresponds to the submodule $\bbD_1$ under the identification with $U\otimes_{\bbZ_p}\brZ$), and we let $\widetilde{U}_1$ denote the image of $\varphi^*(U_1)\rightarrow \varphi^*(U\otimes_{\bbZ_p}\brZ)$.  By \cite[Lemma 5.1.3]{KPZ} (cf. \cite[Corollary 3.1.9]{KP}), there is a canonical isomorphism $$c:\widetilde{U}_1\otimes_{\brZ}\widehat{W}(A_{\tG}/\fka_{\tG})\xrightarrow{\sim}\widetilde{M}_1\otimes_{\widehat{W}(A)}\widehat{W}(A_{\tG}),$$ and the embedding $\tG\rightarrow \GL(U)$ is said to be very good at $\overline{y}$ if we can choose a collection of $s_{\alpha,0}$ cutting out $\tG\subset \GL(U)$ such that $c(s_{\alpha,0}\otimes 1)=s_{\alpha,0}\otimes 1$. This is equivalent the condition  that $c$ defines an isomorphism of $\tG$-torsors and hence is independent of the choice of $s_{\alpha,0}$. More generally, we say that the integral local Hodge embedding $\tG\rightarrow \GL(U)$ is very good, if it is very good at all points of $\Mloc_{\calG}(k)$.

\subsubsection{}

With the corrected assumptions above, we may now apply the construction in \cite[3.2]{KP}; the following is essentially \cite[Proposition 3.2.17]{KP}.

\begin{prop}\label{prop: versal deformation space tensors}
	There exists a versal $p$-divisible group $\scrG_A$ over $\mathrm{Spf} A\otimes_{\brZ}\calO_{\brE}$ deforming $\pdiv_0$ such that   for any $K/\brQ$ finite, a map $\varpi:A\otimes_{\brZ}\Ok_E\rightarrow K$ factors through $A_{\widetilde{\G}}$ if and only if the $p$-divisible group $\pdiv_{\varpi}$ given by the base change of $\scrG_A$ along $\varpi$ is $(\widetilde{\G},\mu_y)$-adapted. 
\end{prop}
\begin{proof} Under our assumptions and using \cite[Proposition 10.3]{Anschutz} (see also \cite[Proposition 5.3.2]{PRshtukas} for a different proof  which  applies in our setting) in place of \cite[Proposition 1.4.3]{KP}, we find that the conditions (3.2.2)-(3.2.4) of \cite{KP} are satisfied; we may thus apply the construction in \cite[\S3.2]{KP} to obtain $\scrG_A$. We briefly recall the construction.

Recall the trivial $\tG$-torsor $\calT=\underline{\mathrm{Isom}}_{s_{\alpha,0}}(\widetilde{M}_1\otimes_{\widehat{W}(A)}\widehat{W}(A_{\widetilde\calG}),M\otimes_{\widehat{W}(A)}\widehat{W}(A_{\widetilde\calG}))$ of tensor-preserving isomorphisms from the previous paragraph.  We let $$\Psi_{A_{\widetilde\calG}}:\widetilde{M}_1\otimes_{\widehat{W}(A)}\widehat{W}(A_{\widetilde\calG})\xrightarrow{\sim}M\otimes_{\widehat{W}(A)}\widehat{W}(A_{\widetilde\calG})$$ be a section of $\calT$ which is constant mod $\fka_E$ in the sense of \cite[\S3.1.11]{KP}; such a section exists by assumption (C). We then lift $\Psi_{A_{\widetilde{\calG}}}$ to an isomorphism $$\Psi:\widetilde{M}_1\otimes_{\widehat{W}(A)}\widehat{W}(A\otimes_{\breve\bbZ_p}\calO_{\breve E})\xrightarrow{\sim}M\otimes_{\widehat{W}(A)}\widehat{W}(A\otimes_{\breve\bbZ_p}\calO_{\breve E})$$ which is constant mod $\fka_E$. By \cite[Lemma 3.1.5]{KP}, this gives rise to a Dieudonn\'e display over $A\otimes_{\breve\bbZ_p}\calO_{\breve E}$, and hence to a $p$-divisible group $\scrG_A$ over $\Spf A\otimes_{\breve\bbZ_p}\calO_{\breve E}$ which is versal by \cite[Lemma 3.1.12]{KP}.

	By construction, the base change $\scrG_{A_{\tilde{\calG}}}:=\scrG_A\otimes _{A\otimes_{\brZ}\calO_{\breve E}}{A_{\tilde{\calG}}}$ is equipped with Frobenius invariant tensors $s_{\alpha,0,A_{\widetilde{\calG}}}\in \bbD(\scrG_{A_{\tilde{\calG}}})(\widehat{W}(A_{\tilde{\calG}}))^\otimes$.
It  is then  clear that for $\varpi:A_{\widetilde{\G}}\rightarrow K$, the tensors $s_{\alpha,0}$ extend to
 $$\widetilde{s}_\alpha\in \bbD(\scrG_{\varpi})(\widehat{W}(\Ok_K))^\otimes$$ 
 so that Definition \ref{def: G-adapted} (1) is  satisfied. Indeed the tensors $\widetilde{s}_\alpha$ are obtained from 
 $s_{\alpha,0,A_{\widetilde{\G}}}$ via base change. The argument in \cite[Proposition 4.8]{Z} shows that condition (2) is also satisfied, so that $\scrG_{\varpi}$ is $(\widetilde\G,\mu_y)$-adapted.
	
The converse is \cite[Proposition 3.2.17]{KP}.
\end{proof}

\subsection{Deformations with \'etale tensors}\label{sec: etale tensors}

\subsubsection{}
Let $K/\brQ$ be a finite extension and $\scrG$  a $p$-divisible group over $\calO_K$ with special fiber $\scrG_0$. We write $T_p\scrG$ for the $p$-adic Tate-module of $\scrG$ and $T_p\scrG^\vee$ its linear dual. We let $s_{\alpha,\et}\in T_p\scrG^{\vee \otimes}$ be a collection of tensors whose stabilizer $\widetilde{\calG}$ has reductive generic fiber $G$ and $\widetilde{\calG}=\widetilde{\calG}_x$ for some $x\in \calB(G,\bbQ_p)$. We write $\bbD:=\bbD(\scrG_0)(\brZ)$ and we let 
$$s_{\alpha,0}\in D_{\mathrm{cris}}(T_p\scrG^\vee\otimes_{\bbZ_p}\bbQ_p)^\otimes \simeq  \bbD^\otimes\otimes_{\brZ}\brQ$$ 
denote the $\varphi$-invariant tensors corresponding to the image of $s_{\alpha,\et}$ under the $p$-adic comparison isomorphism. 

\begin{prop}\label{prop: trivialization of Dieudonne}
\begin{enumerate}\item We have $s_{\alpha,0}\in \bbD^\otimes\subset \bbD^\otimes\otimes_{\brZ}\brQ. $ 
Moreover the $s_{\alpha,0}$ extend canonically to tensors $\widetilde{s}_\alpha\in\bbD(\scrG)(\widehat{W}(\Ok_K))^\otimes$ and there exists an isomorphism 
\begin{equation}
	\label{eqn: trivialization of display}
	T_p\scrG^\vee\otimes_{\brZ}\widehat{W}(\Ok_K)\cong\bbD(\scrG)(\widehat{W}(\Ok_K))	\end{equation} 
	taking $s_{\alpha,0}$ to $\widetilde{s}_\alpha$.
	\item There exists a $G$-valued cocharacter $\mu_y$ such that 
	\begin{enumerate}[label=(\roman*)]
	\item Under the canonical isomorphism 
	$$\gamma:\bbD\otimes_{\brZ} K\cong \bbD(\scrG)(\Ok_K)\otimes_{\calO_K} K,$$ 
	the filtration is induced by a $G$-valued cocharacter conjugate to $\mu_y$.
	\item The filtration on $\bbD\otimes_{\brZ} K$ induced by $\mu_y$ lifts the filtration on the module $\bbD(\scrG_0)\otimes_{\brZ}k$.
\end{enumerate}
Here we consider $G_{\brQ}\subset \GL(\bbD\otimes_{\brZ}\brQ)$ via  base change of (\ref{eqn: trivialization of display}) to $\brQ$.
	\end{enumerate}
\end{prop}
\begin{proof} The argument is the same as \cite[Proposition 3.3.8, Corollary 3.3.10]{KP}, where  again we are using \cite[Proposition 10.3]{Anschutz} in place of \cite[Proposition 1.4.3]{KP}.
\end{proof}

\subsubsection{}The isomorphism (\ref{eqn: trivialization of display}) induces an isomorphism $$T_p\scrG^\vee\otimes_{\bbZ_p}\brZ\cong \bbD$$ taking $s_{\alpha,\et}$ to $s_{\alpha,0}$ which we now fix. Taking $T_p\scrG^\vee$ to be $U$, we place ourselves in the setting of \S\ref{subsec: deformation space with crystalline tensors}. Therefore  we have a notion of $(\widetilde{\G},\mu_y)$-adapted lifting where $\mu_y$ is as in Proposition \ref{prop: trivialization of Dieudonne}. Moreover it follows from the same proposition that $\scrG$ itself  is a $(\widetilde{\G},\mu_y)$-adapted lifting. The next proposition then follows immediately from Proposition \ref{prop: trivialization of Dieudonne} and the definition of $(\widetilde{\G},\mu_y)$-adapted liftings (cf. \cite[Proposition 3.3.13]{KP}).

\begin{prop}\label{prop: G-adapted etale}
	Let $K'/\brQ$ be a finite extension and let ${\scrG}'$ be a deformation of $\scrG_0$ to $\Ok_{K'}$ such that 
\begin{enumerate} \item The filtration on $\bbD\otimes_{\brZ} K'$ corresponding to $\scrG'$ is induced by a $G$-valued cocharacter conjugate to $\mu_y$.
	
\item The tensors $s_{\alpha,0}\in\bbD^\otimes$ correspond to tensors $s_{\alpha,\et}\in T_p\scrG'^{\vee\otimes}$ under the $p$-adic comparison isomorphism.

\end{enumerate}
	Then ${\scrG}'$ is $(\widetilde{\G},\mu_y)$-adapted lifting. 
\end{prop}\qed 
\subsection{Canonical liftings for $\mu$-ordinary $p$-divisible groups}\label{sec: M-adapted}
\subsubsection{}\label{sec: mu ordinary pdiv gp} 

We now study the deformation theory of $\mu$-ordinary $p$-divisible groups. The results in this subsection will be used in \S\ref{sec: canonical liftings} to prove our main result on CM (special) liftings for Shimura varieties.

We return to the setting of \S\ref{sec: adapted liftings}. Thus $\scrG_0$ is a $p$-divisible group over $k $ equipped with $s_{\alpha,0}\in \bbD^\otimes$. We fix a $\brZ$-linear isomorphism \begin{equation}
\label{eqn: trivialization Dieudonne}
U\otimes_{\Z_p}\brZ\cong\D(\pdiv_{0})
\end{equation} 
as in (\ref{eqn: triv Dieudonne}) so that $s_{\alpha,0}\in U^\otimes$ and assume that (A) and (B) are satisfied. In \S\ref{sec: M-adapted}, we will assume in addition  that  $\calG=\widetilde{\calG}$, so that we have a closed immersion $\calG\rightarrow \GL(U)$. Since the $s_{\alpha,0}$ are $\varphi$-invariant, the Frobenius is given by $b\sigma$ for an element $b\in G(\brQ)$, and modifying (\ref{eqn: trivialization Dieudonne}) by an element $h\in \G(\brZ)$ modifies $b$ by $b\mapsto h^{-1}b\sigma(h)$. Therefore $b$ is well-defined up to $\sigma$-conjugation by an element of $\G(\brZ)$ and in particular we obtain a well-defined class $[b]\in B(G)$. 

We choose a maximal $\brQ$-split torus $S$ of $G$ defined over $\Q_p$ such that $x\in \calA(G,S,\brQ)$ and we let $T$ be its centralizer. We fix a $\sigma$-stable alcove $\fka\subset \calA(G,S,\brQ)$ such that $x$ lies in the closure of $\fka$;  this determines a set of simple reflections $\bbS$ for $W$, and $\calG$ corresponds to the subset $J\subset \bbS$  of reflections which fix $x$.  We follow the notation of \S\ref{sec: group theoretic} and let $\widetilde{\mu}\in X_*(T)$ denote the dominant (with respect to a choice of Borel defined over $\brQ$) representative of the conjugacy class $\{\mu_y\}$; we write ${\mu}$ for its image in $X_*(T)_I$. We have a closed immersion of local models $$\bbM^{\mathrm{loc}}_{\G,\{\mu_y^{-1}\}}\hookrightarrow \mathrm{Gr}(U)\otimes_{\bbZ_p}\calO_E,$$ where $\mathrm{Gr}(U)$ classifies submodules of $U$ of rank $\mathrm{dim}_k\mathrm{Fil}^1\bbD\otimes_{\brZ} k.$ By definition, the filtration on $\bbD\otimes_{\brZ}k$ corresponds to an element of $\mathrm{Gr}(U)(k)$ which lies in $\bbM^{\mathrm{loc}}_{\G,\{\mu_y^{-1}\}}(k)$. This filtration  is by definition the kernel of $\varphi$. Thus its preimage in $\bbD$ is given by 
$\{v\in\bbD|b\sigma(v)\in p\bbD\}$, which  
 is just the $\brZ$-lattice $\sigma^{-1}(b^{-1})p\bbD$. It follows from Corollay \ref{cor: mu admissible mixed char} that $\sigma^{-1}(b^{-1})\in \G(\brZ)\dot{w}\G(\brZ)$ for some element $w\in\Adm(\{\mu_y^{-1}\})_J$, and hence that $$b\in\G(\brZ)\sigma(\dot{u})\G(\brZ)$$ for some $u\in \Adm(\{\mu_y\})_J$. In particular we have $[\sigma^{-1}(b)]\in B(G,\{\mu_y\})$ by \cite[Theorem 1.1]{He3}.  

\subsubsection{}\label{subsec: mu ordinary setting}Now assume the existence of $[b]_{\mu}\in B(G,\{\mu_y\})$ as in Definition \ref{def: mu ordinary}, and that  $\sigma^{-1}(b)\in [b]_{\mu}$. We will  construct a $(\calG,\mu_y)$-adapted (recall $\widetilde{\calG}=\calG$) deformation of $\pdiv_0$ which will be the analogue of the Serre--Tate canonical lift in this context.

By Proposition \ref{prop: F-crystal basis} applied to $\sigma^{-1}(b)$, there exists an element $h\in \G(\brZ)$ such that $h^{-1}b\sigma(h)=\sigma(\dot{t}_{\mu'})$ for some $\mu'\in W_0\cdot\mu$ with $t_{\mu'}$ $\sigma$-straight. Upon modifying the isomorphism (\ref{eqn: trivialization Dieudonne}), we may assume $b=\sigma(\dot t_{\mu'})$; we fix this choice of (\ref{eqn: trivialization Dieudonne}) from now on. Let $M$ be the semistandard Levi subgroup of $G$  corresponding to $\nu_{t_{\mu'}}=\nu_{\sigma(t_{\mu'})}$; then $t_{\mu'}$ is central in $W_M$ by Lemma \ref{lemma: cochar central}. 
Let $w\in W_0$ such that $w\cdot{\mu}=\mu'$ and write $\widetilde{\lambda}:=(w\cdot\widetilde{\mu})$; then by Lemma \ref{lemma: cochar central 2}, $\widetilde{\lambda}$ is central in $M$.

 Let $$\calM(\brZ):=M(\brQ)\cap\G(\brZ),$$ which is the $\brZ$-points of a parahoric group scheme $\calM$ of $M$ defined over $\bbZ_p$. Explicitly, we have an identification of apartments $\calA(G,S,\brQ)\cong\calA(M,S,\brQ)$  and hence we may consider $x$ as an element of $\calA(M,S,\brQ)$ which determines the parahoric $\calM=\calM_x$.  Since $\calM(\brZ)$ is stable under $\sigma$, $\calM$ is defined over $\bbZ_p$. 

 The kernel of the map $\pi_1(M)\rightarrow \pi_1(G)$ is freely generated by a subset of the roots of $G$ which are not roots of $M$, and which are stable under the action of $\Gamma$.  Hence $\ker(\pi_1(M)\rightarrow \pi_1(G))$ is  an induced module for the action of $\Gamma$ and $\pi_1(M)_I\rightarrow \pi_1(G)_I$ has torsion-free kernel. Since  $\tG=\G$,  it follows from this fact that the image of $\widetilde{\calM}(\brZ)$ in $\pi_1(M)_I$ is trivial, and hence $\widetilde{\calM}=\calM$.

\begin{lemma} 
	Let $K$ be the field of definition of $\widetilde{\lambda}$.
 The filtration induced by $\widetilde{\lambda}$ on $\bbD\otimes_{\brZ} K$ specializes to $\mathrm{Fil}^1\bbD\otimes_{\brZ} k$.

\end{lemma}
\begin{proof}  The cocharacter $\widetilde{\lambda}^{-1}$ determines a $K$-point $s_{\widetilde{\lambda}^{-1}}$ of $\bbM^{\mathrm{loc}}_{\calG,\{\mu_y^{-1}\}}$  and whose image in $\bbM^{\mathrm{loc}}=\Gr(U)\otimes_{\bbZ_p}\brZ$ corresponds to the filtration induced by $\widetilde{\lambda}$.
	
	By \cite[21.3.1]{SW2}, applied to the torus $T$, the point $s_{\tilde{\lambda}^{-1}}$ reduces to the point $\dot{t}^{-1}_{\mu'}\in \bbM^{\mathrm{loc}}_{\calG,\{\mu_y^{-1}\}}(k)\subset G(\breve\bbQ_p)/\calG({\breve\bbZ_p})$. By construction of the embedding $$\bbM^{\mathrm{loc}}_{\calG,\{\mu_y^{-1}\}}(k)\hookrightarrow \GL_U(\brQ)/\GL_{U}(\brZ)$$ in \S\ref{subsec: admissible set mixed char}, the filtration on $\bbD\otimes_{\brZ}k$ corresponding to the image of this element is given by the mod $p$ reduction of $\dot{t}_{\mu'}^{-1}p\bbD=\sigma^{-1}(b^{-1})p\bbD$. The proposition follows.
\end{proof}

\subsubsection{} \label{sec: M-adapted 2} We extend the tensors $s_{\alpha,0}\in U^\otimes$ to a set of  tensors $t_{\beta,0}\in U^\otimes$ whose stabilizer is $\calM$. 
Viewed in $\D \simeq U\otimes_{\Z_p}\brZ,$ the $t_{\beta,0}$ are $\varphi$-invariant as $b=\sigma(\dot t_{\mu'}) \in M(\brQ).$ 
Since $\widetilde\lambda$ is an $M$-valued cocharacter, we may apply the construction in \S\ref{sec: adapted liftings} to $M$ and the tensors $t_{\beta,0}$. In particular we have a notion of $(\mathcal{M},\widetilde\lambda)$-adapted liftings of $\pdiv_0$. It is clear from the definition that any $(\mathcal{M},\widetilde\lambda)$-adapted lifting is also a $(\G,\mu_y)$-adapted lifting.

Let $J_b$ denote the $\sigma$-centralizer group for $b$. It is a reductive group over $\bbQ_p$ such that $$J_b(R):=\{g\in G(\brQ\otimes_{\bbQ_p}R)|g^{-1}b\sigma(g)=b\}$$
for any $\bbQ_p$-algebra $R$. There is an action of $J_b(\Q_p)$ on $\pdiv_0$ in the isogeny category. Since $\nu_{g^{-1}b\sigma(g)}=g^{-1}\nu_bg$ for any $g\in G(\brQ)$, it follows that for $b=\sigma(\dot{t}_{\mu'})$, we have $J_b(\Q_p)\subset M(\brQ)$. 

\begin{thm}\label{thm: can-lift}Assume we are in the setting of \S\ref{subsec: mu ordinary setting} so that $b=\sigma(\dot{t}_{\mu'})$. Let $K/\brQ$ be an extension over which $\widetilde{\lambda}$ is defined, and suppose $\widetilde{\G}=\G$. There exists a $(\calG,\mu_y)$-adapted lifting ${\pdiv}$ to $\Ok_K$ such that the action of $J_b(\Q_p)$ on $\pdiv_0$ lifts to ${\pdiv}$ in the isogeny category.
\end{thm}
\begin{proof}Suppose there exists an $(\calM,\widetilde{\lambda})$-adapted lifting $\scrG$ of $\scrG_0$; from the above discussion, we have that $\scrG$ is also a $(\calG,\mu_y)$-adapted lifting.	By Definition \ref{def: G-adapted} (2), the filtration on the weakly admissible filtered $\varphi$-module associated to $T_p{\pdiv}^\vee$ is induced by an $M$-valued cocharacter conjugate to $\widetilde{\lambda}$, hence by $\widetilde{\lambda}$ itself since it is central in $M$. Since $J_b(\bbQ_p)\subset M(\brQ)$, the action of $J_b(\bbQ_p)$ respects the filtration and hence lifts to an action on $\scrG$ in the isogeny category.

	It suffices to show the existence of an $(\calM,\widetilde{\lambda})$-adapted lifting. This follows from the same argument as \cite[Proposition 4.9]{Z}; we briefly recall the construction for the convenience of the reader. 
	
	We set $\fkS:=\brZ[[u]]$ and we let $\sigma:\fkS\rightarrow \fkS$ be the map given by the usual Frobenius on $\brZ$ and $u\mapsto u^p$. We define $\fkM:=\bbD\otimes_{\sigma^{-1},\brZ}\fkS$, so that $\sigma^*(\fkM)\cong \bbD\otimes_{\brZ}\fkS$, and we let $\calF\subset \sigma^*(\fkM)$ denote the preimage of the filtration induced by $\widetilde{\lambda}$ on $\bbD\otimes_{\brZ}\calO_K$; here the map $\fkS\rightarrow \calO_K$ is induced by sending $u$ to a uniformizer $\varpi$ in $\calO_K$.  Then $\calF$ is a free $\fkS$-module and $t_{\beta,0}\in \calF^\otimes$; this follows from the argument in  \cite[Lemma 3.2.6]{KP}  using \cite[Proposition 10.3]{Anschutz} in place of \cite[Proposition 1.4.3]{KP}. Moreover the scheme of $\fkS$-linear isomorphisms $\underline{\mathrm{Isom}}_{t_{\beta,0}}(\calF,\sigma^*(\fkM))$ taking $t_{\beta,0}$ to $t_{\beta,0}$ is a trivial $\calM$-torsor. Then arguing as in \cite[Proposition 4.9]{Z}, we may construct a morphism $\varphi:\sigma^*(\fkM)\rightarrow \fkM$ satisfying the following properties:
	\begin{itemize}
		\item The map $\varphi$  gives $\fkM$ the structure of an element of $\mathrm{BT}^\varphi$ (see \cite[\S4.1]{Z} for the definition of $\mathrm{BT}^\varphi$).
		
		\item The canonical identification $\sigma^*(\fkM/u\fkM)\cong \bbD$ is an isomorphism of $F$-crystals.
		\item $\varphi$ preserves the tensors $t_{\beta,0}$.
		
	\end{itemize} By \cite[Theorem 1.4.2]{Ki2}, $\fkM$ corresponds to a $p$-divisible group $\scrG$ over $\calO_K$, and the argument of \cite[Proposition 4.9]{Z} shows that $\scrG$ is an $(\calM,\widetilde{\lambda})$-adapted lifting.
\end{proof}

\section{Integral models of Shimura varieties and canonical liftings}\label{sec: integral models for Shimura + canonical lifts}

In this section we establish the main geometric properties of  integral models for Shimura varieties  that are needed for later applications. These include a version of the local model diagram and the existence of canonical liftings over the $\mu$-ordinary locus which are proved in some special cases in \S\ref{subsec: integral models Hodge type}. In \S4.2--4.4, these results are extended to the main cases of interest, certain Shimura varieties which we term \emph{strongly acceptable}, see Definition \ref{def: strongly acceptable}.

\subsection{Integral models} \label{subsec: integral models Hodge type}
\subsubsection{}\label{subsec: integral models Hodge type preamble} For the rest of this paper we fix an algebraic closure $\overline{\Q}$, and for each place $v$ of $\Q$ (including $v=\infty$) an algebraic closure $\overline{\Q}_v$ together with an embedding $i_v:\overline{\Q}\rightarrow \overline{\Q}_v$ (here $\overline{\bbQ}_\infty\cong \bbC$).
Let $\bfG$ be a reductive group over $\Q$ and $X$ a $\bfG_{\bbR}$-conjugacy class of homomorphisms
$$h:\mathbb{S}:=\text{Res}_{\mathbb{C}/\R}\mathbb{G}_m\rightarrow \bfG_\mathbb{R}$$
such that $(\bfG,X)$ is a Shimura datum in the sense of \cite{De}. 

Let $c$ be complex conjugation. Then $\bbS(\C)=(\C\otimes_{\R}\C)^\times\cong \C^\times \times c^*(\C^\times)$ and we write $\mu_h$ for the cocharacter given by $$\C^\times\rightarrow \C^\times\times c^*(\C^\times)\xrightarrow h \bfG(\C).$$
We set $w_h:=\mu_h^{-1}\mu_h^{c-1}$.

For the rest of this section, we fix a prime $p>2$ and we set $G:=\bfG_{\bbQ_p}$. Let $\A_f$ denote the ring of finite adeles and $\A_f^p$ the ring of prime-to-$p$ adeles which we consider as the subgroup of $\A_f$ with trivial $p$-component. Let $\rmK_p\subset \bfG(\bbQ_p)$ and $\rmK^p\subset\bfG(\A_f)$ be  compact open subgroups  and write $\rmK:=\rmK_p\rmK^p$. Then if $\rmK^p$ is sufficiently small (in fact if $\rmK^p$ is neat, see \cite[p. 34]{Milne}), the set \begin{equation}\label{eqn: points of Shimura var}\Sh_{\rmK}(\bfG,X)_{\bbC}=\bfG(\bbQ)\backslash X\times \bfG(\bbA_f)/\rmK\end{equation}can be identified with the complex points of a smooth algebraic variety. The theory of canonical models implies that $\Sh_{\rmK}(\bfG,X)_{\bbC}$ has a model  $\Sh_{\rmK}(\bfG,X)$ over  the reflex field $\bfE\subset \bbC$, which is defined to be the field of definition of the conjugacy class $\{\mu_h\}$. We may consider $\bfE$ as a subfield of  $\overline{\bbQ}$ via the embedding $i_\infty:\overline{\bbQ}\hookrightarrow \bbC$ and we write $\calO_{\bfE}$ for the ring of integers of $\bfE$. For a general compact open subgroup $\rmK$, we take a sufficiently small  compact open subgroup $\rmK_1^p$  which is normal in $\rmK^p$ and define the Shimura stack $\mathrm{Sh}_{\rmK}(\bfG,X)$ to be the quotient $\mathrm{Sh}_{\rmK_p\rmK_1^p}(\bfG,X)/(\rmK^p/\rmK_1^p)$; it is a smooth algebraic stack over $\bfE$.

We also define $$\Sh_{\rmK_p}(\bfG,X):=\lim_{\leftarrow\rmK^p}\Sh_{\rmK_p\rmK^p}(\bfG,X)$$ $$ \Sh_{\rmK}(\bfG,X):=\lim_{\leftarrow\rmK}\Sh_{\rmK}(\bfG,X);$$ 
these are pro-varieties equipped with actions of $\bfG(\bbA_f^p)$ and $\bfG(\bbA_f)$ respectively.

\subsubsection{}\label{subsubsec: Hodge type assumptions} We now assume that there is an embedding of Shimura data $$\iota:(\bfG,X)\rightarrow (\mathbf{GSp}(V),S^{\pm}).$$Here $\mathbf{GSp}(V)$ is the group of symplectic similitudes of a $\Q$-vector space $V$ equipped with a perfect alternating bilinear
form $\Psi$, and $S^\pm$ is the Siegel double space. Such an $\iota$ is called a Hodge embedding and we say $(\bfG,X)$ is of Hodge type.

Let  $v|p$ be a prime of $\bfE$; upon modifying $i_p:\overline{\bbQ}\rightarrow \overline{\bbQ}_p$, we may assume $v$ is induced by this embedding. We let  $\calO_{\bfE_{(v)}}$ denote the localization of $\calO_{\bfE}$ at $v$, and we write $E$ for the completion of $\bfE$ at $v$. 
We let $k_E$ denote the residue field at $v$ and we fix an algebraic closure $k$ of $k_E$. We let 
$\widetilde{\calG}:=\widetilde{\calG}_x$  for some $x\in \calB(G,\bbQ_p)$  and we write $\calG$ for the associated parahoric. We obtain a local model triple $(G,\{\mu_h\},\calG)$ and the base change $\iota_{\bbQ_p}$ gives a local (symplectic) Hodge embedding (note that $\iota(\bfG)$ contains the scalars since it contains the image of $w_h$).  Thus $(G,\{\mu_h\},\calG)$ is acceptable and of local Hodge type (see \cite[Remark 3.1.5]{KPZ}). Then we have the attached local model $\bbM^{\mathrm{loc}}_{\calG,\{\mu_h\}}$ from \S\ref{sec: embedding into Grassmannian}. The Hodge embedding $\iota$ is said to be \emph{good}  if the corresponding local Hodge embedding $\iota_{\bbQ_p}:G\rightarrow \mathrm{GSp}(V_{\bbQ_p})$ is good. 

\subsubsection{} \label{subsub: integral model Hodge type construction}

For the rest of \S\ref{subsec: integral models Hodge type}, we make the following assumptions, cf.  \S\ref{sssec: very good}.
\begin{enumerate}\item[(A')]$(\bfG,X)$  of Hodge type and $\tG=\calG$.
\item[(B')]$\iota:(\bfG,X)\rightarrow (\mathbf{GSp}(V),S^\pm)$ extends to a good integral local Hodge embedding $\calG\rightarrow \GL(V_{\bbZ_p})$ where $V_{\bbZ_p}\subset V_{\bbQ_p}$ is a $\bbZ_p$-lattice.\end{enumerate}
The following lemma gives sufficient conditions for the existence of an $\iota$ as in (B').
\begin{lemma}\label{lem: v good Hodge embedding}Let $(\bfG,X)$ be a Shimura datum of Hodge type and $\calG$ a parahoric for $G$. Assume that $p\nmid|\pi_1(G^{\der})|$ and that the centralizer of a maximal $\brQ$-split torus in $G$ is $R$-smooth. 	Then $(\bfG,X,\calG)$ admits a good Hodge embedding.
\end{lemma}

\begin{proof}  Our assumptions imply that $\iota_{\bbQ_p}$ satisfies the conditions in Proposition \ref{prop: regular triples}. The construction there provides us with a good local symplectic Hodge embedding $\rho'$ which is easily seen to come from a global Hodge embedding.
\end{proof}
\subsubsection{}We set  $\rmK_p:=\calG(\bbZ_p)$, and we let  $\rmK:=\rmK_p\rmK^p$. Upon scaling, we may assume $V_{\bbZ_p}$ is contained in the dual lattice $V_{\bbZ_p}^\vee$. 
Let $V_{\Z_{(p)}}=V_{\Z_p}\cap V$. We write $G_{\Z_{(p)}}$ for the Zariski closure of $\bfG$ in $\GL(V_{\Z_{(p)}})$; then $G_{\Z_{(p)}}\otimes_{\Z_{(p)}}\Z_p\cong  \G$. Let $\rmK'=\rmK_p'\rmK'^p $ where $\rmK'_p$ is the stabilizer in $\GSp(V_{\bbQ_p})$ of the lattice $V_{\bbZ_p}$ and $\rmK'^p \subset \mathbf{GSp}(\A_f^p)$ is a compact open subgroup.
The choice of $V_{\Z_{(p)}}$ gives rise to an interpretation of $\Sh_{\rmK'}(\mathbf{GSp},S^\pm)$ as a moduli stack of abelian varieties up to prime-to-$p$ isogeny and hence an integral model $\mathscr{S}_{\rmK'}(\mathbf{GSp},S^\pm)$ over $\Z_{(p)}$, see \cite[\S 4]{KP} and \cite[\S6]{Z}.

Assume that  $\rmK^p$ is a neat compact open subgroup. By \cite[Lemma 2.1.2]{Ki2}, we can choose $\rmK'^p$ such that $\iota$ induces a closed immersion
$$\Sh_{{\rmK}}(\bfG,X)\hookrightarrow \Sh_{\rmK'}(\mathbf{GSp},S^\pm)\otimes_{\bbQ}\bfE.$$  Let $\mathscr{S}_{\rmK}(\bfG,X)^-$ be the Zariski closure of $\Sh_{\rmK}(\bfG,X)$ inside $\mathscr{S}_{\rmK'}(\mathbf{GSp},S^\pm)\otimes_{\Z_{(p)}}\Ok_{\bfE_{(v)}}$, and $\mathscr{S}_{\rmK}(\bfG,X)$  the normalization of $\mathscr{S}_{\rmK}(\bfG,X)^-$.  We also define the pro-scheme $$\scrS_{\rmK_p}(\bfG,X):=\lim_{\leftarrow\rmK^p}\scrS_{{\rmK}_p\rmK^p}(\bfG,X).$$
The $\bfG(\bbA_f^p)$-action on $\mathrm{Sh}_{{\rmK}_p}(\bfG,X)$ extends to $\scrS_{\rmK_p}(\bfG,X)$. Hence we may define $\scrS_{\rmK_p\rmK^p}(\bfG,X)$ for a general (not necessarily neat) compact open subgroup $\rmK^p\subset\bfG(\bbA_f)$ as the quotient stack $\scrS_{{\rmK_p}}(\bfG,X)/\rmK^p$. Alternatively, we may take a compact open subgroup $\rmK_1^p\subset \rmK^p$ which is neat and normal in $\rmK^p$, and define $\scrS_{\rmK}(\bfG,X)$ as the quotient of $\scrS_{{\rmK_p}\rmK^p_1}(\bfG,X)$ under the action of the finite group $\rmK^p/\rmK_1^p$.

\subsubsection{}\label{subsubsec:hodgecycles} In order to understand the local structure of $\scrS_{\rmK}(\bfG,X)$, we introduce Hodge cycles. By \cite[Proposition 1.3.2]{Ki2}, the subgroup $G_{\Z_{(p)}}$ is the stabilizer of a collection of tensors $s_\alpha\in V_{\Z_{(p)}}^\otimes$. Let $h:\mathcal{A}\rightarrow \mathscr{S}_{{\rmK}}(\bfG,X)$ denote the pullback of the universal abelian scheme on $\mathscr{S}_{\rmK'}(\mathbf{GSp},S^\pm)$ and let $V_{\mathrm{B}}:=R^1h_{\mathrm{an},*}\Z_{(p)}$, where $h_{\mathrm{an}}$ is the map of complex analytic spaces associated to $h$.  Since the tensors $s_\alpha$ are $\bfG$-invariant, they give rise to sections $s_{\alpha,\rmB}\in V_{\rmB}^\otimes$. We also let $\mathcal{V}=R^1h_*\Omega^\bullet$ be the relative de Rham cohomology of $\mathcal{A}$. Using the de Rham isomorphism, the $s_{\alpha,B}$ give rise to a collection of Hodge cycles $s_{\alpha,\mathrm{dR}}\in \mathcal{V}_\C^\otimes$, where $\mathcal{V}_\C$ is the complex analytic vector bundle associated to $\mathcal{V}$. By \cite[Corollary 2.2.2]{Ki2}, these tensors are defined over $\bfE$.

Similarly for a finite prime $\ell\neq p$, we let $\mathcal{V}_\ell = \mathcal{V}_\ell(\calA) =R^1h_{\mathrm{\acute{e}t}*}\Q_\ell$ and $\mathcal{V}_p = \mathcal{V}_p(\calA) = R^1h_{\eta,\mathrm{\acute{e}t}*}\Z_p$ where $h_\eta$ is the generic fiber of $h$. Using the \'etale-Betti comparison isomorphism, we obtain tensors $s_{\alpha,\ell}\in \mathcal{V}^\otimes_\ell$ and $s_{\alpha,p}\in\mathcal{V}_p^\otimes$.

For  $T$ an $\Ok_{\bfE_{(v)}}$-scheme and $x\in \mathscr{S}_{{\rmK}}(\bfG,X)(T)$, we write $\calA_x$ for the pullback of $\calA$ to $x$, and  for $*=\ell$ or $\mathrm{dR}$, we write $s_{\alpha,*,x}$ for the pullback of $s_{\alpha,*}$ to $x$. Similarly, for $T$ an $\mathbf{E}$-scheme (resp. $\mathbb{C}$-scheme) and $x\in \mathscr{S}_{{\rmK}}(\bfG,X)(T)$, we write $s_{\alpha,p,x}$ (resp. $s_{\alpha,B,x}$) for the pullback of $s_{\alpha,p}$ (resp. $s_{\alpha,B}$) to $x$.

For $T$ an $\Ok_{\bfE_{(v)}}$-scheme, an element $x\in\mathscr{S}_{{\rmK}}(\bfG,X)(T)$  corresponds to a triple $(\mathcal{A}_x,\lambda,\epsilon_{\rmK'}^p)$, where $\lambda$ is a weak polarization (cf. \cite[\S6.3]{Z}) and  $\epsilon^p_{\rmK'}$ is a section of the \'etale sheaf $\underline{\mathrm{Isom}}_{\lambda,\psi}(\widehat{V}(\mathcal{A}_x),V_{\A_f^p})/\rmK'^p$; here $$\widehat{V}(\calA_x) =\varprojlim_{p\nmid n}\calA_x[n] $$ 
is the  adelic prime-to-$p$ Tate module of $\calA_x.$  As in \cite[\S3.4.2]{Ki2}, $\epsilon^p_{\rmK'}$ can be promoted to a section
$$\epsilon_{\rmK}^p\in\Gamma(T,\underline{\mathrm{Isom}}_{\lambda,\psi}(\widehat{V}(\mathcal{A}_x),V_{\A_f^p})/\rmK^p)$$ which takes $s_{\alpha,\ell,x}$ to $s_{\alpha}$ for $\ell\neq p.$

\subsubsection{}\label{subsec: integral models formal nbd} Recall that $k$ is an algebraic closure of $k_E$ and $\brQ=W(k)[1/p]$. Let $\overline{x}\in\mathscr{S}_{\rmK}(\bfG,X)(k)$ and $\widetilde{x}\in\mathscr{S}_{\rmK}(\bfG,X)(\Ok_K)$ a point lifting $\overline{x}$, where $K/\brQ$ is a finite extension. 

Let $\pdiv_{\widetilde{x}}$ denote the $p$-divisible group associated to $\mathcal{A}_{\widetilde{x}}$ and $\pdiv_{\xbar}$ its special fiber; we let $\bbD:=\bbD(\pdiv_{\xbar})(\brZ)$. Then $T_p\pdiv_{\widetilde{x}}^\vee$ is identified with $\rmH^1_{\mathrm{\acute{e}t}}(\mathcal{A}_{\widetilde{x},\overline{K}},\Z_p)$ and we obtain $\mathrm{Gal}(\overline{K}/K)$-invariant tensors $s_{\alpha,p,\widetilde{x}}\in T_p\pdiv^{\vee\otimes}_{\widetilde x}$ whose stabilizer can be identified with ${\G}$. Let $s_{\alpha,0,\widetilde{x}}\in\D[\frac{1}{p}]^\otimes$ denote the tensors corresponding to $s_{\alpha,p,\widetilde{x}}$ via the $p$-adic comparison isomorphism. By \cite[Proposition 1.3.7]{KMS}, $s_{\alpha,0,\widetilde{x}}$ are independent of the choice of lifting $\widetilde{x}\in\scrS_{\rmK}(\bfG,X)(\calO_K)$. We may therefore denote them by $s_{\alpha,0,\overline{x}}$.

By Proposition \ref{prop: trivialization of Dieudonne}, we have $s_{\alpha,0,\overline{x}}\in\D^\otimes$ and there is a $\brZ$-linear bijection \begin{equation}\label{eqn: trivialization Dieudonne rational}V^\vee_{\Z_p}\otimes_{\bbZ_p}\brZ\cong T_p\scrG_{\widetilde{x}}^\vee\otimes_{\bbZ_p}\brZ\cong\bbD\otimes_{\brZ}\brZ\end{equation}  taking $s_{\alpha}$ to $s_{\alpha,0,\overline{x}}$. The filtration on $\D\otimes_{\brZ}K$ corresponding to $\pdiv_{\widetilde{x}}$ is induced by a $G$-valued cocharacter conjugate to $\mu_h^{-1}$.
By a result of Blasius and Wintenberger \cite{Bl}, $s_{\alpha,\mathrm{dR},\widetilde{x}}\in \widetilde{x}^*(\calV)^\otimes
\cong\bbD(\scrG_{\widetilde{x}})(\calO_K)^\otimes$ corresponds to $s_{\alpha,p,\widetilde{x}}$ via the $p$-adic comparison isomorphism. Hence $s_{\alpha,\mathrm{dR},\widetilde{x}}$ may be identified with the image of the elements  $\widetilde{s}_\alpha\in \bbD(\scrG_{\widetilde{x}})(\widehat{W}(
\calO_K))^\otimes$ of 
Proposition \ref{prop: trivialization of Dieudonne} inside $\bbD(\scrG_{\widetilde{x}})(
\calO_K)^\otimes$. The same Proposition implies that there is an $\calO_K$-linear bijection
$$\D(\pdiv_{\widetilde{x}})(\Ok_K)\cong\D\otimes_{\brZ}\Ok_K$$ taking $s_{\alpha,\mathrm{dR},\widetilde{x}}$ to $s_{\alpha,0,\overline{x}}$ and which lifts the identity over $k$. 
Thus there is a $G$-valued cocharacter $\mu_y$ which is $G$-conjugate to $\mu_h^{-1}$ and which induces a filtration on $\D\otimes_{\brZ}\calO_K$ lifting the filtration on $\D\otimes_{\brZ}k$. We may therefore define the notion of  $({\G},\mu_y)$-adapted liftings as in \S\ref{sec: deformation theory} and it follows from Proposition \ref{prop: trivialization of Dieudonne} that $\pdiv_{\widetilde{x}}$ is a $({\G},\mu_y)$-adapted lifting.

\subsubsection{}\label{subsubsec: formal nbd SV}Note that $G\subset \GL(V_{\bbQ_p})$ contains the scalars.  It follows that under our assumptions, conditions (A) and (B) of \S\ref{sssec: very good} are satisfied. We let $P\subset \GL(\D)$ be a parabolic lifting $P_0$ as in \S\ref{sec: adapted liftings}. We obtain formal local models $\widehat{\bbM}^{\mathrm{loc}}=\mbox{Spf}A$ and $\widehat{\bbM}^{\mathrm{loc}}_{\G}=\mbox{Spf}A_{{\G}}\cong \widehat{\bbM}^{\mathrm{loc}}_{\G,\{\mu_h\}}$, and the filtration corresponding to $\mu_y$ is given by a point $y:A_{{\G}}\rightarrow \Ok_K$. Let $\overline{y}\in \Mloc_{\calG,\{\mu_h\}}(k)$ correspond to the closed point of $\widehat{\bbM}^{\mathrm{loc}}_{\calG,\{\mu_h\}}(k)$.

\begin{prop}\label{prop: formal nbd Shimura} Assume $\rmK^p$ is neat and that the embedding ${\calG}\rightarrow \GL(V_{\bbZ_p})$ is very good at $\overline{y}\in \Mloc_{\calG,\{\mu_h\}}(k)$. Let $\widehat{U}_{\overline{x}}$ be the completion of $\mathscr{S}_{\rmK}(\bfG,X)^-$ at  the image of $\overline{x}$. 
\begin{enumerate}	
\item $\widehat{U}_{\overline{x}}$ can be identified with a closed subspace of $\mathrm{Spf}A\otimes_{\brZ}\calO_{\breve E}$ containing $\mathrm{Spf}A_{{\G}}$. 

\item  A deformation $\pdiv$ of $\pdiv_{\xbar}$ corresponds to a point on the  irreducible component of $\widehat{U}_{\overline{x}}$ containing $\widetilde{x}$  if and only if $\pdiv$ is $({\G},\mu_y)$-adapted.
\item Let $\overline{x}'\in\mathscr{S}_{\rmK}(\bfG,X)(k)$ whose image in $\mathscr{S}_{\rmK}(\bfG,X)^-(k)$ coincides with that of $\overline{x}$. Then $s_{\alpha,0,\overline{x}'}=s_{\alpha,0,\overline{x}}\in\D^\otimes$ if and only if $\overline{x}=\overline{x}'$.

\end{enumerate}
\end{prop}
\begin{proof}  Since  the conditions (A)--(C) of \S\ref{sssec: very good} are satisfied, we may apply the construction of Proposition \ref{prop: versal deformation space tensors}; this allows us to view $\mathrm{Spf}A$ as a versal deformation space for $\scrG_{\overline{x}}$ and hence we obtain a map $\Theta:\widehat{U}_{\overline{x}}\rightarrow \mathrm{Spf}A\otimes_{\brZ}\calO_{\brE}$ such that the universal $p$-divisible group over $\mathrm{Spf}A\otimes_{\brZ}\calO_{\brE}$ pulls back to the one over $\widehat{U}_{\overline{x}}$ arising from the universal abelian scheme over $\widehat U_{\overline{x}}$.  The map $\Theta$ is a closed immersion by the Serre--Tate theorem.

Let $Z\subset \widehat{U}_{\overline{x}}$ denote the irreducible component of $\widehat U_{\overline{x}}$ containing $\widetilde{x}$. 	Let $K'$ be  a finite extension of $\breve{E}$ and let $\widetilde{x}'\in Z(K')$. Then the tensors $s_{\alpha,p,\widetilde{x}'}$ correspond to $s_{\alpha,0,\overline{x}}$ under the $p$-adic comparison isomorphism. Moreover the filtration on $\bbD\otimes_{\brZ}K'$ corresponding to $\scrG_{\widetilde{x}'}$ is induced by a $G$-valued cocharacter conjugate to $\mu^{-1}_h$, and hence conjugate to $\mu_y$. By Proposition \ref{prop: G-adapted etale}, $\scrG_{\widetilde{x}'}$ is a $(\G,\mu_y)$-adapted deformation of $\scrG_{\overline{x}}$ and hence $\widetilde{x}'$ corresponds to a point of $\mathrm{Spf}A_{{\G}}$. Since this is true for any $\widetilde x'$, it follows that $\Theta|_{Z}$ factors through $\mathrm{Spf}A_{{\G}}$. Since $Z$ and $\mathrm{Spf}A_{{\G}}$ have the same dimension, it follows that $Z\cong \mathrm{Spf}A_{{\G}}$. We thus obtain (1) and (2).
	
	One direction of (3) is clear. For the other direction, let $\widetilde{x}'\in\scrS_{\rmK}(\bfG,X)(\calO_{K'})$ be a lift of $\overline{x}'$. Then by Proposition \ref{prop: trivialization of Dieudonne}, $s_{\alpha,0,\overline{x}'}$ 
	arises  from the specialization of tensors $\widetilde s_{\alpha} \in \bbD(\scrG_{\widetilde{x}'})(\widehat{W}(\Ok_K))$. By assumption, we have  $s_{\alpha,0,\overline{x}'}=s_{\alpha,0,\overline{x}}$. It follows that $\scrG_{\widetilde{x}'}$ corresponds to a $(\G,\mu_y)$-adapted lifting and hence to a point of $\mathrm{Spf}A_{{\G}}$. By what we have seen, $\widetilde{x}'$ corresponds to a point in the same irreducible component $Z\subset \widehat U_{\overline{x}}$ containing $\widetilde{x}$ and hence $\overline{x}=\overline{x}'$.
\end{proof}

\subsubsection{} \label{sec: mu ordinary locus HT}
We use the above to construct the analogue of the Serre--Tate canonical lift in this setting. Let $\xbar\in \mathscr{S}_{\rmK}(\bfG,X)(k)$ and we fix an isomorphism \begin{equation}\label{eqn: triv isocrystal}V^\vee_{\Z_p}\otimes_{\bbZ_p}\brQ\cong\D\otimes_{\brZ}\brQ,\end{equation} taking $s_{\alpha}$ to $ s_{\alpha,0,\overline{x}}$. Then the Frobenius on $\D\otimes_{\brZ}\brQ$ is given by $b\sigma$ for some $b\in G(\brQ)$.  By \cite[Lemma 1.3.9]{KMS}, we have $[b]\in B(G,\{\mu_h^{-1}\})$. We write $\calS_{\rmK}$ for the special fiber of  $\scrS_{\rmK}(\bfG,X)$ over the residue field $k_{E}$. The map $\calS_{{\rmK}}(k)\rightarrow B(G,\{\mu_h^{-1}\})$ sending $\overline{x}$ to the $\sigma$-conjugacy class $[b]$ of the associated element $b$  induces the Newton stratification of $\calS_{{\rmK},k}:=\calS_{\rmK}\otimes_{k_{E}}k$. For an element $[b]\in B(G,\{\mu_h^{-1}\})$, we write $\calS_{{\rmK},[b]}\subset\calS_{{\rmK},k}$ for the  corresponding stratum; if $\rmK^p$ is neat, it is a locally closed subscheme of  $\calS_{\rmK,k}$.  If there is a class $[b]_{\mu}\in B(G,\{\mu_h^{-1}\})$   as in Definition \ref{def: mu ordinary}, then we define the \emph{$\mu$-ordinary locus of} $\calS_{\rmK,k}$ to be $\calS_{\rmK,[b]_{\mu}}$.

For $\xbar\in\scrS_{\rmK}(\bfG,X)(k)$, define $\text{Aut}_\Q(\mathcal{A}_{\xbar})$ to be the $\Q$-group whose points in a $\Q$-algebra $R$ are given by
$$\mathrm{Aut}_{\Q}(\mathcal{A}_{\xbar})(R)=(\mathrm{End}(\mathcal{A}_{\xbar})\otimes_{\bbZ} R)^\times$$ 
By  functoriality, $\mathrm{Aut}_{\Q}(\mathcal{A}_{\xbar})$ acts on $T_\ell\mathcal{A}_{\xbar}\otimes_{\bbZ_\ell}\bbQ_\ell$ for $\ell\neq p$ and on $\D\otimes_{\brZ}\brQ$, and we write $I_{\xbar}$ for the closed subgroup of $\mathrm{Aut}_{\Q}(\mathcal{A}_{\xbar})$ consisting of automorphisms which preserve $s_{\alpha,\ell,\xbar}$ and $s_{\alpha,0,\xbar}$. There is a canonical inclusion $I_{\xbar}\otimes_\bbQ\Q_p\subset J_b$, where $J_b$ is the $\sigma$-centralizer group for  $b\in G(\brQ)$.

\begin{thm}\label{thm: canonical lift HT}
Let $\xbar\in \calS_{\rmK,[b]_\mu}(k)$ and assume $\calG\rightarrow \GL(V_{\bbZ_p})$ is very good at $\overline{y}\in \Mloc_{\calG,\{\mu_h\}}(k)$. Then $\xbar$ admits a lifting to a special point $\widetilde{x}\in\mathscr{S}_{\rmK}(\bfG,X)(K)$ for some $K/\brQ$ finite such that the action of $I_{\xbar}(\Q)$ on $\mathcal{A}_{\xbar}$ lifts to an action (in the isogeny category) on $\mathcal{A}_{\widetilde{x}}$.
\end{thm}
\begin{remark}

Recall that  $x\in \Sh_{\rmK}(\bfG,X)(\bbC)$ is said to be \emph{special} if there exists a torus $\bfT\subset \bfG$ such that under the identification $$\Sh_{\rmK}(\bfG,X)(\bbC)\cong \bfG(\bbQ)\backslash X\times \bfG(\bbA_f)/\rmK,$$ the point $x$ corresponds to an element $(h,g)\in \bfG(\bbQ)\backslash X\times \bfG(\bbA_f)/\rmK$, with $h(\bbC^\times)\subset \bfT(\bbR)$. More generally, if $K$ is a field of characteristic 0 which contains $\bfE$ and $x\in \Sh_{\rmK}(\bfG,X)(K)$, we say $x$ is a special point if for some (equivalently any) $\bfE_2$-algebra embedding $K\hookrightarrow \bbC$, the induced $\bbC$-point of $\Sh_{\rmK}(\bfG,X)$ is a special point. 
\end{remark}
\begin{proof}	Since the definition of $I_{\overline{x}}$  is independent of the prime-to-$p$ level, it suffices to consider the case of neat $\rmK^p$.
	Applying the construction in  \S\ref{sec: M-adapted}, we obtain a parahoric model $\mathcal{M}$ of a Levi subgroup $M \subset G,$ 
	and an $M$-valued cocharacter $\widetilde\lambda$ lying in the $G$-conjugacy class of $\mu_{h}$ and such that $\widetilde{\lambda}$ is central in $M$. Let $\pdiv$ be the $(\mathcal{M},\widetilde{\lambda})$-adapted deformation to $\Ok_K$ constructed in Theorem \ref{thm: can-lift}. By  Proposition \ref{prop: formal nbd Shimura}, $\pdiv$ corresponds to a point $\widetilde{x}\in \scrS_{\rmK}(\bfG,X)(\calO_K)$ lifting $\overline{x}$ and hence to an abelian variety $\mathcal{A}_{\widetilde{x}}$ over $K$.	By Theorem \ref{thm: can-lift}, the action of $J_b(\Q_p)$ on $\pdiv_{\xbar}$ lifts to $\pdiv$. Since $I_{\xbar}(\Q)\subset J_b(\Q_p)$, by the Serre--Tate theorem, the action of $I_{\xbar}$ lifts to $\mathcal{A}_{\widetilde{x}}$ in the isogeny category.
	
	We now show $\widetilde{x}$ is a special point. Since $I_{\xbar}$ fixes the tensors $s_{\alpha,0,\xbar}$, it also fixes $s_{\alpha,p,\widetilde{x}}$, and hence it fixes $s_{\alpha,B}$. Thus we may consider $I_\xbar$ as a   subgroup of $\mathbf{G}$. By \cite[Theorem 6]{KMS}, the absolute rank of $I_{\xbar}$ is equal to the absolute rank of $\bfG$. Let $\bfT$ be a maximal torus of $I_{\xbar}$, which is therefore a maximal torus of $\bfG$. The Mumford--Tate group of $\mathcal{A}_{\widetilde{x}}$ is a subgroup of $\bfG$ which commutes with $\bfT$ hence must be contained in $\bfT$. Therefore $\widetilde{x}$ is a special point.
	\end{proof}
\begin{remark}
For $\xbar\in \calS_{\rmK,[b]_\mu}(k)$, the corresponding $\overline{y}$ lies in the stratum of  $\Mloc_{\calG,\{\mu_h\}}$ corresponding to $t_{\mu'}\in \Adm(\{\mu\})_J$, with $\mu'\in W_0\cdot\mu$ as in \S\ref{subsec: mu ordinary setting}. It is possible to show that this stratum lies inside the smooth locus of $\Mloc_{\calG,\{\mu_h\}}\otimes_{\calO_E}k$ and hence the embedding $\calG\rightarrow \GL(V_{\bbZ_p})$ is automatically very good at $\overline{y}$ by \cite[Corollary 4.3.9]{KPZ}.
\end{remark}

\subsubsection{}\label{subsubsec: Assumption C'} The description of the local structure of $\scrS_{\rmK}(\bfG,X)$ in Theorem \ref{thm: local model diagram} can be globalized. In addition to assumptions (A') and (B') of \S\ref{subsub: integral model Hodge type construction}, we will need the following assumption, cf. \S\ref{sssec: very good}.
\begin{enumerate}
	\item[(C')] The embedding $\calG\rightarrow \GL(V_{\bbZ_p})$ is very good at \emph{all} points of $\Mloc_{\calG,\{\mu_h\}}(k)$.
\end{enumerate}

\begin{thm}\label{thm: local model diagram}
	Under the assumptions (A')--(C'), the schemes  $\scrS_{\rmK}(\bfG,X)$  satisfy the following properties.
	\begin{enumerate}\item For $R$ a discrete valuation ring of mixed characteristic $(0,p)$, we have a bijection 
		$$
		\varprojlim_{\rmK^p}\scrS_{\rmK_p\rmK^p}(\bfG,X)(R)=\Sh_{\rmK_p}(\bfG,X)(R[1/p]).
		$$
		
		\item There exists a local model  diagram 
	\[\xymatrix{ &\widetilde{\mathscr{S}}_{\rmK}(\bfG,X)_{\Ok_E}\ar[dr]^q\ar[dl]_\pi&\\
	\mathscr{S}_{\rmK}(\bfG,X)_{\Ok_E} & &\bbM^{\mathrm{loc}}_{\G,\{\mu_h\}}}\]
		where $\pi$ is a $\calG$-torsor and $q$ is  $\calG$-equivariant and smooth of relative dimension $\dim G$.

	\end{enumerate}
\end{thm}
\begin{proof} This follows from \cite[Theorem 7.1.3]{KPZ} which proves the result for neat level structure $\rmK^p$. In general, we take a normal neat compact open subgroup $\rmK_1^p\subset \rmK^p$, and take the quotient of the diagram by the finite group $\rmK^p/\rmK_1^p$.
	\end{proof}

\begin{remark}\begin{enumerate}
		\item By \cite[Theorem 4.5.2]{PRshtukas} and our assumption that $\tG=\calG$, the integral model $\scrS_{{\rmK}}(\bfG,X)$ is independent of the choice of Hodge embedding $\iota$.
		
		\item Formally,  Theorem \ref{thm: local model diagram} and Theorem  \ref{thm: canonical lift HT} are all that are needed to prove our main results on $\ell$-independence in \S5 and \S6. In the next three subsections, we will extend these theorems to the required generality needed for these applications.
	\end{enumerate}
\end{remark}

	\subsection{Strongly acceptable Shimura varieties}\label{sec: integral models abelian type}
For later applications, we need to consider integral models for certain Shimura varieties  of Hodge type with the conditions (A')--(C') of \S\ref{subsub: integral model Hodge type construction} and \S\ref{subsubsec: Assumption C'} relaxed. To do this we will view the Shimura variety as one of abelian type and we may construct an integral model using an auxiliary Shimura variety of Hodge type as in \cite[\S7.2]{KPZ}. 

\subsubsection{}

Let $(\bfG,X)$ be a Shimura datum of Hodge type. Then the center $\bfZ_{\bfG}$ of $\bfG$ splits over a CM field, and so  the maximal compact subtorus $\bfZ_{\bfG,0}$ is defined over $\bbQ$.  We let $\bfZ_{\bfG}^c$ denote the subgroup of $\bfG$ generated by $\bfZ_{\bfG,0}$ and the center $\bfZ_{\bfG^{\der}}$ of the derived group $\bfG^{\der}$. We let $Z_G^c$ denote the base change of this group to $\bbQ_p$. We will now focus on Shimura data satisfying the following property.

\begin{definition}\label{def: strongly acceptable}
Let $(\bfG_2,X_2)$ be a Shimura datum and set $G_2:=\bfG_{2,\bbQ_p}$.	Then $(\bfG_2,X_2)$ is said to be \emph{strongly acceptable} if the following conditions are satisfied:
	
	\begin{itemize}
		\item $(\bfG_2,X_2)$ is of Hodge type.
		\item$G_2^{\der}\cong \prod_{i=1}^r\Res_{F_i/\bbQ_p}H_i$, where $F_i/\bbQ_p$ is finite and $H_i$ is a \emph{split} reductive group over $F_i$.
		\item  $Z^c_{G_2}$ is a product of induced tori.
	\end{itemize}

	If $\calG_2$ is a  parahoric group scheme for $G_2$, we say the triple $(\bfG_2,X_2,\calG_2)$ is \emph{strongly acceptable} if $(\bfG_2,X_2)$ is strongly acceptable and $\calG_2$ is a very special parahoric (recall that a parahoric $\calG_2$ is very special if $\calG_2(\brZ)$ is a special parahoric of $G_2(\brQ)$, which exists by  \cite[Lemma 6.1]{Zhu2}).
\end{definition}
\begin{prop}\label{lemma: auxiliary Hodge type datum}
	
	Let $(\bfG_2,X_2,\calG_2)$ be a strongly acceptable triple. Then there exists  a  Shimura datum $(\bfG,X)$ together with a central isogeny  $\bfG^{\der}\rightarrow \bfG_2^{\der}$ which induces an isomorphism $(\bfG^{\ad},X^{\ad})\cong (\bfG_2^{\mathrm{ad}},X_2^{\mathrm{ad}})$. Moreover, $(\bfG,X)$ may be chosen to satisfy the following properties.
	
	\begin{enumerate}

		\item $\pi_1(G^{\der})$ is a 2-group and is trivial if $(\bfG_2^{\mathrm{ad}},X_2^{\mathrm{ad}})$ has no factors of type $D^{\bbH}$. 
		
		\item Any prime $v_2|p$ of $\bfE_2$  splits in the composite $\bfE':=\bfE.\bfE_2$.
		
		\item There exists a Hodge embedding $\iota:(\bfG,X)\rightarrow (\mathbf{GSp}(V),S^\pm)$ which extends to a good integral local Hodge embedding $\calG\rightarrow \GL(V_{\bbZ_p})$, which is very good at all points of $\Mloc_{\calG,\{\mu_h\}}(k)$. Here $\calG$ is the parahoric group scheme for $G$ which is associated to $\calG_2$.

		\item $Z_G$ is an $R$-smooth torus and $Z_G^c$ is a product of Weil restrictions of tame tori.
		
		\item $X_*(G^{\ab})_I$ is torsion.

	\end{enumerate}

\end{prop}
\begin{proof} We let $(\bfG,X)$ be the Shimura datum constructed in  \cite[Proposition 7.2.14]{KPZ} which is equipped with a central isogeny  $\bfG^{\der}\rightarrow \bfG_2^{\der}$ inducing an isomorphism $(\bfG^{\ad},X^{\ad})\cong (\bfG_2^{\mathrm{ad}},X_2^{\mathrm{ad}})$. Then $(\bfG,X)$  satisfies (1), (2), (4), and property (5) follows since  our assumptions imply that $G_2^{\ad}$  does not have a simple factor of the form $\Res_{F/\bbQ_p}\mathrm{PGL}_m(D)$, where $D$ is a division algebra over $F$ with index divisible by $p$. 
	
	Note that if $\bfG_2^{\ad}\cong \prod_{i=1}^s\Res_{\rmF_i/\bbQ}\bfH_i$ for some $\rmF_i/\bbQ$ totally real and $\bfH_i$ absolutely simple over $\rmF_i$, then $\bfG^{\der}\cong \prod_{i=1}^s\Res_{\rmF_i/\bbQ}\bfH_i^{\sharp}$, where $\bfH_i^{\sharp}$ is defined in \cite[\S7.2.1]{KPZ}. In particular $G$ satisfies the first assumption in Lemma \ref{lem: v special connected} below. Thus by that lemma, we have $\tG_x=\calG_x$ for any $x\in \calB(G,\bbQ_p)$ lifting the image of $x_2$ in $\calB(G_2^{\ad},\bbQ_p)=\calB(G^{\ad},\bbQ_p)$. Property (3) then follows from the corresponding property in \cite[Proposition 7.2.18]{KPZ} using this fact.
\end{proof}
\begin{lemma}\label{lem: v special connected}
	Let $G$ be a reductive group over $\bbQ_p$ and $\calG$ a parahoric of $G$ corresponding to $x\in \calB(G,\bbQ_p)$. We assume the following conditions are satisfied.
	\begin{itemize}
		\item $G^{\der}\cong \prod_{i=1}^r\Res_{ F_i/\bbQ_p}H_i$ where $H_i$ is a split group over $F_i$.
		\item $X_*(G^{\ab})_I$ is torsion-free
		\item $\calG$ is a very special parahoric. 
	\end{itemize} Then we have $\tG=\calG$.
\end{lemma}
\begin{proof}
	Let $x^{\ad}\in \calB(G^{\ad},\bbQ_p)$ denote the image of $x$ and $\tG^{\ad}$ (resp. $\G^{\ad}$) the associated stabillizer scheme (resp. parahoric group scheme). Then $\calG^{\ad}$ is of the form $\prod_{i=1}^r\Res_{\calO_{F_i}/\bbZ_p}\calH^{\ad}_i$ for $\calH^{\ad}_i$ a special (equivalently hyperspecial) parahoric of $H_i^{\ad}$, and hence we  have $\tG^{\ad}=\G^{\ad}$. 
	
	There is a natural map $\tG\rightarrow \tG^{\ad}=\G^{\ad}$ and a commutative diagram 
	\[\xymatrix{\tG(\brZ) \ar[r]\ar[d]_{\widetilde\kappa_{G}} & \G^{\ad}(\brZ)\ar[d]^{\widetilde\kappa_{G^{\ad}}}\\ \pi_1(G)_I \ar[r]&\pi_1(G^{\ad})_I.}\]
	It follows that $\tG(\brZ)$ maps to $\ker(\pi_1(G)_I\rightarrow \pi_1(G^{\ad})_I)$ and it suffices to show this group is torsion-free.

	We have a commutative diagram with exact rows.
	\[\xymatrix{ &\pi_1(G^{\der})_I\ar[r] \ar[d]&\pi_1(G)_I\ar[r]\ar[d]& X_*(G^{\ab})_I\ar[r]\ar[d]& 0\\
		0 \ar[r]& \pi_1(G^{\ad})_I \ar[r]^\sim&\pi_1(G^{\ad})_I\ar[r] &\{1\}\ar[r]& 0	}\]
	Since $\pi_1(G^{\der})\rightarrow \pi_1(G^{\ad})$ is injective and these are induced modules, it follows that $\pi_1(G^{\der})_I\rightarrow \pi_1(G^{\ad})_I$ is injective.
	Thus $\mathrm{ker}(\pi_1(G)_I\rightarrow \pi_1(G^{\ad})_I)$ is torsion-free by the snake Lemma.  	
\end{proof}

\subsubsection{}We use the previous proposition to extend the construction of integral models to strongly acceptable triples.
\begin{thm}\label{cor: strongly acceptable integral model}Let $(\bfG_2,X_2,\calG_2)$ be a strongly acceptable triple with reflex field $\bfE_2$ and set $\rmK_{2,p}=\calG_2(\bbZ_p)$.  Then  for any prime $v_2|p$ of $\bfE_2$ with corresponding completion $E_2$,  there is a $\bfG_2(\bbA_f^p)$-equivariant $\calO_{E_2}$-scheme $\mathscr{S}_{\rmK_{2,p}}(\bfG_2,X_2)$ extending $\Sh_{\rmK_{2,p}}(\bfG_2,X_2)_{{E_2}}$ satisfying the following properties:
	\begin{enumerate}
		\item $\mathscr{S}_{\rmK_{2,p}}(\bfG_2,X_2)$ is \'etale locally isomorphic to $\bbM^{\mathrm{loc}}_{\calG_2,\{\mu_{h_2}\}}$.
		\item For any discrete valuation ring  $R$ of mixed characteristic the map 
		$$\mathscr{S}_{\rmK_{2,p}}(\bfG_2,X_2)(R)\rightarrow\mathscr{S}_{\rmK_{2,p}}(\bfG_2,X)(R[1/p])$$ 
		is a bijection.
		
		\item There exists a diagram 
		\begin{equation}\label{eqn: local model diagram strongly acceptable}\xymatrix{ &\widetilde{\mathscr{S}}^{\mathrm{ad}}_{\rmK_{2,p}}(\bfG_2,X_2)\ar[dr]^q\ar[dl]_\pi&\\
			\mathscr{S}_{\rmK_{2,p}}(\bfG_2,X_2) & &\bbM^{\mathrm{loc}}_{\G_2,\{\mu_{h_2}\}}}\end{equation}
		where $\pi$ is a $\bfG_2(\A_f^p)$-equivariant ${\calG}_2^{\ad}$-torsor and $q$ is $\calG_{2}^{\ad}$-equivariant, smooth of relative dimension $\dim \bfG^{\ad},$ and 
		$\bfG_2(\A_f^p)$-equivariant, when $\bbM^{\mathrm{loc}}_{\G_2,\{\mu_{h_2}\}}$ is equipped with the trivial $\bfG_2(\A_f^p)$-action. Here $\calG_2^{\ad}$ is the parahoric group scheme for $G_2^{\ad}$ associated to $\calG_2$.
	\end{enumerate}
\end{thm}

\begin{proof} Let $(\bfG,X)$ be the Shimura datum of Hodge type from Proposition \ref{lemma: auxiliary Hodge type datum} which satisfies the assumptions in \cite[Proposition 7.1.14]{KPZ}. The result then follows from  \emph{loc. cit.} noting that we have $\tG=\calG$.
\end{proof}
\begin{remark}\begin{enumerate}
		\item
	The condition that $Z_{G_2}^c$ is a product of induced tori in the definition of strongly acceptable datum is not needed for this theorem. It is used in the next subsection to prove certain functoriality properties for integral modes.
	\item A key property in Theorem \ref{cor: strongly acceptable integral model} that we need is that  $\widetilde{\mathscr{S}}^{\mathrm{ad}}_{\rmK_{2,p}}(\bfG_2,X_2)$ in (3) is a torsor for a smooth group scheme with connected special fiber and  is one of the reasons we restrict to considering strongly acceptable triples.
	\end{enumerate}
\end{remark}

\subsubsection{}\label{subsubsec: construction of integral model} We recall some features of the construction in Theorem \ref{cor: strongly acceptable integral model} which will be needed in the next subsection. We let $(\bfG,X)$ denote the auxiliary Hodge type Shimura datum from Proposition \ref{lemma: auxiliary Hodge type datum}. This is equipped with a central isogeny $\bfG^{\der}\rightarrow \bfG_2^{\der}$ inducing an isomorphism $(\bfG^{\ad},X^{\ad})\cong (\bfG_2^{\ad},X_2^{\ad})$. There is a Hodge embedding $(\bfG,X)\rightarrow (\mathbf{GSp}(V),S^\pm)$ satisfying the assumptions (A')--(C') of \S\ref{subsub: integral model Hodge type construction} and \S\ref{subsubsec: Assumption C'}, and hence we may construct an integral model $\scrS_{\rmK}(\bfG,X)$ for $\Sh_{\rmK}(\bfG,X)$ as before by taking closure and normalization inside the Siegel Shimura variety.

Fix  a connected component $X^+\subset X$. By real approximation, upon modifying the isomorphism $\bfG^{\ad}\cong \bfG_2^{\mathrm{ad}}$ by an element of $\bfG^{\ad}(\bbQ)$,  we may assume that the image of  $X_2\subset X_{2}^{\mathrm{ad}}$ contains the image of $X^+.$ We write $X_2^+$ for $X^+$ viewed as a subset of $X_2.$  
We denote by $\Sh_{\rmK_p}(\bfG, X)^+ \subset \Sh_{\rmK_p}(\bfG, X)$ and 
$\Sh_{\rmK_{2,p}}(\bfG_2, X_2)^+  \subset \Sh_{\rmK_{2,p}}(\bfG_2, X_2)$ the geometrically connected components corresponding 
to $X^+$ and $X_2^+$. These are defined over extensions  of $\bfE$ and $\bf E'$ respectively, which are unramified at primes above $p$ by \cite[Proposition 7.1.11]{KPZ}.
We let $\scrS_{\rmK_p}(\bfG,X)^+$ denote the connected component of $\scrS_{\rmK_p}(\bfG,X)$ corresponding to $\Sh_{\rmK_p}(\bfG, X)^+$.

  For a subgroup $H\subset \bfG(\bbR)$, we write $H_+$ for the preimage of $\bfG^{\ad}(\bbR)^+$, the connected component of the identity in $\bfG^{\ad}(\bbR)$.  We write $\bfG^{\ad}(\bbQ)^+$ (resp. $\bfG^{\ad}(\bbZ_{(p)})^+$)  for $\bfG^{\ad}(\bbQ)\cap \bfG^{\ad}(\bbR)^+$ (resp. $\bfG^{\ad}_{\bbZ_{(p)}}(\bbZ_{(p)})\cap \bfG^{\ad}(\bbR)^+$) and we write $\bfZ=\bfZ_{\bfG}$ for the center of $\bfG$. We let $\bfZ(\bbQ)^-$ and $\bfG(\bbQ)_+^-$ denote the closures of $\bfZ(\bbQ)$ and $\bfG(\bbQ)_+$  in $\bfG(\bbA_f)$, respectively. We let $\bfZ(\bbZ_{(p)})^-$ and $\bfG(\bbZ_{(p)})_+^-$ denote  the closures of $\bfZ_{\bbZ_{(p)}}(\bbZ_{(p)})$ and $\bfG_{\bbZ_{(p)}}(\bbZ_{(p)})_+$  in $\bfG(\bbA_f^p)$, respectively.  As in \cite[\S4.5.6]{KP}, we set
$$\scrA(\bfG):=\bfG(\bbA_f)/\bfZ(\bbQ)^-*_{\bfG(\bbQ)_+/\bfZ(\bbQ)}\bfG^{\ad}(
\bbQ)^+$$ $$\scrA(\bfG_{\Z_{(p)}}):=\bfG(\bbA_f^p)/\bfZ({\Z_{(p)})^-}*_{\bfG(\Z_{(p)})_+/\bfZ(\bbZ_{(p)})}\bfG^{\ad}(\Z_{(p)})^+,$$
and  as in \cite[\S4.6.3]{KP}, we set 
$$\scrA(\bfG)^\circ:=\bfG(\bbQ)_+^-/\bfZ(\bbQ)^-*_{\bfG(\bbQ)_+/\bfZ(\bbQ)}\bfG^{\ad}(\Q)^+$$
$$\scrA(\bfG_{\bbZ_{(p)}})^\circ:=\bfG(\bbZ_{(p)})_+^-/\bfZ(\bbZ_{(p)})^-*_{\bfG(\bbZ_{(p)})_+/\bfZ(\bbZ_{(p)})}\bfG^{\ad}(\bbZ_{(p)})^+.$$
We refer to \emph{loc. cit.} \S4.5.6 for the definition of the $*$ product.
We obtain an  $\scrA(\bfG)$-action (resp. $\scrA(\bfG_{\bbZ_{(p)}})$-action) on $\Sh(\bfG,X)$ (resp. $\Sh_{\rmK_p}(\bfG,X)$).
Here, 
the fact that the center of $G$ is an $R$-smooth torus implies that the $\scrA(\bfG_{\bbZ_{(p)}})$-action on $\Sh_{\rmK_p}(\bfG,X)$ extends to an $\scrA(\bfG_{\Z_{(p)}})$-action on $\scrS_{\rmK_p}(\bfG,X)$. As in \cite[\S4.6.12]{KP}, the natural map 
\begin{equation}\label{eqn: injection A groups}
\scrA(\bfG_{\bbZ_{(p)}})^\circ\backslash\scrA(\bfG_{2,\bbZ_{(p)}})\rightarrow \scrA(\bfG)^\circ\backslash\scrA(\bfG_2)/\rmK_{2,p}
\end{equation} 
is an injection. We fix a set $J\subset \bfG_2(\bbQ_p)$ which maps bijectively to a set of coset representatives for the image of $\scrA(\bfG_{2,\bbZ_{(p)}})$ in $\scrA(\bfG)^\circ\backslash\scrA(\bfG_2)/\rmK_{2,p}$. A calculation  shows that $J$ is a finite set. 
Then $\scrS_{\rmK_{2,p}}(\bfG_2,X_2)$ is constructed as 
\begin{equation}\label{eqn: abelian type construction}\scrS_{\rmK_{2,p}}(\bfG_2,X_2)=\left[[\scrS_{\rmK_p}(\bfG,X)^+\times\scrA(\bfG_{2,\bbZ_{(p)}})]/\scrA(\bfG_{\bbZ_{(p)}})^\circ\right]^{|J|}.
\end{equation}

\subsection{Some functorial properties of integral models}\label{sec: functorial properties of integral models}

In this subsection we prove some functorial properties of the integral models. The main result is Proposition \ref{prop: morphism integral models} which will be used to define the $\mu$-ordinary locus in the next subsection.

 \subsubsection{}Let $f:(\bfG,X)\rightarrow(\bfG',X')$ be a morphism of Shimura data and let  $\calG$ and $\calG'$ be parahorics of $G$ and $G'$ respectively.   We assume that $\tG=\G$ and $\tG'=\G'$, and that $G\rightarrow G'$ extends to a morphism $\calG\rightarrow \calG'$.
Let $\rmK = \rmK_p\rmK^p $, $\rmK' = \rmK'_p\rmK^{\prime p},$ be compact open subgroups of $\bfG(\mathbb A_f)$ 
and $\bfG'(\mathbb A_f)$ respectively, 
with $\rmK_p = \G(\Z_p)$ and $\rmK'_p = \G'(\Z_p)$. We fix a prime $v|p$ of the reflex field $\bfE$ of $(\bfG,X)$, 
and write $E = \bfE_v.$

We assume there are Hodge embeddings 
$$ \iota: (\bfG,X) \rightarrow (\mathbf{GSp}(V),S^\pm) \quad \text{\rm and} \quad \iota': (\bfG',X') \rightarrow (\mathbf{GSp}(V'),S'^\pm),$$
and $\Z_p$-lattices $V_{\Z_p} \subset V_{\Q_p}$ and $V'_{\Z_p} \subset V'_{\Q_p}$ such that $\iota$ and $\iota'$ 
extend to good integral local Hodge embeddings  $\calG\rightarrow \GL(V_{\bbZ_p})$, $\calG'\rightarrow \GL(V'_{\bbZ_p})$.
Thus $(\bfG,X)$ and $(\bfG',X')$ both satisfy assumptions (A') and (B') of \S\ref{subsub: integral model Hodge type construction}.

\begin{prop}\label{prop: morphism SV iso completions}  The morphism  $\Sh_{\rmK}(\bfG,X)\rightarrow \Sh_{\rmK'}(\bfG',X')_{\bfE}$ induced by $f$ 
	extends to a morphism of integral models over $\calO_E$
	$$f_{\scrS}:\scrS_{\rmK}(\bfG,X)_{\calO_E}\rightarrow \scrS_{\rmK'}(\bfG',X')_{\calO_E},$$
	associated to $\iota$ and $\iota'$. 
	
	Moreover, if $f$ induces an isomorphism of derived groups $\bfG^{\der}\cong \bfG'^{\der}$ and the parahorics $\G$ and $\G'$ are associated, then for $\rmK$ and $\rmK'$ neat and  $\xbar\in \scrS_{\rmK}(\bfG,X)_{\calO_E}(k)$ with image $\xbar'\in \scrS_{\rmK'}(\bfG',X')_{\calO_E}(k)$, the morphism $f_{\scrS}$ induces an isomorphism of  completions $\widehat{U}_{\xbar}\cong \widehat{U}_{\xbar'}$ at $\xbar$ and $\xbar'$.	\end{prop}
\begin{proof}We set $(\bfG'',X'')=(\bfG\times \bfG',X\times X')$ and $\rmK''=\rmK\times \rmK'$. Then the product $$\scrS'':=\scrS_{\rmK}(\bfG,X)_{\calO_E}\times_{\calO_E}\scrS_{\rmK}(\bfG',X')_{\calO_E}$$ is an integral model for the Shimura variety $\Sh_{\rmK''}(\bfG'',X'')$ which satisfies the conditions in \cite[Conjecture 4.2.2]{PRshtukas}. Therefore there exists a unique map $\scrS_{\rmK}(\bfG,X)\rightarrow \scrS''$ extending the diagonal morphism on the generic fiber by \cite[Theorem 4.3.1]{PRshtukas}, and its composition with the projection $\scrS''\rightarrow \scrS_{\rmK'}(\bfG',X')_{\calO_E}$ gives the desired morphism $f_{\scrS}$.
	
	Now assume that $\bfG^{\der}\cong \bfG'^{\der}$ and that $\calG$ and $\calG'$ are associated. To show that $f_{\scrS}$ induces isomorphisms on completions, we follow the proof of \cite[Theorem 4.2.4]{PRshtukas}. We let $\scrP$ and $\scrP'$ denote the shtukas  over the $p$-adic completions $\widehat{\scrS}_{\rmK}(\bfG,X)$ and $\widehat{\scrS}_{\rmK'}(\bfG',X')$ constructed in \cite[Theorem 4.5.2]{PRshtukas}. Then by \cite[Theorem 2.7.7]{PRshtukas}, we have an isomorphism of $\calG'$-shtukas $$\scrP\times^{\calG}\calG'\cong f_{\scrS}^*\scrP'$$ over $\widehat{\scrS}_{\rmK}(\bfG,X)$, since they extend the same $\calG'$-shtuka over the generic fiber. Let $\xbar\in \scrS_{\rmK}(\bfG,X)(k)$ with image $\xbar'\in \scrS_{\rmK}(\bfG',X')(k)$.	 Since $\widehat{U}_{\xbar}$ and $\widehat{U}_{\xbar'}$ are normal, it suffices to show $f_{\scrS}$ induces an isomorphism of  $\widehat{U}_{\xbar}^\lozenge\cong\widehat{U}^{\lozenge}_{\xbar'}$ of the associated $v$-sheaves.  
	
	We set $\mu=\mu_h$ and $\mu'=\mu_{h'}$. Let $b_{\xbar}\in G(\brQ)$ the element corresponding to the $\calG$-shtuka $\scrP_{\xbar}$ at $\xbar$, which is well-defined up to $\sigma$-conjugacy by $\calG(\brZ)$. By construction, this is the element corresponding to Frobenius on $\bbD$ given by the choice of a tensor preserving isomorphism $V_{\bbZ_p}^\vee\otimes_{\bbZ_p}\brZ\cong \bbD$. Then $b_{\xbar'}=f(b_{\xbar})\in G'(\brQ)$ corresponds to the $\calG'$-shtuka $\scrP_{\xbar'}$. Let $\calM^{\mathrm{int}}_{\calG,b_{\xbar},\mu}$ be the integral model for the local Shimura variety associated to the local Shimura datum $(G,b_{\xbar},\mu)$ and the parahoric $\calG$, cf. \cite[Definition  25.1.1]{SW2}, and let $\widehat{\calM^{\mathrm{int}}_{\calG,b_{\xbar},\mu,
		\xbar_0}}$ denote the $v$-sheaf completion at the base point $\xbar_0$ . By \cite[Theorem 4.5.2]{PRshtukas}, there is an isomorphism:
	$$\Theta_{\xbar}: \widehat{\calM^{\mathrm{int}}_{\calG,b_{\xbar},\mu,
			\xbar_0}}\xrightarrow{\sim}\widehat{U}_{\xbar}^\lozenge$$ such that $\Theta_{\xbar}^*(\scrP)$ is isomorphic to the universal $\calG$-shtuka on $\widehat{\calM^{\mathrm{int}}_{\calG,b_{\xbar},\mu,
			\xbar_0}}$. There is a similar isomorphism $\Theta_{\xbar'}: \widehat{\calM^{\mathrm{int}}_{\calG',b_{\xbar'},\mu',
			\xbar'_0}}\xrightarrow{\sim}\widehat{U}_{\xbar'}^\lozenge$ for $(G',b_{\xbar'},\mu')$.
	
	For $r>\!>0$, we let $\widetilde{\scrS}_{\xbar}$ be the $v$-sheaf over $\widehat{U}_{\xbar}$, classifying trivializations of $\scrP$ as in \cite[Proof of Theorem 4.2.4]{PRshtukas}. Explicitly, for $S=\mathrm{Spa}(R,R^+)$ a perfectoid space over $k$ and $y$ an $S$-point of $\widehat{U}_{\xbar}^\lozenge$, $\widetilde{\scrS}_{\xbar}$ classifies trivializations
	$$i_r:\calG_{\calY_{[r,\infty)(S)}}\xrightarrow{\sim}y^*(\scrP)|_{\calY_{[r,\infty)(S)}},$$
	where $\calY_{[r,\infty)}$ is as in \cite[\S2.1]{PRshtukas}. Then  there is a natural map $$\mathrm{nat}:\widetilde{\scrS}_{\xbar}\rightarrow  \calM^{\mathrm{int}}_{\calG,b_{\xbar},\mu}.$$
		We define $\widetilde{\scrS}_{\xbar}'$ over $\widehat{U}_{\xbar'}$ similarly as trivializations of the $\calG'$-shtuka $\scrP'$. Then we have a commutative diagram
		\[\xymatrix{\widetilde{\scrS}_{\xbar}\ar[r]^{\!\!\!\!\!\!\!\mathrm{nat}}\ar[d] &\calM^{\mathrm{int}}_{\calG,b_{\xbar},\mu}\ar[d]\\
\widetilde{\scrS}_{\xbar'}\ar[r]^{\!\!\!\!\!\!\!\!\!\!\mathrm{nat}}&	\calM^{\mathrm{int}}_{\calG',b_{\xbar'},\mu'}	}\]
	where the vertical  maps are obtained via pushout along $\calG\rightarrow\calG'$.
	As in  the proof of
	\cite[Theorem 4.2.4]{PRshtukas}, upon modifying $\Theta_{\xbar}$ by an element of the group $G_{b_{\xbar}}(\bbQ_p)=\{g\in G(\brQ)|g^{-1}b_{\xbar}\sigma(g)=b_{\xbar}\}$, we have a commutative diagram
	\[\xymatrix{\widetilde{\scrS}_{\xbar}\ar[r]\ar[dr] \ar[d]_{\mathrm{nat}}&  \widehat{U}^\lozenge_{\xbar}\ar[d]\ar[d]^{\simeq} \\
 \calM^{\mathrm{int}}_{\calG,b_{\xbar},\mu} \ar[r] & \widehat{\calM^{\mathrm{int}}_{\calG,b_{\xbar},\mu,\xbar_0}},	}\]  and the right side of the diagram is determined by the left. Similarly, we obtain a diagram for $\widehat{\scrS}_{\xbar'}$ and $\widehat{U}_{\xbar'}^\lozenge$. It follows that the following diagram is commutative\[\xymatrix{ \widehat{U}^{\lozenge}_{\xbar}\ar[r]^{\!\!\!\!\!\!\!\!\!\!\Theta_{\xbar}} \ar[d] &  \widehat{\calM^{\mathrm{int}}_{\calG,b_{\xbar},\mu,
 		\xbar_0}}\ar[d]\\
 \widehat{U}^\lozenge_{\xbar'}\ar[r]^{\!\!\!\!\!\!\!\!\!\!\!\!\Theta_{\xbar'}}  &  \widehat{\calM^{\mathrm{int}}_{\calG',b_{\xbar'},\mu',
 		\xbar'_0}}.}\] Since the right vertical morphism is an isomorphism by \cite[Theorem 5.2]{PRLSV}, it follows that $\widehat{U}^{\lozenge}_{\xbar}\rightarrow\widehat{U}_{\xbar'}^\lozenge$ is an isomorphism as desired.
	\end{proof}

\subsubsection{}We now assume that $f:(\bfG,X)\rightarrow(\bfG',X')$ induces an isomorphism of derived groups and that the parahorics $\G$ and $\G'$ are associated. As in \S\ref{subsubsec: construction of integral model}, we fix a connected component $X^+\subset X$ which determines  neutral connected components $\scrS_{\rmK_p}(\bfG,X)^+$ and $\scrS_{\rmK'_p}(\bfG',X')^+$; for notational convenience we assume these are base changed to $\calO_{E^{\ur}}$.
\begin{cor}\label{cor: iso connected components}
	The morphism $f_{\scrS}$ induces an isomorphism of $\calO_{E^{\ur}}$-schemes.
	$$\scrS_{\rmK_p}(\bfG,X)^+\rightarrow \scrS_{\rmK'_p}(\bfG',X')^+.$$
\end{cor}
\begin{proof} We will consider neat compact open subgroups $\rmK^{1,p}, \rmK^{2,p} \subset \bfG'(\bbA_f^p),$ and we write 
	$\rmK^1= \rmK'_{p}\rmK^{1,p}$ and $\rmK^2= \rmK'_p\rmK^{2,p}.$ 
	Since the morphism $\bfG \rightarrow \bfG'$ induces an isomorphism of derived groups, the map 
	$ \Sh_{\rmK_p}(\bfG, X)^+ \rightarrow  \Sh_{\rmK'_p}(\bfG', X')^+ $
	is an isomorphism. Thus for any sufficiently small neat compact open $\rmK^{p} \subset \bfG(\bbA_f^p),$ there exist 
	$\rmK^{1,p}, \rmK^{2,p} \subset \bfG'(\bbA_f^p)$  such that $f$ induces maps 
	\begin{equation}\label{eqn:compositegenshims}
	\Sh_{\rmK^2}(\bfG', X')^+ \rightarrow  \Sh_{\rmK}(\bfG, X)^+ \rightarrow 
	\Sh_{\rmK^1}(\bfG', X')^+.
	\end{equation}
	Let $\scrS^\dagger_{\rmK}(\bfG,X)^+$ be the normalization of $\scrS_{\rmK^1}(\bfG',X')^+$ in 
	$\Sh_{\rmK}(\bfG,X)^+$. Then \eqref{eqn:compositegenshims} extends to a sequence of morphisms 
	$$\scrS_{\rmK^2}(\bfG',X')^+\rightarrow\scrS^\dagger_{\rmK}(\bfG,X)^+\rightarrow 
	\scrS_{\rmK^1}(\bfG',X')^+$$  
	whose composite is finite \'etale. It follows that  both maps in the sequence are finite, and since all the schemes are normal, 
	both maps are finite \'etale. Passing to the limit with $\rmK^{2,p}$ and $\rmK^p$ we obtain a commutative diagram
	$$ \xymatrix{ \scrS_{\rmK^2_p}(\bfG',X')^+ \ar[r]\ar[d] &\scrS^\dagger_{\rmK_p}(\bfG,X)^+ \ar[d] \\
		\scrS_{\rmK^2}(\bfG',X')^+ \ar[r] & \scrS^\dagger_{\rmK}(\bfG,X)^+}
	$$
	Since the map on the left is pro-finite \'etale, and the bottom map is finite \'etale, the 
	map on the right is pro-finite \'etale. 
	
	By Proposition \ref{prop: morphism SV iso completions} and the normality of $\scrS_{\rmK}(\bfG,X)^+$, 
	there is also a morphism  
	$\alpha:\scrS_{\rmK}(\bfG,X)^+\rightarrow \scrS^\dagger_{\rmK}(\bfG,X)^+$, whose composite with $\scrS^\dagger_{\rmK}(\bfG,X)^+\rightarrow \scrS_{\rmK^1}(\bfG',X')^+$ is \'etale, 
	and hence $\alpha$ is \'etale. Again, passing to the limit with $\rmK^p,$ we obtain a commutative diagram 
	$$ \xymatrix{ \scrS_{\rmK_p}(\bfG,X)^+ \ar[r]\ar[d] &\scrS^\dagger_{\rmK_p}(\bfG,X)^+ \ar[d] \\
		\scrS_{\rmK}(\bfG,X)^+ \ar[r]^{\alpha} & \scrS^\dagger_{\rmK}(\bfG,X)^+}
	$$
	where the vertical maps are pro-finite \'etale. For any finite extension $K$ of $W(k)[1/p],$ a point $x^\dagger \in\scrS^\dagger_{\rmK}(\bfG,X)^+(\calO_K)$ lifts to a point of $\tilde x^\dagger \in\scrS^\dagger_{\rmK_p}(\bfG,X)^+(\calO_K),$ 
	and hence to a point $\tilde x \in\scrS_{\rmK_p}(\bfG,X)^+(K).$ By Theorem \ref{thm: local model diagram} (2), 
	$\tilde x$ extends to a point in $\scrS_{\rmK_p}(\bfG,X)^+(\O_K).$ This implies that $\alpha$ is surjective. 
	
	Thus $\alpha$ is a surjective \'etale birational morphism between normal schemes, hence an isomorphism.  We thus obtain a morphism 
	$\scrS_{\rmK^2}(\bfG',X')^+\rightarrow\scrS_{\rmK}(\bfG,X)^+$ which, 
	after taking the inverse limit, gives an inverse for the morphism 
	$$\scrS_{\rmK_{p}}(\bfG,X)^+\rightarrow \scrS_{\rmK'_p}(\bfG',X')^+$$ induced by $f_{\scrS}$.
\end{proof}

\subsubsection{}\label{subsec: choice of i_2} 
We now use the notation of \S\ref{sec: integral models abelian type}. We let $(\bfG_2,X_2,\calG_2)$ be a  strongly acceptable triple   and  we write $\rmK_{2,p}=\calG_2(\bbZ_p)$ a very special parahoric. Fix $\rmK_2^p\subset \bfG_2(\bbA_f^p)$ a compact open subgroup and set $\rmK_2=\rmK_{2,p}\rmK_2^p$.  By Theorem \ref{cor: strongly acceptable integral model}, we may construct an integral model $\scrS_{\rmK_2}(\bfG_2,X_2)=\scrS_{\rmK_{2,p}}(\bfG_2,X_2)/\rmK_2^p$ over $\calO_{E_2}$ for $\Sh_{\rmK_2}(\bfG_2,X_2)$ by viewing $(\bfG_2,X_2)$ as a Shimura datum of abelian type  and using an auxiliary Shimura datum 
$(\bfG,X)$ from Proposition \ref{lemma: auxiliary Hodge type datum} together with a choice of Hodge embedding $\iota:(\bfG,X)\rightarrow (\mathbf{GSp}(V),S^\pm)$ satisfying the assumptions (A')--(C') in \S\ref{subsub: integral model Hodge type construction} and  \S\ref{subsubsec: Assumption C'}. We fix such a $(\bfG,X)$ and $\iota$ for the rest of this section.

Now let $\iota_2:(\bfG_2,X_2)\rightarrow (\mathbf{GSp}(V_2),S_2^\pm)$ be \emph{any} Hodge embedding. By the main theorem of \cite{Landvogt}, 
$\iota_2$ induces a $G_2(\bbQ_p^{\ur})$ and $\mathrm{Gal}(\bbQ_p^{\ur}/\bbQ_p)$-equivariant  embedding of buildings. Upon replacing $\iota_2$ with a new Hodge embedding and applying Zarhin's trick we may assume there is a $\bbZ_p$-lattice  $V_{2,\bbZ_p}\subset V_{2,\bbQ_p}$   with 
$V_{2,\bbZ_p}=V_{2,\bbZ_p}^\vee$ such that $G_2\rightarrow \GSp(V_{2,\bbQ_p})$ extends to a morphism of Bruhat--Tits stabilizer schemes $\widetilde{\calG}_2\rightarrow \mathcal{GSP}$, where $\mathcal{GSP}$  is the group scheme stabilizer  of $V_{2,\bbZ_p}$ in $\mathrm{GSp}(V_{2,\bbQ_p})$.
We set $\rmK'_{2,p}:=\mathcal{GSP}(\bbZ_p)$ and $\rmK_2'=\rmK_{2,p}'\rmK_2'^p$ where  $\rmK_2'^p\subset \mathbf{GSp}(V_{2,\bbA_f^p})$ is a compact open subgroup containing $\rmK_2^p$.

\begin{prop}\label{prop: morphism integral models} There is a map of $\calO_{E_2}$-stacks 
	\begin{equation}\label{eqn: Hodge type SV map}\scrS_{\rmK_2}(\bfG_2,X_2)\rightarrow \scrS_{\rmK_2'}(\mathbf{GSp}(V_2),S_2^\pm)_{\calO_{E_2}}\end{equation}
	extending the natural map on the generic fiber.  
\end{prop}

\subsubsection{} To prove this proposition, we make use of the following auxiliary construction. Let 
$\bfG_3$ be the identity component of  $\bfG_2 \times_{\mathbf{ G}^{\ad}, \mathbb{G}_m} \bfG,$ where the projections 
onto $\mathbb G_m$ are given by composing $\iota, \iota_2$ with the symplectic multipliers.  	There are natural morphisms $\bfG_3\rightarrow\bfG_2$  and $\bfG_3\rightarrow\bfG$, the latter of which induces an isomorphism $\bfG_{3}^{\der}\xrightarrow{\sim}\bfG^{\der}$. Let $h\in X^+$. As in \S\ref{subsubsec: construction of integral model}, we may choose the
	isomorphism $\bfG^{\ad}\cong \bfG_2^{\ad}$ in such a way that we may  view $X^+$ as a subset of $X_2$, and we let $h_2\in X_2$ denote the element determined by $h$. The homomorphism $$h_3:=(h_2, h):\bbS\rightarrow \bfG_2\times \bfG$$ factors through $\bfG_3$, and we denote by $X_3$ the $\bfG_{3,\bbR}$-orbit of $h_3$. The pair $(\bfG_3,X_3)$ forms a Shimura datum which is equipped with a Hodge embedding $\iota_3:(\bfG_3,X_3)\rightarrow (\mathbf{GSp}(V_3),S_3^\pm)$  induced from $(\iota,\iota_2)$; here $V_3=V\oplus V_2$. The lemma below ensures $(\bfG_3,X_3)$ satisfies the assumptions in Proposition \ref{prop: morphism SV iso completions}. We set $G_3=\bfG_{3,\bbQ_p}$.

\begin{lemma}\label{lem: G_3 torsion free}Let $\calG_3$ be the very special parahoric for $G_3$ associated to $\calG_2$. \begin{enumerate}
	\item 	We have $\tG_3=\G_3$.
		\item $(\bfG_3,X_3)$ admits a good Hodge embedding (with respect to $\calG_3$).
	\end{enumerate}
\end{lemma}
\begin{proof} 
		
	For (1), it suffices by Lemma \ref{lem: v special connected} to show that $X_*(G_3^{\ab})_I$ is torsion-free.  By \cite[Lemma 7.2.5]{KPZ}, we have an sequence \[\xymatrix{1 \ar[r] &Z_{G_2}^c\times Z_{G}^c\ar[r]&Z_{G_3}\ar[r] & \bbG_m\ar[r] &1.}\] By Proposition \ref{lemma: auxiliary Hodge type datum} (4) and Definition \ref{def: strongly acceptable}, $Z_G^c$ and $Z_{G_2}^c$ are both tori. It follows that $Z_{G_3}$ is a torus with $Z_{G_3}^{c}\cong Z_{G_2}^{c}\times Z_G^{c}$ and  the map $Z_{G_3^{\der}}=Z_{G^{\der}}\rightarrow Z^c_{G_2}\times Z_{G}^c$ is given by the diagonal embedding. Thus we have an exact sequence  of tori
	\[\xymatrix{1\ar[r] &Z_{G_2}^c\ar[r] & Z_{G_3}^c/Z_{G_3^{\der}}\ar[r]& Z_G^c/Z_{G^{\der}}\ar[r]& 1.} \]
	Note that $X_*(G^{\ab})_I$ is an extension of $\bbZ$ by $X_*(Z_G^c/Z_{G^{\der}})_I$; hence $X_*(Z_G^c/Z_{G^{\der}})_I$ is torsion-free by Proposition \ref{lemma: auxiliary Hodge type datum} (5). By assumption (see Definition \ref{def: strongly acceptable}), $X_*(Z^c_{G_2})_I$ is  torsion-free. It follows that $X_*(Z_{G_3}^c/Z_{G_3^{\der}})_I$ is torsion-free, and hence $X_*(G^{\ab}_3)_I$, which is an extension of $\bbZ$ by $X_*(Z_{G_3}^c/Z_{G_3^{\der}})_I$ is torsion-free.
	
	For (2), note that $p\nmid|\pi_1(G_3^{\der})|=|\pi_1(G^{\der})|$ since $p>2$. Thus by Lemma \ref{lem: v good Hodge embedding}, it suffices to show that the centralizer of a maximal $\brQ$-split torus in $G_3$ is $R$-smooth.
The isomorphism $Z_{G_3}^{c}\cong Z_{G_2}^{c}\times Z_G^{c}$ implies that $Z_{G_3}^c$ is a product of Weil-restrictions of tame tori, and hence is $R$-smooth by Proposition \ref{prop: examples of R smooth torus}. Then $Z_{G_3}$ is an extension of $\bbG_m$ by an $R$-smooth torus and  $Z_{G_3}$ is $R$-smooth. The result then follows from \cite[Lemma 7.2.6]{KPZ}.
	\end{proof}

\begin{proof}[Proof of Proposition \ref{prop: morphism integral models}] It suffices to construct a map 
	\begin{equation}\label{eqn: Hodge type SV mapII}
	\scrS_{\rmK_{2,p}}(\bfG_2,X_2)\rightarrow \scrS_{\rmK_{2,p}'}(\mathbf{GSp}(V_2),S_2^\pm)_{\calO_{E_2}}\end{equation}
	which is $\bfG_2(\mathbb A^p_f)$-equivariant. Let $\scrS_{\rmK_{2,p}}(\bfG_2,X_2)'$ be the closure of 
	\begin{equation}
	\Sh_{\rmK_{2,p}}(\bfG_2,X_2)\rightarrow \scrS_{\rmK_{2,p}}(\bfG_2,X_2) \times 
	\scrS_{\rmK_{2,p}'}(\mathbf{GSp}(V_2),S_2^\pm)_{\calO_{E_2}}\end{equation}
	Then the existence of \eqref{eqn: Hodge type SV mapII} is equivalent to requiring that 
	$$ \scrS_{\rmK_{2,p}}(\bfG_2,X_2)' \rightarrow \scrS_{\rmK_{2,p}}(\bfG_2,X_2) $$
	is an isomorphism. We may check this over $\O_{E'},$ where $E' \supset E_2,$ is any complete, discretely valued extension of $E_2.$ 
	In particular, we may assume that the connected components of $\scrS_{\rmK_{2,p}}(\bfG_2,X_2)$ are defined over $\O_{E'}.$ 
	
	Let $\scrS \rightarrow \scrS'$ be a map of connected components induced by (\ref{eqn: Hodge type SV mapII}). 
	Then the explicit description given by 
	\eqref{eqn: abelian type construction} shows that one may identify the diagrams  
	\[\xymatrix{\scrS[1/p] \ar[r]\ar[d] & \scrS'[1/p] \ar[d] \\
		\scrS & \scrS'}\]
	coming from different choices of $\scrS.$ Thus, it suffices to construct the map 
	\begin{equation}\label{eqn: Hodge type SV mapIII}
	\scrS_{\rmK_{2,p}}(\bfG_2,X_2)^+_{\O_{E'}}\rightarrow \scrS_{\rmK_{2,p}'}(\mathbf{GSp}(V_2),S_2^\pm)^+_{\calO_{E'}}\end{equation}
	where $\scrS_{\rmK'_{2,p}}(\mathbf{GSp}(V_2),S_2^\pm)_{\calO_{E'}}^+$ is the connected component corresponding to the  
	connected component of $S_2^{\pm}$ containing the image of $X_2^+$.

	To do this we make use of the Shimura datum $(\bfG_3,X_3)$ constructed above. This is equipped with morphisms of Shimura data	\[\xymatrix{(\bfG,X)&(\bfG_3,X_3) \ar[r]\ar[l]& (\bfG_2, X_2) \ar[r] & (\mathbf{GSp}(V_2),S_2^\pm),}\]where the leftmost morphism induces an isomorphism on derived groups. Let $\calG$ and $\calG_3$ denote the parahoric group schemes for $G$ and $G_3$ associated to 
	$\calG_2$, and set $\rmK_p=\calG(\bbZ_p)$, $\rmK_{3,p}=\calG_3(\bbZ_p)$. We may construct an integral model $\scrS_{\rmK_{3,p}}(\bfG_3,X_3)$ for $\Sh_{\rmK_{3,p}}(\bfG_3,X_3)$ as in \S\ref{subsec: integral models Hodge type} using a good Hodge embedding provided by Lemma \ref{lem: G_3 torsion free}.
	
	By Lemma \ref{lem: v special connected} and Lemma \ref{lem: G_3 torsion free}, we have $\calG_3=\tilde{\calG}_3$ and similarly $\tilde{\calG}=\calG$. Thus we may apply Corollary \ref{cor: iso connected components} to the morphism $(\bfG_3,X_3)\rightarrow (\bfG,X)$. We assume that $E'$ is large enough that the connected components of $\Sh_{\rmK_3}(\bfG_3,X_3)$ are defined over $E'.$
	Thus  we  have an isomorphism 
	\begin{equation}\label{eqn: G_2 to G}
	\scrS_{\rmK_{3,p}}(\bfG_3,X_3)_{\calO_{E'}}^+\xrightarrow{\sim} \scrS_{\rmK_{p}}(\bfG,X)_{\calO_{E'}}^+. \end{equation}
	
	By Proposition \ref{prop: morphism SV iso completions} and taking inverse limits, there is a morphism of integral models 
	$$ \scrS_{\rmK_{3,p}}(\bfG_3,X_3)_{\calO_{E'}}\rightarrow \scrS_{\rmK_{2,p}'}(\mathbf{GSp}(V_2), S_2^\pm)_{\calO_{E'}}.$$
Restricting to neutral connected components, we obtain a morphism
	\begin{equation}\label{eqn: G_2 to GSp}\scrS_{\rmK_{3,p}}(\bfG_3,X_3)_{\calO_{E'}}^+ \rightarrow \scrS_{\rmK'_{2,p}}(\mathbf{GSp}(V_2),S_2^\pm)_{\calO_{E'}}^+.\end{equation}
	By the construction of $\scrS_{\rmK_{2,p}}(\bfG_2,X_2)_{\calO_{E'}}$ (cf. \eqref{eqn: abelian type construction}), we have 
	$$\scrS_{\rmK_{2,p}}(\bfG_2,X_2)_{\calO_{E'}}^+=\scrS_{\rmK_p}(\bfG,X)_{\calO_{E'}}^+/\Delta(\bfG,\bfG_2)\cong \scrS_{\rmK_{3,p}}(\bfG_3,X_3)_{\calO_{E'}}^+/\Delta(\bfG,\bfG_2),$$
	where 
	$\Delta(\bfG,\bfG_2):=\ker(\scrA(\bfG_{\bbZ_{(p)}})^\circ \rightarrow \scrA(\bfG_{2,\bbZ_{(p)}}))$. 
	The  map \eqref{eqn: G_2 to GSp} factors through the action of $\Delta(\bfG,\bfG_2)$, since it does so on the generic fiber. 
	We thus obtain a map 
	$\scrS_{\rmK_{2,p}}(\bfG_2,X_2)_{\calO_{E'}}^+\rightarrow \scrS_{\rmK'_{2,p}}(\mathbf{GSp}(V_2),S_2^\pm)_{\calO_{E'}}^+$ 
	as desired.
\end{proof}

\subsection{$\mu$-ordinary locus and canonical liftings}\label{sec: canonical liftings}  
\subsubsection{}\label{sec: canonical liftings 1}

 In this subsection, we study the $\mu$-ordinary locus in the strongly acceptable case and prove the existence of canonical liftings. As in the construction of the local model diagram, the result will be deduced from the corresponding result in the special Hodge type case given by Theorem \ref{thm: canonical lift HT}. 

Let $(\bfG_2,X_2,\calG_2)$ be a strongly acceptable triple and $\rmK_2^p\subset \bfG_2(\bbA_f^p$) a compact open subgroup. We have the integral model $\scrS_{\rmK_2}(\bfG_2,X_2)$ over $\calO_{E_2}$ which is constructed from an auxiliary Hodge-type Shimura datum $(\bfG,X)$ and a choice of good Hodge embedding $ \iota$ satisfying assumptions (A')--(C') of \S\ref{subsub: integral model Hodge type construction} and \S\ref{subsubsec: Assumption C'}. Let $\iota_2:(\bfG_2,X_2)\rightarrow (\mathbf{GSp}(V_2),S_2^\pm)$  be a Hodge embedding and $V_{2,\bbZ_p}\subset V_{2,\bbQ_p}$ a self-dual lattice as in \S\ref{subsec: choice of i_2}. Then by Proposition \ref{prop: morphism integral models}, there is a morphism of integral models \begin{equation}\label{eqn: morphism integral models}\scrS_{\rmK_2}(\bfG_2,X_2)\rightarrow \scrS_{\rmK_2'}(\mathbf{GSp}(V_2),S_2^\pm)_{\calO_{E_2}}.\end{equation}

Let $h:\calA_2\rightarrow \scrS_{\rmK_2}(\bfG_2,X_2)$ denote the pullback of the universal abelian variety along (\ref{eqn: morphism integral models}). Let $s_{\alpha}\in V_2^\otimes$ be a collection of tensors whose stabilizer is $\bfG_2$. Then as in \S\ref{subsubsec:hodgecycles}, these give rise to tensors $s_{\alpha,B}\in V_B:=R^1h_{\mathrm{an*}}\bbQ$, $s_{\alpha,\ell}\in\calV_\ell(\calA_2):=R^1h_{\mathrm{\acute{e}t*}}\bbQ_\ell$ for all $\ell\neq p$ and $s_{\alpha,p}\in\calV_p(\calA_2):=R^1h_{\eta,\mathrm{\acute{e}t*}}\bbQ_p$. For any $\calO_{E_2}$-scheme $T$ and $x\in \scrS_{\rmK_2}(\bfG_2,X_2)(T)$, we write $\calA_{2,x}$ for the pullback of $\calA_2$ to $x$.

For $K/\brQ$ finite and $\widetilde{x}\in \scrS_{\rmK_2}(\bfG_2,X_2)(\calO_K)$ with special fiber $\overline{x}$, we let $s_{\alpha,0,\widetilde{x}}\in \bbD(\calA_{2,\overline{x}}[p^\infty])[1/p]^\otimes$ denote the images of $s_{\alpha,p,\widetilde{x}}$ under the $p$-adic comparison isomorphism. As in \S\ref{subsec: integral models formal nbd}, these tensors depend only on $\overline{x}$ and not on $\widetilde{x}$; we thus write $s_{\alpha,0,\overline{x}}$ for these tensors.  Note that \cite[Prop. 1.3.7]{KMS} applies here  since the morphism   $\scrS_{\rmK_2}(\bfG_2,X_2)\rightarrow \scrS_{\rmK_2'}(\mathbf{GSp}(V_2),S_2^\pm)_{\calO_{E_2}}$ factors through the normalization of its scheme theoretic image, and all objects are pulled back from this.

\subsubsection{}Let $\xbar\in \mathscr{S}_{\rmK_2}(\bfG_2,X_2)(k)$, and set $\bbD:=\bbD(\calA_{2,\overline{x}}[p^\infty])$. We fix an isomorphism $$V^\vee_{2,\Z_p}\otimes_{\bbZ_p}\brQ\cong\D\otimes_{\brZ}\brQ,$$ taking $s_{\alpha}$ to $ s_{\alpha,0,\overline{x}}$; such an isomorphism exists by Steinberg's theorem (cf. \cite[1.3.8]{KMS}). Then as in \S\ref{sec: mu ordinary locus HT}, we obtain an element $b\in G_2(\brQ)$ with $[b]\in B(G_2,\{\mu_2\})$ where $\{\mu_2\}=\{\mu_{h_2}^{-1}\}$. This induces the Newton stratification on the geometric special fiber $\calS_{\rmK_2,k}$  (resp. $\calS_{\rmK_{2,p},k}$) of  $\scrS_{\rmK_2}(\bfG_2,X_2)$ (resp. $\scrS_{\rmK_{2,p}}(\bfG_2,X_2)$). 
We write $\calS_{{\rmK_2},[b]}\subset\calS_{{\rmK_2},k}$ for the strata corresponding to $[b]\in B(G,\{\mu_2\})$. We also write
$$ \calS_{{\rmK_{2,p}},[b]} = \underset{\leftarrow\rmK_2^p}{\lim} \calS_{{\rmK_{2,p}\rmK_2^p},[b]},$$ which makes sense since  $\calS_{{\rmK_{2}},[b]}$ is compatible with the prime-to-$p$ level. For the rest of \S\ref{sec: canonical liftings} we assume the existence of the class $[b]_{\mu_2}\in B(G_2,\{\mu_2\})$ in Definition \ref{def: mu ordinary}.

\begin{definition}
	We define the $\mu_2$-ordinary locus of $\calS_{\rmK_2,k}$ to be $\calS_{\rmK_2,[b]_{\mu_2}}$.
\end{definition}

 The following is deduced easily from \cite[Corollary 1.3.16]{KMS}.
\begin{thm}\label{thm: density}
	Assume $\rmK_2^p$ is neat.  Then
	\begin{enumerate}
		\item $\calS_{\rmK_2}$ is normal.
		
		\item The $\mu_2$-ordinary locus $\calS_{\rmK_{2},[b]_{\mu_2}}$ is Zariski open and dense in $\calS_{\rmK_2,k}$.
	\end{enumerate}
\end{thm}
\begin{proof} To show (1), it suffices by Theorem \ref{cor: strongly acceptable integral model} to show that the  special fiber of $\bbM^{\mathrm{loc}}_{\calG_2,\{\mu_{h_2}\}}$ is normal. Note that the geometric irreducible components of this special fiber are normal (see \S\ref{sec: construction local model}), and hence it suffices to show that  $\bbM^{\mathrm{loc}}_{\calG_2,\{\mu_{h_2}\}}\otimes_{\calO_E}k$ is integral.
	This follows from the argument in \cite[Corollary 9.4]{PZ}, 
	noting that as in {\em loc.~cit.} the $\mu$-admissible set $\Adm(\{\mu\})_J$ has a single extremal element when $J\subset \bbS$ corresponds to a very special standard parahoric of $G(\brQ)$.
	
	(2) follows from (1) by \cite[Corollary 1.3.16]{KMS}.
\end{proof}

\subsubsection{}\label{subsec: construction of I groups} Let $\xbar\in \mathscr{S}_{\rmK_2}(\bfG_2,X_2)(k)$. Then we can define $I_{\xbar}\subset \text{Aut}_\Q(\mathcal{A}_{2,\xbar})$ to be the subgroup preserving $s_{\alpha,0,\ell}$ and $s_{\alpha,0,\xbar}$ as in \S\ref{sec: mu ordinary locus HT}.
The goal of the rest of this section is to prove the following generalization of Theorem \ref{thm: canonical lift HT}.
\begin{thm}\label{thm: can lift for Shimura var}
	Let $(\bfG_2,X_2,\calG_2)$ be a strongly  acceptable triple. Let  $\xbar\in \calS_{\rmK_2,[b]_{\mu_2}}(k)$. Then $\xbar$ admits a lifting to a special point $\widetilde{x}\in\mathscr{S}_{\rmK_2}(\bfG_2,X_2)(K)$ for some $K/\brQ$ finite such that the action of $I_{\xbar}(\Q)$ on $\mathcal{A}_{2,\xbar}$ lifts to an action (in the isogeny category) on $\mathcal{A}_{2,\widetilde{x}}$.
\end{thm}

We will deduce this theorem from \ref{thm: canonical lift HT} using the auxiliary construction from Proposition \ref{prop: morphism integral models}.  For notational convenience, we write $(\bfG_1,X_1)$ for $(\bfG,X)$ and $\iota_1:(\bfG_1,X_1)\rightarrow (\mathbf{GSp}(V_1),S_1^\pm)$ for the Hodge embedding $\iota$. Then $\bfG_3$ is defined  to be  the identity component of $\bfG_1\times_{\bfG_1^{\ad},\bbG_m}\bfG_2$.
We obtain a Shimura datum $(\bfG_3,X_3)$ together with morphisms  $(\bfG_1,X_1)\leftarrow (\bfG_3,X_3)\rightarrow(\bfG_2,X_2)$ and a Hodge embedding $\iota_3:(\bfG_3,X_3)\rightarrow (\mathbf{GSp}(V_3),S_3^\pm)$, where $V_3=V_1\oplus V_2$.

For $i=1,2,3$, let $\bfE_i$ denote the reflex field of $(\bfG_i,X_i)$; then  we have $\bfE_3\subset\bfE':=\bfE_1\bfE_2$.  We let $v_i$  (resp. $v'$) denote the place of $\bfE_i$ (resp. $\bfE'$) induced by the embedding $i_p$ and we let $E_i$ (resp. $E'$) denote  the completion. By construction, we have $E'=E_2$. 
Set $G_i:=\bfG_{i,\bbQ_p}$, and let $\calG_1$ (resp. $\calG_3$) denote the parahoric group scheme of $G_1$ (resp. $G_3$) determined by $\calG_2$. For $i=1,2,3$, we set $\rmK_{i,p}:=\calG_i(\bbZ_p)$ and we fix compact open subgroups $\rmK_i^p\subset\bfG_i(\bbA_f^p)$ such that $\rmK_3^p$ maps to $\rmK_1^p$ and $\rmK_2^p$. We set $\rmK_i:=\rmK_{i,p}\rmK_i^p$. We then have integral models $\scrS_{\rmK_i}(\bfG_i,X_i)$, where $\scrS_{\rmK_1}(\bfG_1,X_1)$ is constructed from the (very) good Hodge embedding $\iota_1$, $\scrS_{\rmK_2}(\bfG_2,X_2)$ is constructed from $\scrS_{\rmK_1}(\bfG_1,X_1)$  by viewing $(\bfG_2,X_2)$ as a Shimura data of abelian type, and $\scrS_{\rmK_3}(\bfG_3,X_3)$ is constructed from a good Hodge embedding as in Lemma \ref{lem: G_3 torsion free}.

\subsubsection{}  Let $\bfH$   denote the subgroup of $ \mathbf{GSp}(V_1)\times\mathbf{GSp}(V_2)$ consisting of elements $(g_1,g_2)$ such that $c_{1}(g_1)=c_{2}(g_2)$. Then the natural map $\bfG_3\rightarrow \mathbf{GSp}(V_1)\times\mathbf{GSp}(V_2)$ factors through $\bfH$ and we let $S_H$ denote the $\bfH_{\bbR}$-conjugacy class of homomorphisms $\bbS\rightarrow \bfH_{\bbR}$ induced by $X_3$.
There  are natural morphisms of Shimura data $(\bfH,S_H)\rightarrow (\mathbf{GSp}(V_i),S_i^\pm)$ for $i=1,2,3$.

We let $V_{1,\bbZ_p}\subset V_{1,\bbQ_p}$ be a $\bbZ_p$-lattice such that $\iota_1$ extends to a good local integral Hodge embedding $\calG_1\rightarrow V_{1,\bbZ_p}$ which is very good at all points of $\Mloc_{\calG_1,\mu_{h_1}}(k)$, and we set $V_{3,\bbZ_p}:=V_{1,\bbZ_p}\oplus V_{2,\bbZ_p}\subset V_{3,\bbQ_p}$.  For $i=1,2,3$, we let $\rmK'_{i,p}$ denote the stabilizer of $V_{i,\bbZ_p}$ inside $\mathbf{GSp}(V_{i,\bbQ_p})$ and  let $\rmH_p$ denote the stabilizer of $V_{3,\bbZ_p}$ inside $\bfH(\bbQ_p)$.  We also fix compact open subgroups  $\rmK_i'^p\subset \mathbf{GSp}(V_{i,\bbA_f^p})$ containing the image of $\rmK_i^p$ for $i=1,2,3$, $\rmH^p\subset \bfH(\bbA_f^p)$ containing the image of $\rmK_3^p$, and we set $\rmK_i'=\rmK'_{i,p}\rmK_i'^p$, $\rmH=\rmH_p\rmH^p$.
Then the Shimura variety $\Sh_{\rmH}(\bfH,S_H)$ admits a moduli interpretation as pairs of tuples $(A_i,\lambda_i,\epsilon_i)$, $i=1,2$, where $A_i$ are abelian varieties of $\dim(V_i)/2$, $\lambda_i$ is a weak polarization, and $\epsilon_i$ are level $\mathrm{Im}(\rmH^p\rightarrow \mathbf{GSp}(V_{i,\bbA_f^p}))$-structures which preserve symplective pairings up to the same $\bbA_f^{p\times}$-scalar  (cf. \cite[7.2]{Z}). This  moduli problem extends to $\bbZ_{(p)}$,  hence we obtain   an integral model $\scrS_{\rmH}(\bfH,S_H)/\bbZ_{(p)}$. 

\begin{prop}
	There is a commutative diagram of $\calO_{E'}$-stacks
	\begin{equation}\label{eqn: product SV diagram}\xymatrix{\scrS_{\rmK_1}(\bfG_1,X_1)_{\calO_{E'}}\ar[d]^{i_1}&\scrS_{\rmK_3}(\bfG_3,X_3)_{\calO_{E'}}\ar[r]^{j_2}\ar[d]^{i_3}\ar[l]_{j_1}&\scrS_{\rmK_2}(\bfG_2,X_2)_{\calO_{E'}}\ar[d]^{i_2}\\
		\scrS_{\rmK_1'}(\mathbf{GSp}(V_1),S_1^\pm)_{\calO_{E'}}&\scrS_{\rmK'}(\bfH,S_H)_{\calO_{E'}}\ar[r]\ar[l]&\scrS_{\rmK'}(\mathbf{GSp}(V_2),S_2^\pm)_{\calO_{E'}}}.\end{equation}

\end{prop}
\begin{proof}It suffices to consider the case of neat prime-to-$p$ level structure so that we may assume all objects are schemes. The existence of the bottom row follows from the moduli interpretations of the integral models. 
The morphism $j_1$ is constructed in Proposition \ref{prop: morphism SV iso completions} and $j_2$ is constructed  in a similar way to Proposition \ref{prop: morphism integral models}.
	
	The morphism $i_1$ exists by construction of $\scrS_{\rmK_1}(\bfG_1,X_1)_{\calO_{E'}}$ and $i_2$  is  constructed in Proposition \ref{prop: morphism integral models}. For $\iota_3$, note that there is a finite morphism $$\scrS_{\rmK'}(\bfH,S_H)_{\calO_{E'}}\rightarrow \scrS_{\rmK_1'}(\mathbf{GSp}(V_1),S_1^\pm)_{\calO_{E'}}\times\scrS_{\rmK_2'}(\mathbf{GSp}(V_2),S_2^\pm)_{\calO_{E'}}.$$ The morphism $$\scrS_{\rmK_3}(\bfG_3,X_3)_{\calO_{E'}}\rightarrow \scrS_{\rmK_1'}(\mathbf{GSp}(V_1),S_1^\pm)_{\calO_{E'}}\times\scrS_{\rmK_2'}(\mathbf{GSp}(V_2),S_2^\pm)_{\calO_{E'}}$$  induced by  $(i_1\circ j_1,i_2,\circ j_2)$ factors through $\scrS_{\rmK'}(\bfH,S_H)_{\calO_{E'}}$ on the generic fiber, and hence lifts to a morphism $i_3:\scrS_{\rmK_3}(\bfG_3,X_3)_{\calO_{E'}}\rightarrow \scrS_{\rmK'}(\bfH,S_H)_{\calO_{E'}}$ as desired.
\end{proof} 

\subsubsection{}

Let  $\calA_{i}\rightarrow \scrS_{\rmK_i}(\bfG_i,X_i)_{\calO_{E'}}$,  denote the pullback of the universal abelian variety along $\scrS_{\rmK_i}(\bfG_i,X_i)_{\calO_{E'}}\rightarrow \scrS_{\rmK_i'}(\mathbf{GSp}(V_i,S_i^\pm)_{\calO_{E'}}$. For $i=3$, this map factors through $\scrS_{\rmH}(\bfH,S_H)_{\calO_{E'}}$ and there is an identification \begin{equation}
\label{eqn: product abelian variety}
\calA_3\cong j_1^*\calA_1\times j_2^*\calA_2.\end{equation}

Let $\overline{x}_3\in\scrS_{\rmK_3}(\bfG_3,X_3)(k)$ and write  $\overline{x}_1\in\scrS_{\rmK_1}(\bfG_1,X_1)(k)$, $\overline{x}_2\in\scrS_{\rmK_2}(\bfG_2,X_2)(k)$ for the image of $\overline{x}_3$ under $j_1$ and $j_2$.  The isomorphism (\ref{eqn: product abelian variety}) implies we have an isomorphism $\calA_{3,\overline{x}_3}\cong \calA_{1,\overline{x}_1}\times\calA_{2,\overline{x}_2}$.  We let $I_{\overline{x}_3}\subset \mathrm{Aut}_{\bbQ}(\calA_{3,\overline{x}_3})$, $I_{\overline{x}_1}\subset \mathrm{Aut}_{\bbQ}(\calA_{1,\overline{x}_1})$ denote the groups constructed in the same way as \S\ref{subsec: construction of I groups}.

\begin{prop}\label{prop: relation between I groups}
	There are natural exact sequences:\[\xymatrix{0\ar[r] & \bfC_1 \ar[r]&I_{\overline{x}_3}\ar[r]&I_{\overline{x}_1}\ar[r]&0}\]
	\[\xymatrix{0\ar[r] & \bfC_2 \ar[r]&I_{\overline{x}_3}\ar[r]&I_{\overline{x}_2}\ar[r]&0}\]
	where $\bfC_1$ (resp. $\bfC_2$)  is the kernel of the map $f:\bfG_3\rightarrow \bfG_1$ (resp. $g:\bfG_3\rightarrow \bfG_2$).
\end{prop}

\begin{proof} 
	Since $\bfG_3\subset\bfH$, we may assume that the set of tensors  defining $\bfG_3\subset \mathbf{GL}(V_3)$ includes tensors   corresponding to the projections of $V_{3,\bbZ_{(p)}}$ onto the direct summands $V_{i,\bbZ_{(p)}}\subset V_{3,\bbZ_{(p)}}$ for $i=1,2$.  
	It follows that $I_{\overline{x}_3}$ respects the product decomposition $\calA_{3,\overline{x}_3}\cong \calA_{1,\overline{x}_1}\times\calA_{2,\overline{x}_2}$ and hence we obtain a natural map $I_{\overline{x}_3}\rightarrow \mathrm{Aut}_{\bbQ}(\calA_{1,\overline{x}_1})$.
	Similarly, by considering the pullback to $V_{3}$ of tensors defining $\bfG_{1}$, one can show that $I_{\overline{x}_3}\rightarrow \mathrm{Aut}_{\bbQ}(\calA_{1,\overline{x}_1})$ factors through $I_{\overline{x}_1}$. We obtain a natural map $I_{\overline{x}_3}\rightarrow I_{\overline{x}_1}$.  
	
	Let $\widetilde{x}_3\in\scrS_{\rmK_3}(\bfG_3,X_3)(\calO_K)$ denote a lift of $\overline{x}_3$.
	Since $\bfC_1$ lies in the center of $\bfG_3$, we have natural maps $$\bfC_1\rightarrow \mathrm{Aut}_{\bbQ}(\calA_{3,\widetilde{x}_3}\otimes_K\overline{K})\rightarrow \mathrm{Aut}_{\bbQ}(\calA_{3,\overline{x}_3,k})$$
	whose image lies in $I_{\overline{x}_3}$. 
	
	We thus obtain a sequence $\bfC_1\rightarrow I_{\overline{x}_3}\rightarrow I_{\overline{x}_1}$ and it suffices to check the exactness upon base changing to $\bbQ_{\ell}$ for some prime $\ell\neq p$. By \cite[Theorem 6]{KMS} there is a semisimple element $\gamma_{\ell}\in \bfG_3(\bbQ_{\ell})$ such that the natural inclusion $I_{\overline{x}_3}\otimes_{\bbQ}\bbQ_\ell\subset \bfG_{3,\bbQ_\ell}$ (resp.  $I_{\overline{x}_1}\otimes_{\bbQ}\bbQ_\ell\subset \bfG_{1,\bbQ_\ell}$) identifies $I_{\overline{x}_3}\otimes_{\bbQ}\bbQ_\ell$ (resp. $I_{\overline{x}_1}\otimes_{\bbQ}\bbQ_\ell)$ with the centralizer of $\gamma_{\ell}$ in $\bfG_{3,\bbQ_\ell}$ (resp. $f(\gamma_\ell)$ in $\bfG_{1,\bbQ_\ell}$).  We thus obtain the first exact sequence and  the argument for $I_{\overline{x}_2}$ is analogous.
\end{proof}

\subsubsection{Proof of Theorem \ref{thm: can lift for Shimura var}}
	It suffices to consider the case of neat prime-to-$p$ level structure. For $i=1,2,3$, we write $\calS_{\rmK_i}$ for the special fiber of the integral model $\scrS_{\rmK_{i}}(\bfG_i,X_i)$.	Let $\overline{x}_2\in\calS_{\rmK_2,[b]_{\mu_2}}(k)$. We first assume $\overline{x}_2=j_2(\overline{x}_3)$ for some  $\overline{x}_3\in\calS_{\rmK_3}(k)$; by Lemma \ref{lemma: mu-ordinary class change of groups} we have $\overline{x}_3\in\calS_{\rmK_3,[b]_{\mu_3}}(k)$.
	Let $\overline{x}_1\in \calS_{\rmK_1,[b]_{\mu_1}}(k)$ denote the image of $\overline{x}_3$. By Theorem \ref{thm: canonical lift HT} there exists $K/\brQ$ finite and  $\widetilde{x}_1\in\Sh_{\rmK_1}(\bfG_1,X_1)(K)$ lifting $\overline{x}_1$ such that the action of $I_{\overline{x}_1}(\bbQ)$  lifts to $\calA_{1,\widetilde{x}_1}$. Then we may consider $I_{\overline{x}_1}$ as a subgroup of $\bfG_1$ and we let $\bfT_1$ denote the connected component of the center of $I_{\overline{x}_1}$. The Mumford--Tate group of $\calA_{1,\widetilde{x}_1}$ is a connected subgroup of $\bfG_1$ which commutes with $I_{\overline{x}_1},$ hence is contained in $\bfT_1,$ as 
	$I_{\overline{x}_1}$ and $\bfG_1$ have the same rank.
	
	Let $\bfT_3\subset\bfG_3$ denote the identity component of the preimage of $\bfT_1$ in $\bfG_3$ and $\bfT_2$ the image of $\bfT_3$ in $\bfG_2$. By construction, the morphisms of integral models $$\scrS_{\rmK_1}(\bfG_1,X_1)_{\calO_{E'}}\leftarrow
	\scrS_{\rmK_3}(\bfG_3,X_3)_{\calO_{E'}}\rightarrow \scrS_{\rmK_2}(\bfG_2,X_2)_{\calO_{E'}}$$ induce isomorphisms of completions at geometric points in the special fiber. Thus let $\widetilde{x}_3$ (resp. $\widetilde{x}_2$) denote the point lifting $\overline{x}_3$ (resp. $\overline{x}_2$) corresponding to $\widetilde{x}_1$.
	Then the Mumford--Tate group for $\calA_{3,\widetilde{x}_3}$ (resp. $\calA_{2,\widetilde{x}_2}$) is contained in $\bfT_3$ (resp. $\bfT_2$). It follows from Proposition \ref{prop: relation between I groups} that $I_{\overline{x}_3}$ (resp. $I_{\overline{x}_2}$) is contained in the centralizer of $\bfT_3$ in $\bfG_3$ (resp. $\bfT_2$ in $\bfG_2$), and hence the action of $I_{\overline{x}_2}(\bbQ)$ lifts to one  on $\calA_{\widetilde{x}_2}$. 
	
	Now let $\overline{x}_2\in\calS_{\rmK_2,[b]_{\mu_2}}(k)$ be any point. It suffices to prove the result with 
	$\scrS_{\rmK_{2,p}}(\bfG_2,X_2)$ in place of $\scrS_{\rmK_2}(\bfG_2,X_2),$ and with $\overline{x}_2$ 
	replaced by a lift to a point of $\calS_{\rmK_{2,p},[b]_{\mu_2}}(k),$ which we will again denote 
	$\overline{x}_2$. Recall $J\subset \bfG_2(\bbQ_p)$ is a set mapping bijectively to a set of coset representatives for the image of (\ref{eqn: injection A groups}). Then by the construction of $\scrS_{\rmK_2,p}(\bfG_2,X_2)$ via $\scrS_{\rmK_1,p}(\bfG_1,X_1)$ in \S\ref{subsubsec: construction of integral model}, there exists $j\in J$ such that $\overline{x}_2\in [\scrS_{\rmK_{1,p}}(\bfG_1,X_1)^+\times\scrA(\bfG_{2,\bbZ_{(p)}})j]/\scrA(\bfG_{1,\bbZ_{(p)}})^\circ$. We let $\overline{x}_2'\in [\scrS_{\rmK_{1,p}}(\bfG_1,X_1)^+\times\scrA(\bfG_{2,\bbZ_{(p)}})]/\scrA(\bfG_{1,\bbZ_{(p)}})^\circ$ be the point corresponding to $\overline{x}_2$ under the isomorphism induced by $j$. Then upon modifying $\overline{x}_2$ by an element of $\bfG_2(\bbA_f^p)$ which only changes the abelian variety $\calA_{2,\overline{x}_2}$ up to prime-to-$p$ isogeny, we may assume $\overline{x}_2'=j_2(\overline{x}_3')$ for some $\overline{x}_3'\in \scrS_{\rmK_{3,p}}(\bfG_3,X_3)(k)$.  
	
	Let $\widetilde{x}_2'\in\scrS_{\rmK_{2,p}}(\bfG_2,X_2)(\calO_K)$ be a lift of of $\overline{x}_2',$ for some finite extension 
	$K/\breve\Q_p.$ 
	By construction, corresponding to the element $j,$ there is (after possibly increasing $K$) a point $\widetilde{x}_2 \in \scrS_{\rmK_{2,p}}(\bfG_2,X_2)(\calO_K)$ lifting $\overline{x}_2,$  and 
	a $p$-power quasi-isogeny $\calA_{2,\widetilde{x}_2}\rightarrow \calA_{2,\widetilde{x}_2'}$ taking $s_{\alpha,0,\overline{x}_2}$ to $s_{\alpha,0,\overline{x}_2'}$ (resp. $s_{\alpha,\ell,\overline{x}_2}$ to $s_{\alpha,\ell,\overline{x}_2'}$ for $\ell\neq p$). 
	By considering the reduction of this quasi-isogeny one sees that $\overline{x}'_2\in\calS_{\rmK_{2,p},[b]_{\mu}}(k),$ 
	and one also obtains an induced isomorphism $I_{\overline{x}_2}\cong I_{\overline{x}_2'}.$ 
	From what we saw above, it follows that we may choose $\widetilde{x}_2'$ such that the action of 
	$I_{\overline{x}_2'}$ lifts to $\calA_{2,\widetilde{x}'_2}$. Then the action of $I_{\overline{x}_2}\cong I_{\overline{x}_2'}$ 
	lifts to $\calA_{2,\widetilde{x}_2}$.
\qed

\subsubsection{}\label{sec: indep mu ordinary}We will use the above to deduce properties about the conjugacy class of Frobenius as in \cite[\S2.3]{Ki3}.   Assume $\xbar\in \mathcal{S}_{\rmK_2,[b]_{\mu_2}}(k)$ arises from an $\F_{q}$-point $x\in\mathscr{S}_{\rmK_2}(\bfG_2,X_2)(\bbF_q)$ where $\bbF_q$ is a finite extension of $k_{E_2}$. For $\ell \neq p$ a prime, let $\gamma_\ell$ denote the geometric $q$-Frobenius in $\mathrm{Gal}(\overline{\mathbb{F}}_q/\bbF_{q})$ acting on the dual of the $\ell$-adic Tate module $T_\ell\mathcal{A}_{2,\xbar}^{\vee}$. Since the tensors $s_{\alpha,\ell,\overline{x}}\in T_\ell\calA_{2,\xbar}^{\otimes} $  are Galois-invariant, we may consider $\gamma_\ell$ as an element of $\bfG_2(\bbQ_\ell)$ via the level structure $V_{\bbQ_\ell}\cong T_\ell\calA_{2,\overline{x}}\otimes_{\bbZ_\ell}\bbQ_\ell$.

\begin{cor}\label{cor: l indep mu ordinary}Assume $(\bfG_2,X_2,\calG_2)$ is a strongly acceptable triple. Suppose $\overline x\in\mathcal{S}_{\rmK_2,[b]_{\mu_2}}(k)$ arises from $x\in\mathscr{S}_{\rmK_2}( \bfG_2,X_2)(\bbF_q)$.
	There exists an element $\gamma_0\in \bfG_2(\Q)$, such that 
	
	\begin{enumerate}\item For $\ell\neq p$, $\gamma_0$ is conjugate to $\gamma_\ell$ in $\bfG_2(\Q_\ell)$.
		
		\item $\gamma_0$ is elliptic in $\bfG_2(\R)$.
	\end{enumerate}
\end{cor}
\begin{proof} The proof is the same as in \cite[Corollary 2.3.1]{Ki3}. Since $\mathcal{A}_{2,x}$ is defined over $\mathbb{F}_q$, the $q$-Frobenius $\gamma$ lies in $ I_{\xbar}(\Q)$. Let $\widetilde{x}\in\scrS_{\rmK_2}(\bfG_2,X_2)(K)$ denote the lifting constructed in Theorem \ref{thm: can lift for Shimura var}. Then by considering the action of $I_{\overline{x}}(\bbQ)$ on the Betti cohomology of  $\calA_{2,\widetilde{x}}$, we may consider $I_{\overline{x}}(\bbQ)$ as a subgroup of $\bfG_2(\bbQ)$. Defining $\gamma_0$ to be the image of $\gamma$ inside $\bfG_2(\bbQ)$, we have that $\gamma_0$  is conjugate to $\gamma_\ell$ in $\bfG_2(\bbQ_\ell)$ by the Betti-\'etale comparison isomorphism. 
	If $\bfT$ is any torus in $I_{\overline{x}}$ containing $\gamma_0$, the positivity of the Rosati involution implies $\bfT(\bbR)/w_{h_2}(\bbR^\times)$ is compact. Hence $\gamma_0\in \bfT(\bbQ)$ is elliptic in $\bfG_2(\bbR)$.
\end{proof}

\begin{remark}
	The elements $\gamma_\ell$ arise as the local Frobenii acting on the stalk of a $\bfG_2(\bbQ_\ell)$-local system $\bbL_\ell$ over $\calS_{\rmK_2}$; see \S\ref{sssec: local system}. Thus even though the proof of Corollary \ref{cor: l indep mu ordinary} uses the Hodge embedding $\iota_2$ in order to define the abelian variety $\calA_{2,\widetilde{x}}$, one can view it as proving a property of the local systems $\bbL_\ell$ over $\calS_{\rmK_2,[b]_{\mu_2}}$, which is intrinsic to $\scrS_{\rmK_2}(\bfG_2,X_2)$. In particular, the image of $\gamma_0$ in $\Conj_{\bfG}(\bbQ)$ is independent of $\iota_2$.
\end{remark}

\section{Independence of $\ell$ for Shimura varieties}\label{sec: l indep for Shimura var} 
We now apply the results of the previous section to prove $\ell$-independence  for the conjugacy class of Frobenius at all points on the special fiber of Shimura varieties. 

\subsection{Frobenius conjugacy classes}\label{sec: Frob conj classes}

\subsubsection{}\label{sssec: local system}

Let $p>2$ be a prime. We fix a strongly acceptable triple $(\bfG,X,\calG)$ and set $\rmK_p=\calG(\bbZ_p)$.  The associated Shimura variety has an integral model $\scrS_{\rmK}(\bfG,X)$  over $\calO_E$ constructed from an auxiliary  triple $(\bfG_1,X_1,\calG_1)$ and a (very good) Hodge embedding $\iota_1$ as in Proposition \ref{lemma: auxiliary Hodge type datum}. The auxiliary Shimura datum $(\bfG_1,X_1)$ plays a minor role in what follows.

Let  $\ell\neq p$ be a prime and suppose that in addition the compact open subgroup $\rmK\subset \bfG(\bbA_f)$ is of the form $\rmK_\ell\rmK^\ell$, with $\rmK_\ell\subset \bfG(\bbQ_\ell)$ and $\rmK^\ell\subset \bfG(\bbA_f^\ell)$. We let $\widetilde{\bbL}_\ell$ denote  the  $\bfG(\bbQ_\ell)$-local system on $\scrS_{\rmK}(\bfG,X)$ arising from the pro-\'etale covering $$\scrS_{\rmK^\ell}(\bfG,X):=\lim_{\underset{\rmK_\ell'\subset \rmK_\ell}\leftarrow}\scrS_{\rmK'_\ell\rmK^\ell}(\bfG,X)\rightarrow \scrS_{\rmK}(\bfG,X)$$ and we write $\bbL_\ell$ for the induced local system on the special fiber $\calS_{\rmK}$ over $k_E$.
 If $\iota:(\bfG,X)\rightarrow (\mathbf{GSp}(V),S^\pm)$ is a Hodge embedding as in \S\ref{sec: canonical liftings 1}  then we have an identification \begin{equation}\label{eqn: id of local systems}\bbL_\ell=\underline{\mathrm{Isom}}_{(s_\alpha,s_{\alpha,\ell})}(V_{\Q_\ell},\calV_\ell^\vee)\end{equation} where the scheme classifies  $\Q_\ell$-linear isomorphisms taking $s_{\alpha}$ to $s_{\alpha,\ell}$; here the notation is as in \S\ref{sec: canonical liftings 1}.

\subsubsection{}Let $y\in \calS_{\rmK}(\bbF_q)$ and we write $\overline{y}$ for the induced geometric point of $\calS_{\rmK}$. We let $\calS_{\rmK}^0$ denote the connected component of $\calS_{\rmK}$ containing $y$ and $\overline{x}\in \calS_{\rmK}^0(k)$ a fixed geometric point.  Over $\calS_{\rmK}^0$, the $\bfG(\bbQ_\ell)$-local system $\bbL_{\ell}$ corresponds to a homomorphism $$\rho^0_{\ell}:\pi_1(\calS_{\rmK}^0,\overline{x})\rightarrow \bfG(\bbQ_{\ell}).$$

We have a map
$$\text{Gal}(\overline{\bbF}_q/\bbF_{q})\rightarrow \pi_1(\mathcal{S}^0_{\rmK},\overline{y})\xrightarrow{\sim} \pi_1(\calS_{\rmK}^0,\overline{x}),$$ where the isomorphism $\pi_1(\mathcal{S}^0_{\rmK},\overline{y})\xrightarrow{\sim}\pi_1(\calS_{\rmK}^0,\overline{x})$
 is well-defined up to conjugation. We thus obtain a well defined conjugacy class in $\pi_1(\calS_{\rmK}^0,\overline{x})$  corresponding to the image of the geometric $q$-Frobenius and we write $\mathrm{Frob}_y$ for a representative of this conjugacy class.

\subsubsection{}\label{subsec: Conj_H} For a reductive group $H$ over a field $F$ of characteristic $0$, we write $\mathrm{Conj}_{H}$ for the variety of conjugacy classes in $H$. Explicitly, if $H=\mathrm{Spec}\  R$, the action of $H$ on itself via conjugation induces an action of $H$ on $R$, and we have $\mathrm{Conj}_{H}=\mathrm{Spec}\ R^{H}$. Then $\Conj_{H}$ is an $F$-variety which is a universal categorical quotient for this action, and the set $\mathrm{\Conj}_{H}(\overline{F})$ can be identified with the set of semisimple $H(\overline{F})$  conjugacy classes in $H(\overline{F})$ (see \cite[Chapters 0,1]{GIT}).  We write $\chi_H:H\rightarrow \mathrm{Conj}_{H}$ for the projection map. For example if $H=\GL_n$, $\Conj_{\GL_n}$ is the variety $\bbA^{n-1}_F\times\bbG_{m,F}$ and the map $\chi$ takes an element of $\GL_n$ to its associated characteristic polynomial. 

 In our setting, we thus obtain for each prime $\ell\neq p$, a well-defined element $\gamma_{y,\ell}\in\Conj_{\bfG}(\bbQ_\ell)$ corresponding to $\chi_{\bfG}(\rho_\ell^0(\mathrm{Frob}_y))$. Our main theorem concerning the $\ell$-independence property of Shimura varieties is the following.

\begin{thm}\label{thm: l indep full}
Let $p>2$ and $(\bfG,X,\calG)$ a strongly acceptable triple. Let $y\in \calS_{\rmK}(\bbF_{q})$ where $\bbF_q/k_E$ is a finite extension. Then there exists an element $\gamma_0\in \Conj_{\bfG}(\bbQ)$  such that $\gamma_0=\gamma_{y,\ell}\in\Conj_{\bfG}(\bbQ_\ell)$ for all $\ell\neq p$.
\end{thm}

\begin{remark}A group theoretic argument shows that if we assume in addition that $\bfG^{\der}$ is simply-connected,  $\gamma$ can be lifted to an element of $\bfG(\bbQ)$ (cf. Corollary \ref{cormainthm}). See also  Remark \ref{rem: refinements} about the expectations surrounding liftability of $\gamma$.
\end{remark}

The rest of \S\ref{sec: l indep for Shimura var} will be devoted to the proof of Theorem \ref{thm: l indep full}.

 \subsection{Explicit curves in the special fiber of local models} 
 
 \subsubsection{}\label{sec: KR stratification v special}
 We begin by recalling the local model diagram and  properties of the Kottwitz--Rapoport stratification. 
  By Theorem \ref{cor: strongly acceptable integral model} (3), there exists a diagram  of stacks
\begin{equation}\label{eqn: local model diagram scheme}
\xymatrix{ & \widetilde{\scrS}^{\mathrm{ad}}_{\rmK}(\bfG,X)\ar[dr]^{q} \ar[dl]_{\pi}& \\
 \scrS_{\rmK}(\bfG,X)& & \bbM^{\mathrm{loc}}_{\G,\{\mu_{h}\}}} \end{equation}
where $\pi:\widetilde{\scrS}^{\mathrm{ad}}_{\rmK}(\bfG,X)\rightarrow \scrS_{\rmK}(\bfG,X)$ is a $\G^{\ad}$-torsor. 

Let $\calM$ denote the special fiber of $\bbM^{\mathrm{loc}}_{\G,\{\mu_{h}\}}$; it is a scheme over $k_E$. By the construction of $\Mloc_{\calG,\{\mu_h\}}$ in \cite[\S3]{KPZ} (cf. \cite{Levin}), there is a reductive group scheme $\underline{G}$ over $\bbF_p((t))$ and a parahoric group scheme $\underline{\calG}$ for $\underline{G}$ such that $\calM$ is identified with a union of Schubert varieties inside the partial  affine flag variety $\mathcal{FL}_{\ucalG}$. By definition, $\mathcal{FL}_{\ucalG}=L\uG/L^+\ucalG$ is the fpqc quotient of the loop group $L\uG$ by the positive group $L^+\ucalG$ (see \cite{PR}). Let $W$  denote the Iwahori Weyl group for $G$. Fix an alcove $\fka$ such that $\calG$ is a standard parahoric  and let $J\subset \bbS$ be the subset of simple affine reflections corresponding to $\calG$. Then $\fka$ determines an alcove $\underline{\fka}$ for $\uG$, and we have an identification of simple reflections $\bbS\cong \underline{\bbS}$ in $\fka$ and $\underline{\fka}$ respectively  (see \cite[\S3.3]{Levin}). The parahoric $\ucalG$ corresponds to the image $\underline{J}$ of $J$ under this identification.

Let $W_{\underline{G}}$ denote the Iwahori Weyl group for $\uG$ and let $\{\mu\}=\{\mu_h\}$. Then the union of Schubert varieties appearing in $\calM$ is naturally indexed by $\Adm(\{\mu\})_J$, under an order preserving embedding $$\Adm(\{\mu\})_J\rightarrow W_{\underline{J}}\backslash W_{\underline{G}}/W_{\underline{J}}.$$ In particular the closure relations are given by the Bruhat order on $\Adm(\{\mu\})_J$. Under our assumption that $\calG$ is very special,  this ordering has the following alternative description. 

We let $\fks\in \calB(G,\brQ)$ denote the special vertex associated to $\calG$. Let $S$  be  a maximal $\brQ$-split torus of $G$ defined over $\bbQ_p$ such that $\fks\in\calA(G,S,\brQ)$ and $T$ the centralizer of $S$. Fix a Borel subgroup of $G$ defined over $\bbQ_p$ and assume we have identified $X_*(T)_I\otimes_{\bbZ}\bbR$ with $\calA(G,S,\brQ)$ via the choice of special vertex $\fks$.  We let $\mu\in X_*(T)_I$ be the image of a dominant representative of $\{\mu\}$ in $X_*(T)$. For $\lambda,\lambda'\in X_*(T)_I^+$, we write $\lambda\curlyeqprec\lambda'$ if $\lambda' -\lambda$ is an \textit{integral} linear combination of positive coroots in the reduced root system $\Sigma$ associated to $G$; we write $\lambda\prec \lambda'$ if in addition $\lambda\neq\lambda'$.
Then there is an identification $$W_{J}\backslash W/W_{J}\cong X_*(T)_I^+,$$ and the ordering $\curlyeqprec$ agrees with the Bruhat order on $W_{J}\backslash W/W_{J}$ under this identification (cf. \cite{Lu}). 
It follows that we have an identification $$\Adm_{G}(\{\mu\})_{J}=\{t_\lambda|\lambda\in X_*(T)_I^+,\ \lambda\curlyeqprec {\mu}\}.$$

For $\lambda\in X_*(T)_I$, we  write $\calM_k^\lambda$ for the open stratum corresponding to $t_\lambda\in \Adm_{G}(\{\mu\})_{J}$.  Then $\calM_k^\lambda$ is the $\underline{\calG}(k[[t]])$-orbit of an element $\underline{\dot{t}}_\lambda\in \underline{G}(k((t)))$ representing the image of $t_\lambda$ in $W_{\underline{J}}\backslash W_{\underline{G}}/W_{\underline{J}}$.
 It follows formally from the existence of the diagram (\ref{eqn: local model diagram scheme}) and the fact that $\calG^{\ad}$-orbits on $\calM_k$ and $\calG$-orbits on $\calM_k$ agree, that $\calS_{\rmK,k}$ admits a stratification by $\Adm_{G}(\{\mu\})_{J}$. This is known as the Kottwitz--Rapoport stratification and we write $\calS_{\rmK,k}^\lambda$ for the stratum corresponding to $t_{
\lambda}\in\Adm_{G}(\{\mu\})_{J}$. From the definition of this stratification, for $\overline{x}\in\calS_{\rmK}(k)$ the complete local ring of $\calS_{\rmK,k}^\lambda$ at $\overline x$ is identified with the complete local ring at a point $\overline x'\in\calM_k^\lambda(k)$.  Thus under our assumptions $\calM_k$ and $\calS_{\rmK,k}$ are normal schemes; cf. Theorem \ref{thm: density}.  Since $t_\mu\in\Adm_{G}(\{\mu\})_{J}$ is the unique maximal element, it follows that $\calM^{{\mu}}_k$ is contained in the smooth locus of $\calM$ and hence $\calS_{\rmK,k}^{{\mu}}$ is contained in the smooth locus of $\calS_{\rmK,k}.$

The strata $\calM_k^\lambda$ and $\calS_{\rmK,k}^\lambda$ are both defined over the field of definition of $\lambda\in W_{J}\backslash W/W_{J}$. In other words, if $n$ is the smallest positive integer such that $\sigma^n(\lambda)=\lambda$, then $\calM^\lambda_k$ and $\calS^\lambda_{\rmK,k}$ are both defined over $\bbF_{p^n}$; we write $\calM^\lambda $ and $\calS_{\rmK}^\lambda$ for the models over $\bbF_{p^n}$.

\subsubsection{} The key geometric property of the Kottwitz--Rapoport stratification  on $\calM_k$ that we will need is the following.

\begin{prop}\label{prop: map from curve}
	Let $y\in\calM^\lambda(\bbF_{q})$ with $\lambda\in \Adm_{G}(\{\mu\})_{J}$ and $\lambda \neq {\mu}$. There  exists a smooth, geometrically connected curve $C$ over $\bbF_{q}$ and a map $\phi:C\rightarrow \calM_{\bbF_q}$ such that 
	\begin{enumerate}[label=(\roman*)]
	\item There exists $y'\in C(\bbF_{q})$ such that $\phi(y')=y$.
	\item $\phi^{-1}(\calM_k^{\lambda'})$ is open and dense in $C$  for some $\lambda'\in\Adm_{G}(\{\mu\})_{J}$ with $\lambda\prec\lambda'.$
	\end{enumerate}
\end{prop}
\begin{remark}
Using an ampleness argument, it is easy to show that such a map always exists if we replace $\bbF_{q}$ by its algebraic closure $k$. The key property is that for $\calM$, this map exists without extending the residue field. By \cite[\S6]{Drinfeld}, there are normal and Cohen--Macaulay  schemes where this property fails.
\end{remark}
\begin{proof}[Proof of Proposition \ref{prop: map from curve}]
We first show using the $\underline{\G}$-action on $\calM$ that it suffices to consider the case 
$$y=\underline{\dot{t}}_\lambda\in \underline{G}(k((t)))/\underline{\G}(k[[t]]).$$ 

Let $\sigma_q$ denote the $q$-Frobenius; then since $y\in \calM^\lambda(\bbF_{q})$, we have $\sigma_q(\lambda)=\lambda$.  Therefore we may choose the lift $\underline{\dot{t}}_\lambda\in\underline{G}(\bbF_{q}((t)))$ so that $\underline{\dot{t}}_\lambda\in \calM^\lambda(\bbF_q)$. By Lemma \ref{lemma: rational G orbit} below, there exists $g\in \underline{\G}(\bbF_q[[t]])$ such that $g\underline{\dot{t}}_\lambda=y$ in $\mathcal{FL}_{\underline{\G}}$. Therefore if $C$ satisfies the conditions (i) and (ii) for the point $\underline{\dot{t}}_\lambda$, $gC$ satisfies (i) and (ii) for the point $y$. It therefore suffices to prove the case $y=\underline{\dot{t}}_\lambda$; we make this assumption from now on.
	
	Now since $\lambda\prec{\mu}$, by Stembridge's Lemma \cite[Lemma 2.3]{Ra1}, there exists a positive root $\alpha\in \Sigma$ such that $\lambda+\alpha^\vee\curlyeqprec{\mu}$. Since $\lambda,{\mu}\in X_*(T)_I^{\sigma_q}$, it follows that $$\lambda+\sigma^i_q(\alpha^\vee)\curlyeqprec {\mu}$$ for all $i$. If $\{\alpha, \sigma_q(\alpha),\dotsc,\sigma_q^{m-1}(\alpha)\}$ denotes the orbit of $\alpha$ under $\sigma_q$, it follows that $$\lambda':=\lambda+\sum_{i=0}^{m-1}\sigma_q^i(\alpha^\vee)\curlyeqprec{\mu},$$ 
	and hence $\lambda'\in\Adm_{G}(\{\mu\})_{J}$.	Now $\alpha$ determines a relative root $\widetilde{\alpha}$  of $\underline{G}$ over $\bbF_q((t))$ which we always take to be the short root; then either $2\widetilde{\alpha}$ is a relative root, or no rational multiple of $\ta$ is a relative root. We let  $U_{\widetilde{\alpha}}$ denote the relative root subgroup corresponding to $\widetilde{\alpha}$ and $\uG_{\widetilde{\alpha}}$ the simply connected covering of the (semi-simple) group generated by $U_{\widetilde{\alpha}}$ and $U_{-\widetilde{\alpha}}$; it is a reductive group over $\bbF_{q}((t))$. We will identify $U_{\widetilde{\alpha}}$ with the corresponding unipotent subgroup of $\uG_{\widetilde{\alpha}}$. The parahoric $\underline{\G}$ determines a parahoric model $\ucalG_{\widetilde{\alpha}}$ of $\underline{G}_{\widetilde \alpha}$ and there is a morphism $$\iota_{\widetilde{\alpha}}:\mathcal{FL}_{\ucalG_{\widetilde{\alpha}}}\rightarrow \mathcal{FL}_{\ucalG,\bbF_q}$$
	defined over $\bbF_q$, where $\mathcal{FL}_{\ucalG_{\widetilde{\alpha}}}$ is the partial affine flag variety associated to $\ucalG_{\widetilde{\alpha}}$. Then $\iota_{\widetilde{\alpha}}$ factors as $\mathcal{FL}_{\ucalG_{\widetilde{\alpha}}}\rightarrow \mathcal{FL}_{\ucalG'_{\widetilde{\alpha}}}\rightarrow \mathcal{FL}_{\ucalG,\bbF_q}$, where $\mathcal{FL}_{\ucalG'_{\widetilde{\alpha}}}$ is the corresponding  partial affine flag variety for the group generated by $U_{\widetilde{\alpha}}$ and $U_{-\widetilde{\alpha}}$. The first map identifies $\mathcal{FL}_{\ucalG_{\widetilde{\alpha}}}$ with the neutral connected component of $\mathcal{FL}_{\ucalG'_{\widetilde{\alpha}}}$ by \cite[\S6.a.1]{PR} and the  second is a proper monomorphism when restricted to a connected component.  It follows that $\iota_{\widetilde{\alpha}}$ is a closed immersion. We write $\calU_{\widetilde{\alpha}}$  (resp. $ \calU_{-\widetilde{\alpha}}
	$) for the group schemes over $\bbF_q[[t]]$ corresponding to $U_{\widetilde{\alpha}}(\bbF_q((t)))\cap \underline{\G}(\bbF_q[[t]]) $  (resp. $U_{-\widetilde{\alpha}}(\bbF_q((t)))\cap \underline{\G}(\bbF_q[[t]])$). Then we claim that for each positive $\alpha$, there exists a morphism $$f:\bbA^1_{\bbF_q}\rightarrow \mathcal{FL}_{\ucalG_{\widetilde{\alpha}}}$$ defined over $\bbF_q$ satisfying the following two conditions
	
	\begin{enumerate}[label=(\roman*')]
	\item $f(0)=\dot{e}$, where $\dot{e}$ is the base point in $\mathcal{FL}_{\ucalG_{\widetilde{\alpha}}}$.
	\item $f(\bbA^1_{\bbF_q}\backslash\{0\})\subset L^+\calU_{\widetilde{\alpha}}\underline{\dot{t}}_{\alpha^\vee}L^+\ucalG_{\widetilde{\alpha}}/L^+\ucalG_{\widetilde{\alpha}}$.
	\end{enumerate}

	Assuming the claim we may prove the proposition as follows. We consider the morphism $$\phi:\bbA_{\bbF_q}^1\rightarrow \mathcal{FL}_{\ucalG},\ \ \ x\mapsto \underline{\dot{t}}_{{\lambda}}(\iota_{\widetilde{\alpha}}\circ f)(x),$$ in other words we translate the composition $\iota_{\widetilde{\alpha}}\circ f$ by $\underline{\dot{t}}_\lambda$. Then condition (i) follows from (i') and condition (ii) follows from (ii') using the fact that $\lambda$ is dominant.
	
	It remains to prove the existence of $f$ satisfying (i') and (ii'). We will construct $f$ explicitly using a presentation of the group $\uG_{\widetilde{\alpha}}$; it turns out that by \cite[\S4.1.4]{BT2} there are essentially three distinct cases to consider which we now describe. 
	
	If $2\widetilde{\alpha}$ is not a relative root then there is an identification $$\uG_{\widetilde{\alpha}}\cong \mathrm{Res}_{K/\bbF_q((t))}\SL_2$$ where $K$ is some finite separable extension of $\bbF_q((t))$ and the parahoric $\ucalG_{\widetilde{\alpha}}$ is characterized by the property $$\ucalG_{\widetilde{\alpha}}(k[[t]])=\SL_2(\calO_K\otimes_{\bbF_q[[t]]}k[[t]]).$$ 
	
	If $2\widetilde{\alpha}$ is also a relative root, then there is an identification $$\uG_{\widetilde{\alpha}}\cong \mathrm{Res}_{K/\bbF_q((t))}\SU_3$$where $K/\bbF_{q}((t))$ is finite separable and $\SU_3$ is the special unitary group associated to a hermitian space over a (separable)\footnote{Since we have assumed $p>2$, this is automatic.} quadratic extension $K'/K.$ We recall the presentation of the $K$-group $\SU_3$ in \cite[Example 1.15]{Ti1}. We let $\tau\in \mathrm{Gal}(K'/K)$ denote the non-trivial element and we consider the hermitian form on $K'^3$ given by $$\langle(x_{-1},x_0,x_1),(y_{-1},y_0,y_1)\rangle=\tau(x_{-1})y_1+\tau(x_0)y_0+\tau(x_1)y_{-1}.$$
	The group $\SU_3$  is the special unitary group attached to this form. For $c,d\in K'$ such that $\tau(c)c+d+\tau(d)=0$, we define $\bfu_+(c,d),\bfu_{-}(c,d)\in \mathrm{SU}_3(K)$ by $$\bfu_{\pm}(c,d)=I_3+(g_{rs})$$ where $I_3$ is the identity matrix and $(g_{rs})$ is the matrix with entries $g_{\mp1,0}=-\tau(c)$, $g_{0,\pm 1}=c$, $g_{\mp1,\pm 1}=d$ and $g_{rs}=0$ otherwise. The root subgroups $U_{\pm\ta}$ are then given by $$U_{\pm\widetilde{\alpha}}(K)=\{\bfu_{\pm}(c,d)|c,d\in K', \tau(c)c+\tau(d)+d=0\}.$$

	Then we may consider the parahoric $$\ucalG_{\widetilde{\alpha}}(\bbF_q[[t]])= \SU_3(K)\cap \GL_3(\calO_{K'});$$ we call this the standard parahoric. 
	
	When $K'/K$ is unramified this is the only very special parahoric (up to conjugacy). When $K'/K$ is ramified, there is another conjugacy class of very special parahorics in addition to the standard parahoric which we shall call the non-standard parahoric. We let $u'$ be a uniformizer of $K'$  such that $\tau(u')=-u'$ and we define $s\in \GL_3(K')$ to be the element $\mathrm{diag}(1,1,u')$. Then the non-standard parahoric $\ucalG_{\widetilde{\alpha}}$ is given by $$\ucalG_{\widetilde{\alpha}}(\bbF_q[[t]])=\SU_3(K)\cap s\GL_3(\calO_{K'})s^{-1}.$$
	
	We label the cases as follows.
	
	Case (1): $2\widetilde{\alpha}$ is not a root, $\uG_{\widetilde{\alpha}}\cong \mathrm{Res}_{K/\bbF_q((t))}\SL_2$ and $\ucalG_{\widetilde{\alpha}}(\bbF_q[[t]])=\SL_2(\Ok_K)$.
	
	Case (2): $2\widetilde{\alpha}$ is  a root, $ \uG_{\widetilde{\alpha}}\cong \mathrm{Res}_{K/\bbF_q((t))}\SU_3$ and $\ucalG_{\widetilde{\alpha}}$ is the standard parahoric.
	
	Case (3): $2\widetilde{\alpha}$ is a root, $ \uG_{\widetilde{\alpha}}\cong \mathrm{Res}_{K/\bbF_q((t))}\SU_3$ with $K'/K$ ramified and $\ucalG_{\widetilde{\alpha}}$ is the non-standard parahoric.
	
	We now proceed with the construction of $f$ in each of the three cases.
	
	Case (1). In this case the isomorphism
	$\uG_{\widetilde{\alpha}}\cong\mathrm{Res}_{K/\bbF_{q}((t))}\SL_2$  	induces  identifications $$\bfu_{\pm}:\mathrm{Res}_{K/\bbF_{q}((t))}\bbG_a\xrightarrow{\sim}U_{\pm\widetilde{\alpha}}.$$
	Let $u$ be a uniformizer of $K$; then we may define a map $$f:\bbA^1_{\bbF_{q}}\rightarrow \mathcal{FL}_{\underline{\G}_{\widetilde{\alpha}}},\ \ \ \ x\mapsto \bfu_{-}(u^{-1}x).$$
	Clearly (i') is satisfied, and a simple calculation in $\SL_2$ shows that for $0\neq x$, we have $$\bfu_{-}(u^{-1}x)\in \bfu_{+}(ux^{-1})\underline{\dot{t}}_{\alpha^\vee}L^+\ucalG_{\widetilde{\alpha}}$$
	so that (ii')  also holds.

	Case (2). Recall in this case,  the parahoric $\ucalG_{\widetilde{\alpha}}$ is characterized by $\ucalG_{\widetilde{\alpha}}(\bbF_q[[t]])=\SU_3(K)\cap \GL_3(\calO_{K'})$. We define $$f:\bbA_{\bbF_q}^1\rightarrow \mathcal{FL}_{\underline{\G}_{\widetilde{\alpha}}}, \ \ x\mapsto \bfu_{-}(0,u'^{-1}x),$$ where we recall that $u'\in K'$ is a uniformizer with $\tau(u')=-u'$.
	A calculation using the presentation recalled above shows that for $x\neq0$, we have $$\bfu_{-}(0,u'^{-1}x)\in \bfu_+(0,u'x^{-1})\underline{\dot{t}}_{\alpha^\vee}L^+\ucalG_{\widetilde{\alpha}};$$ as in Case (1), it follows that (i') and (ii') are satisfied.
	
  Case (3). Recall $K'/K$ is ramified and $\ucalG_{\widetilde{\alpha}}(\bbF_q[[t]])=\SU_3(K)\cap s\GL_3(\calO_{K'})s^{-1}$. We consider the map $$\bbA_{\bbF_q}^1\rightarrow \mathcal{FL}_{\ucalG_{\widetilde{\alpha}}}, \ \ x\mapsto \bfu_{-}(x,-\frac{x^2}{2}).$$
	Then in the presentation  above, we have  that $\dot{t}_{\alpha^\vee}^{-1}\mathbf{u}_{+}(-2x^{-1},2x^{-2})^{-1}\mathbf{u}_-(x,-x^2/2)$ is equal to 
	$$\left(\begin{smallmatrix}u'^{-1} & 0&0\\0&-1 & 0\\ 0 & 0 & -u'
	\end{smallmatrix}\right)\left(\begin{smallmatrix} 1&-2x^{-1} & -2x^{-2}\\0&1 & 2x^{-1}\\ 0 & 0 & 1
	\end{smallmatrix}\right)\left(\begin{smallmatrix}1 & 0 &0\\x&1 & 0\\ -x^2/2 & -x & 1
	\end{smallmatrix}\right) =\left(\begin{smallmatrix}0 & 0&-u'^{-1} 2x^{-2}\\0 &1 & -2x^{-1}\\ u'x^2/2 & u'x & -u'
	\end{smallmatrix}\right).$$	
	This lies in the parahoric $\underline{\calG}_{\widetilde{\alpha}}$, and hence we have
$$\bfu_{-}(x,-\frac{x^2}{2})\in \bfu_+(-2x^{-1},2x^{-2})\dot{t}_{\alpha^\vee}L^+\ucalG_{\widetilde{\alpha}}.$$
As in the previous two cases it follows that (i') and (ii') are satisfied.
\end{proof}

\begin{lemma}\label{lemma: rational G orbit}
	Let $y\in \calM^{\lambda}(\bbF_{q})$ and assume $\underline{\dot{t}}_\lambda\in\underline{G}(\bbF_q[[t]])$. Then there exists $g\in \underline{\G}(\bbF_{q}[[t]])$ such that $g\underline{\dot{t}}_\lambda L^+\underline{\G}=y$ in $\mathcal{FL}_{\underline{\G}}$.
\end{lemma}
\begin{proof} By definition, there exists $h\in \ucalG(k[[t]])$ such that $h\underline{\dot{t}}_\lambda=y$. We consider the subgroup $$\ucalG(k[[t]])\cap \underline{\dot{t}}_\lambda\ucalG(k[[t]])\underline{\dot{t}}_\lambda^{-1}\subset \uG(k((t)));$$
	it is the intersection of the kernel of the Kottwitz homomorphism $\widetilde{\kappa}_{\underline{G}}$ and the stabilizer of a bounded subset of the building $\calB(\underline{G},k((t)))$. Thus by \cite[Prop. 3 and Remark 4]{HaRa}, it arises  as the $k$-points of a smooth group scheme $\underline{\calK}_{\lambda}$ defined over $\bbF_{q}[[t]]$ with connected special fiber.
	
	The element $h$ is defined up to right multiplication by $\underline{\calK}_\lambda(k[[t]])$; hence since $\sigma_q(y)=y$, we have $\sigma_q(h)=hk$ for some $k\in \underline{\calK}_\lambda(k[[t]])$. By Lang's theorem applied to $\underline{\calK}_\lambda$, there exists $k_1\in\underline{\calK}_\lambda(k[[t]])$ such that $g:=hk_1$ is fixed by $\sigma_q$, and we have $g\dot{t}_\lambda=y$ in $\mathcal{FL}_{\ucalG}$.
\end{proof}

\subsubsection{}Using Theorem \ref{eqn: local model diagram scheme}, we may deduce the following result about the local structure of the Shimura stack $\calS_{\rmK}$.

\begin{cor}\label{cor: map from smooth stack to Shimura var}
		Let $x\in\calS_{\rmK}^\lambda(\bbF_{q})$ with $\lambda\in \Adm_{G}(\{\mu\})_{J}$ and $\lambda \neq {\mu}$. There  exists a  smooth, geometrically connected curve $C'$ over $\bbF_{q}$ and a map $\phi':C'\rightarrow \calS_{\bbF_q}$ such that 
	\begin{enumerate}[label=(\roman*)]
	\item There exists $x'\in C'(\bbF_{q})$ such that $\phi'(x')=x$.
	\item $ \phi'^{-1}(\calS_{\rmK,k}^{\lambda'})\subset C'$ is an open dense subscheme for some $\lambda'$
	$\in\Adm_{G}(\{\mu\})_{J}$ with $\lambda\prec\lambda'$.
	\end{enumerate}
\end{cor}

\begin{proof} We write
	\[\xymatrix{ & \widetilde{\calS_{\rmK}}\ar[dr]^{q_{k_E}}\ar[dl]_{\pi_{k_E}}& \\ 
	\calS_{\rmK} & & \calM}\] for the special fiber of (\ref{eqn: local model diagram scheme}). Since $\pi_{k_E}$ is  a torsor for the smooth  group scheme $\G_{\mathrm{ad},k_E}$ with connected special fiber, the point $x$ lifts to a point $\widetilde{x}\in\widetilde{\calS_{\rmK}}(\bbF_q)$ and we write $y$ for its image in $\calM(\bbF_q)$. By definition of the stratification on $\calS_{\rmK}$, we have $y\in\calM^\lambda(\bbF_q)$. We apply Proposition \ref{prop: map from curve} to $y$ to obtain a map $\phi':C\rightarrow \calM_{\bbF_q}$ satisfying (i) and (ii) in Proposition \ref{prop: map from curve} for some $\lambda'\in\Adm_{G}(\{\mu\})_{J}$ with $\lambda\prec\lambda'$; we let $y'\in C(\bbF_q)$ mapping to $y$. 

Consider the pullback $\widetilde{\calS}_{\rmK,\bbF_q}{\times}_{\calM_{\bbF_q}} C$ which is a smooth stack over $\bbF_q$.  By \cite[Th\'eor\`eme 6.3]{LMB}, there exists  a smooth scheme $Y/\bbF_q$ and a smooth map $Y\rightarrow \widetilde{\calS}_{\rmK,\bbF_q}{\times}_{\calM_{\bbF_q}} C$ defined over $\bbF_q$ such that $\widetilde{x}$ lies in the image of a point $\widetilde{y}\in Y(\bbF_q)$. Now let $Y^{\lambda'}$ denote the preimage of $\calM^{\lambda'}$ in $Y$; by the assumption on $C$, it is a dense open subscheme of $Y$. By \cite[Corollary 3.4]{Poonen}, there exists a smooth geometrically connected curve $C'\subset Y$ such that $\widetilde{y}\in C'(\bbF_q)$ and $C'\cap Y^{\lambda'}\neq\emptyset$ so that the preimage of $Y^{\lambda'}$ in $C'$ is open and dense. We write $\phi':C'\rightarrow \calS_{\rmK,\bbF_q}$ for the composition $$C'\rightarrow Y\rightarrow \widetilde{\calS}_{\rmK,\bbF_q}{\times}_{\calM_{\bbF_q}} C\rightarrow \widetilde{\calS}_{\rmK,\bbF_q}\rightarrow \calS_{\rmK,\bbF_q} .$$
Then setting $x'=\widetilde{y}\in C'(\bbF_q)$, we have $ \phi'(x')=x$, so (i) is satisfied, and property (ii) follows by the construction.
	\end{proof}

\subsection{Compatible local systems and $\ell$-independence}

\subsubsection{}We recall the theory of compatible local systems. Let $X$ be a normal scheme  over $\bbF_q$ where $q$ is a power of $p$  and let $\calL_\ell$ be a $\overline{\bbQ}_\ell$-local system (lisse sheaf) on $X$. For $x\in X(\bbF_{q^n})$, we write $\mathrm{Frob}_x$ for the local Frobenius automorphism acting on the stalk $\mathcal{L}_{\ell,\overline{x}}$ of $\mathcal{L}_\ell$ at a geometric point $\overline{x}$ lying over $x$. Suppose that for every closed point $x\in X(\bbF_{q^n})$ the characteristic polynomial $\det(1-\mathrm{Frob}_xt|\mathcal{L}_{\ell,\overline{x}}),$ has coefficients in a number field $E\subset \overline{\bbQ}_\ell$ (this is conjectured to be the case if $\calL_{\ell}$ has determinant of finite order).  Let $\ell'$ be a prime not equal to $p$ and $\lambda':E\hookrightarrow\overline{\bbQ}_{\ell'}$ an embedding of fields. A  $\overline{\bbQ}_{\ell'}$-local system $\mathcal{K}_{\ell'}$ is said to be \emph{$\lambda'$-compatible} for $\mathcal{L}_\ell$ if for  every closed point $x\in X(\bbF_{q^n})$, the characteristic polynomial  $\det(1-\mathrm{Frob}_xt|\mathcal{K}_{,\ell',\overline{x}})$ has coefficients in $E$ and there is an equality 
$$\det(1-\mathrm{Frob}_xt|\mathcal{L}_{\ell,\overline{x}})=\det(1-\mathrm{Frob}_xt|\mathcal{K}_{\ell',\overline{x}})\in E[t].$$

The existence of $\lambda'$-compatible local systems over smooth curves is due to   Lafforgue \cite[Th\'eor\`eme VII.6]{Laf} (under the assumption of finite determinant), and the case of  smooth schemes is due to Drinfeld \cite[Theorem 1.1]{Drinfeld}.

\subsubsection{} We now continue with the notations of \S\ref{sec: Frob conj classes}. For the rest of this section, it will be convenient to fix a Hodge embedding $\iota: (\bfG,X) \rightarrow (\bGSp(V),S^\pm)$ as in \S\ref{sec: canonical liftings 1}.

The element $\gamma_{y,\ell}\in\Conj_{\bfG}(\bbQ_\ell)$ arises as an element of $\mathrm{Conj}_{\bfG}(\overline{\bbQ})$. 
Indeed the image of $\gamma_{y,\ell}$ in 
$\mathrm{Conj}_{{\bGL(V)}}(\bbQ_\ell)$ under the map induced by $\iota$ lies in $\mathrm{Conj}_{\bGL(V)}(\bbQ)$ 
since it corresponds to the action of Frobenius on the $\ell$-adic Tate module of an abelian variety.  
Since $\Conj_{\bfG}\rightarrow \Conj_{\bGL(V)}$ is a finite map, $\gamma_{y,\ell}\in \Conj_{\bfG}(\overline{\bbQ})$. 
Similarly if $\ell'\nmid p\ell$ is another prime, $\gamma_{y,\ell'}$ arises as an element of $\Conj_{\bfG}(\overline{\bbQ})$. 

We let $F$ be a finite extension of $\bbQ$ such that $\gamma_{y,\ell},\gamma_{y,\ell'}\in\Conj_{\bfG}(F)$; such an extension exists since $\Conj_{\bfG}$ is a $\bbQ$-variety. Let $\lambda,\lambda'$ be the two places over $F$ induced by the fixed embeddings $i_\ell:\overline{\bbQ}\rightarrow\overline{\bbQ}_\ell$ and $i_{\ell'}:\overline{\bbQ}\rightarrow\overline{\bbQ}_{\ell'}$. We take $\vartheta:\bfG_F\rightarrow \mathbf{GL_n}_F$ to be an arbitrary representation over $F$ (not necessarily coming from the Hodge embedding $\iota$); then the $\bfG(\bbQ_\ell)$-local system $\bbL_\ell$ induces an $F_\lambda$-adic local system $\calL_\ell$ over $\calS_{\rmK}$. Similarly we obtain an $F_{\lambda'}$-adic local system $\calL_{\ell'}$. 

\begin{lemma}\label{lemma: Frob l adic unit}
	For any closed point $x\in \calS_{\rmK}(\bbF_{q})$, the eigenvalues of $\mathrm{Frob}_x$ acting on $\mathcal{L}_{\ell,\overline{x}}$ are  $\ell$-adic units.
\end{lemma}
\begin{proof} It suffices to prove this for a single faithful representation of $\bfG$. For the representation $\bfG\rightarrow \mathbf{GL}(V)$ induced by $\iota$, the action of $\mathrm{Frob}_x$ on $\mathcal{L}_{\ell,\overline{x}}$ corresponds to the action of Frobenius on the $\ell$-adic Tate  module of an abelian variety and hence its eigenvalues are all $\ell$-adic units.
	\end{proof}

\subsubsection{}We let $\vartheta(\gamma_{y,\ell})\in\Conj_{\mathbf{GL_n}}(F)\subset\Conj_{\mathbf{GL_n}}(F_\lambda)$ denote the image of the conjugacy class of $\mathrm{Frob}_y$ under $\vartheta$  and we similarly define $\vartheta(\gamma_{y,\ell'})\in\Conj_{\mathbf{GL_n}}(F)\subset\Conj_{\mathbf{GL_n}}(F_{\lambda'})$.

\begin{prop}\label{prop: indep in GL} For any representation $\vartheta:\bfG_F\rightarrow \mathbf{GL}_{nF}$, we have
 $$\vartheta(\gamma_{y,\ell})=\vartheta(\gamma_{y,\ell'})$$ in $\mathrm{Conj}_{\mathbf{GL_n}}(F).$
\end{prop}

\begin{proof} Let $C$ be a smooth geometrically connected curve and $\psi:C\rightarrow \calS_{\rmK,\bbF_q}$ a morphism defined over $\bbF_q$ such that there exists a point $x\in C(\bbF_q)$ with $\psi(x)=y$. We first show that if the proposition holds for the image  under $\psi$ of a Zariski open and dense set $U\subset C$, then it holds for $y$. 
	
	We write $\calL_{\ell}^{C}$ (resp. $\calL_{\ell'}^{C}$) for the pullback $\psi^*\calL_\ell$ of $\calL_\ell$ (resp. $\psi^*\calL_{\ell'}$ of $\calL_{\ell'}$) to ${C}$.
	By  Lemma \ref{lemma: Frob l adic unit}, $\calL^C_{\ell}$ satisfies the conditions in Chin's refinement of Lafforgue's Theorem  \cite[Theorem 4.6]{Chin}. Thus upon enlarging $F$, there exists a  $\overline{\bbQ}_{\ell'}$-local system $\calK^{C}_{\ell'}$  over $C$ which is $\lambda'$-compatible for $\calL^{C}_{\ell}$.

For any  closed point  $x\in C(\bbF_{q^s})$, 	
	$$\det(1-\mathrm{Frob}_xt|\calL_{\ell,\bar x}^{C})=\det(1-\mathrm{Frob}_xt|\calK_{\ell',\bar x}^{C}) \in F[t].$$	
By assumption, for any closed point $x\in U(\bbF_{q^s})$, we have 
	$$\det(1-\mathrm{Frob}_xt|\calL_{\ell',\bar x}^{C})=\det(1-\mathrm{Frob}_xt|\calL_{\ell,\bar x}^{C}) =\det(1-\mathrm{Frob}_xt|\calK_{\ell', \bar x}^{C}).$$
	Therefore, by the Chebotarev density Theorem, the semisimplifications of $\calK^{C}_{\ell'}$ and  $\calL^{C}_{\ell'}$ are isomorphic, and hence 
	$$\vartheta(\gamma_{y,\ell})=\det(1-\mathrm{Frob}_yt|\calL_{\ell,\bar y}^{C})=\det(1-\mathrm{Frob}_yt|\calL_{\ell',\bar y}^{C})=\vartheta(\gamma_{y,\ell'})$$ as desired.
	
	We now show that the Proposition holds for $y\in \calS_{\rmK}^\mu(\bbF_q)$; we recall that $\calS_{\rmK}^\mu$ is the open Kottwitz--Rapoport stratum and is smooth. Using the same argument as in the proof of \ref{cor: map from smooth stack to Shimura var} (i.e. applying \cite[Th\'eor\`eme 6.3]{LMB} and \cite[Corollary 3.4]{Poonen}), we may find a smooth geometrically connected curve $C$ over $\bbF_q$ and a map $\psi:C\rightarrow \calS_{\rmK,\bbF_q}^\mu$ defined over $\bbF_q$ such that there exists a point $x\in C(\bbF_q)$ with $\psi(x)=y$ and such that the preimage $U:=\psi^{-1}(\calS_{\rmK,[b]_\mu})\subset C$ of the $\mu$-ordinary locus is open and dense. By Corollary \ref{cor: l indep mu ordinary}, the Proposition holds for points $y'$ lying in the image of $U$, and hence it holds for $y$ by the above argument.

	Finally we show that the Proposition holds for all $y\in \calS_{\rmK}(\bbF_q)$. We assume $y\in \calS_{\rmK}^\nu(\bbF_q)$ and we proceed by descending induction on $\nu$; the case of the maximal element $\nu=\mu$ was proved above. Now suppose the result is true for all $\nu'\succ\nu$.
 	Let $\psi:C\rightarrow \calS_{\rmK,\bbF_q}$ be a map as in Corollary \ref{cor: map from smooth stack to Shimura var} where $C$ is a smooth geometrically connected curve over $\bbF_q$. We let $U\subset C$ denote the preimage of $\bigcup_{\nu\prec\nu'}\calS^{\nu'}_{\rmK,\bbF_q}$ which is Zariski open and dense. By induction hypothesis, the proposition holds for the image of $U$, hence it holds for $y$.
\end{proof}

\subsubsection{}We may now prove Theorem \ref{thm: l indep full}.

\begin{proof}[Proof of Theorem \ref{thm: l indep full}] For all $\ell,\ell' \neq p,$ and $\vartheta$ as above, we have 
$\vartheta(\gamma_{y,\ell})=\vartheta(\gamma_{y,\ell'})$ by Proposition \ref{prop: indep in GL}. This implies that 
$\gamma_{y,\ell}=\gamma_{y,\ell'}\in\Conj_{\bfG}(\overline{\bbQ}),$ by a result of Steinberg \cite[6.6]{Steinberg:regular}.
Hence, there exists $\gamma_y\in \Conj_{\bfG}(\overline{\bbQ})$ such that $\gamma_{y}=\gamma_{y,\ell}$ for all $\ell\neq p$. It suffices to show $\gamma_y$ is defined over $\bbQ$.

Since $\Conj_{\bfG}$ is a $\bbQ$-variety, the residue field of the point $\gamma_y$ is a finite extension $F/\bbQ$. Since $\gamma_y\in \Conj_{\bfG}(\bbQ_\ell)$ for all $\ell$, each finite prime of $\bbQ$ has a split prime in $F$ above it; hence the Chebotarev density theorem implies $\gamma_y\in \Conj_{\bfG}(\bbQ)$. Indeed let $F'/\bbQ$ be the Galois closure of $F.$ 
Then for every prime $\ell\neq p$, there exists $l$ a prime of $F'$ above $\ell$ such that the Frobenius $\mathrm{Frob}_{l}$ lies in 
$\mathrm{Gal}(F'/F) \subset \mathrm{Gal}(F'/\bbQ).$
It follows that $\mathrm{Gal}(F'/F)$ intersects every conjugacy class of $\mathrm{Gal}(F'/\bbQ)$ and hence these groups are equal.
	\end{proof}
\begin{remark}The application of \cite[6.6]{Steinberg:regular} in the previous theorem is one of the reasons we  obtain $\gamma$ as an element of $\Conj_{\bfG}(\bbQ)$, as opposed to an element of $\bfG(\bbQ)$.
	
	\end{remark}

\section{Conjugacy class of Frobenius for abelian varieties}
We apply the results of \S\ref{sec: l indep for Shimura var} to prove our main result concerning abelian varieties.

\subsection{Mumford--Tate groups}\label{subsec:Mumford-Tate groups}

\subsubsection{} Let $A$ be an abelian variety over a number field $\rmE$. Recall we have fixed an embedding $i_\infty:\overline{\bbQ}\rightarrow\bbC$; using this we may consider $\rmE$ as a subfield of $\bbC$.  We write $V_B$ for the Betti cohomology $\rmH^1_B(A(\C),\Q)$ which is equipped with a Hodge structure of type $((0,-1),(-1,0))$. This Hodge structure is induced by a morphism $$h:\mathbb{S}:=\text{Res}_{\C/\R}\mathbb{G}_m\rightarrow \GL(V_B)$$
We write $$\mu:\C^\times\xrightarrow{z\mapsto (z,1)}\C^\times\times c^*(\C^\times)\xrightarrow{h} {\GL}(V_B\otimes\C)$$ for the Hodge cocharacter. 

\begin{definition} 
	The Mumford--Tate group $\bfG$ of $A$ is  the smallest algebraic subgroup of $\GL(V_B)$ defined over $\Q$ such that  $\bfG(\bbC)$ contains the image of $\mu$. \end{definition}
The group $\bfG$ can also be characterized as the algebraic subgroup of $\GL(V_B)$ that stabilizes all Hodge cycles of type (0,0)  on the tensor spaces $V_B^{\otimes r}\otimes (V_B^\vee)^{\otimes r}$ for $r\in \bbZ_{\geq0}$;  it is known that $\bfG$ is  a connected  reductive group.  

We remark that $\bfG$ depends on the embedding $\rmE\hookrightarrow \bbC$; if $\bfG_1$ is the group defined by a different embedding then there is a canonical inner twisting $\bfG_{\overline{\bbQ}}\cong\bfG_{1,\overline{\bbQ}}$ induced by the  torsor of tensor preserving isomorphisms between the Betti cohomology groups (see \cite[Proof of Theorem 3.8]{De1} for the construction of this torsor).

\subsubsection{} For a prime number $\ell$, we write $T_\ell A$ for the Tate module of $A$. The action of the absolute Galois group 
$\Gamma_{\rmE}:=\mathrm{Gal}(\overline{\rmE}/\rmE)$ on $T_\ell A^\vee$ gives rise to a representation $\rho_\ell:\Gamma_{\rmE}\rightarrow \mathbf{GL}(T_\ell A^\vee)$ and the Betti-\'etale comparison gives us a canonical isomorphism 
$$\rmH^1_B(A(\C),\Q)\otimes_{\bbQ}\Q_\ell§\cong T_\ell A^\vee\otimes_{\bbZ_\ell}\Q_\ell.$$
Deligne's theorem that Hodge cycles are absolutely Hodge \cite{De1}, implies that upon replacing $\rmE$ by a finite extension, the map 
$\rho_\ell$ factors through $\bfG(\Q_\ell)$; see \cite[Remarque 1.9]{Noot}. In fact this condition does not depend on $\ell.$

\begin{lemma}\label{indepoflfac} The representation $\rho_\ell$ factors through $\bfG(\Q_\ell)$ for some prime $\ell,$ if and only if it factors through 
$\bfG(\Q_\ell)$ for all primes $\ell.$
\end{lemma}
\begin{proof} The subgroup $\bfG \subset \GL(V_B)$ is the stabilizer of a collection of Hodge cycles $(s_{\alpha})_{\alpha}.$ 
We consider the $\ell$-adic components $(s_{\alpha,\ell})_{\ell},$ as in \S\ref{subsubsec:hodgecycles}. 
For $\sigma \in \Gamma_{\rmE},$ $(\sigma(s_{\alpha,\ell}))_{\ell},$ is again a Hodge cycle, by Deligne's theorem \cite[Theorem 2.11]{De1}. 
In particular, if $(\sigma(s_{\alpha,\ell}))_{\ell},$ and $(s_{\alpha,\ell})_{\ell}$ have equal components at some prime $\ell,$ then 
they are equal. 
\end{proof}
The lemma shows that the condition that $\Gamma_{\rmE}$ fixes $(s_{\alpha,\ell})_{\alpha}$ pointwise does not depend on $\ell.$ This condition is equivalent to asking that $\Gamma_{\rmE}$ maps to $\bfG(\Q_\ell).$
\subsubsection{}\label{subsubsec:notation} We replace $\rmE$ by the smallest extension such that $\Gamma_{\rmE}$ maps to $\bfG(\Q_\ell),$ and 
we write $\rho_\ell^{\bfG}$ for the induced map $\Gamma_{\rmE}\rightarrow \bfG(\Q_\ell)$ and $\iota_\ell$ for the inclusion  $\bfG(\Q_\ell)\rightarrow \mathbf{GL}(T_\ell A^\vee)$.

Let $v$ be a prime of $\rmE$ lying above a prime $p$ such that  $A$ has good reduction at $v.$ Upon modifying the embedding $i_p:\overline{\bbQ}\rightarrow \overline{\bbQ}_p$ fixed in \S\ref{subsec: integral models Hodge type preamble}, we may assume that $v$ is induced by $i_p$. We write $E = \rmE_v,$ and we let $\bbF_q$ denote the residue field of $\rmE$ at $v.$  For $\ell\neq p$ a prime, $\rho_\ell$ is unramified at $v$. Let $\Fr_v$ be a geometric Frobenius element at $v$, we write $\gamma_\ell(v)=\chi_{\bfG}(\rho_\ell^\bfG(\Fr_v))\in \Conj_{\bfG}(\Q_\ell)$ for the conjugacy class of $\rho_\ell^\bfG(\mathrm{Fr}_v)$ which only depends on $v$ and not the choice of Frobenius element. We write $P_{v,\ell}(t)$ for the characteristic polynomial of $\Fr_v$ acting on $T_\ell A^\vee$, which has coefficients in $\Z$ and is independent of $\ell$.

\subsubsection{}\label{subsubsec: restriction of scalars} We will make use of the following auxiliary construction. 
Let $\rmF/\bbQ$ be a totally real field, and let $\bfH':=\text{Res}_{\rmF/\Q}\bfG_{\mathrm{F}}$. 
There is a canonical inclusion $\bfG\hookrightarrow \bfH'$. 	We let $(V,\psi)$ be the symplectic space corresponding to $\rmH_1(A(\bbC),\Q)$ where $\psi$ is a Riemann form for $A$ and $\bfG\rightarrow \mathbf{GSp}(V)$ is  the natural map. We let $W$ denote the symplectic space over $\Q$ whose underlying vector space is $V\otimes_{\Q}\rmF$ and whose alternating form $\psi'$ is given by the composition
$$W\times W\xrightarrow{\psi\otimes_{\Q}\rmF}\rmF\xrightarrow{\text{Tr}_{\rmF/\Q}}\Q.$$

Let $c_{\bfG}:\bfG\rightarrow \mathbb{G}_m$ denote the restriction of the multiplier homomorphism $c:\mathbf{GSp}(V)\rightarrow \mathbb{G}_m$ to $\bfG$. We  form the fiber product

$$\xymatrix{\bfH \ar[rr]\ar[d] && \mathbb{G}_m \ar[d]_{\Delta}\\
	\bfH' \ar[rr]^{\!\!\!\!!\!\!\mathrm{Res}_{\rmF/\Q}c_{\bfG}} &&  \mathrm{Res}_{\rmF/\Q}\mathbb{G}_m
}$$
where the map $\Delta$ is the diagonal map. Then $\bfH$ is an extension of $\bbG_m$ by the group $\Res_{\rmF/\bbQ}\bfG^c_{\rmF}$, where $\bfG^c\subset \bfG$ is the subgroup generated by $\bfG^{\der}$ and the largest compact subtorus of the center of $\bfG$; see \cite[Lemma 7.2.5]{KPZ}. Thus $\bfH$ is a connected reductive group over $\bbQ$. The inclusion $\bfG\hookrightarrow \bfH'$ factors through $\bfH$ and we let $h'$ denote the composition $$\mathbb{S}\xrightarrow{h} \bfG_{\R}\rightarrow \bfH_{\R}.$$ 
Write $X$ for the $\bfG(\R)$-conjugacy class of $h$ and $X_{\bfH}$ for the $\bfH(\R)$-conjugacy class of $h'$. 

Consider the composition $$\iota':\bfH'\xrightarrow{\mathrm{Res}_{\rmF/\Q}\iota}\mathrm{Res}_{\rmF/\Q}\mathbf{GSp}(V)\xrightarrow{f} \mathbf{GL}(W)$$
where $f$ is induced by the forgetful functor from $\rmF$-vector spaces to $\Q$-vector spaces. It is easy to see that the restriction of $\iota'$ to $\bfH$ factors through $\mathbf{GSp}(W)$, and we also denote by $\iota'$ the induced map. We write $S'^\pm$ for the Siegel half space corresponding to $W$. One checks easily that $(\bfG, X),$ and $(\bfH,X_{\bfH})$ are Shimura data, and that we have embeddings of Shimura data
$$ (\bfG, X) \hookrightarrow (\bfH, X_{\bfH}) \hookrightarrow (\mathbf{GSp}(W), S'^{\pm}).$$

\subsection{The main theorem} \label{sec: main thm}
We now prove our main theorem (cf. Theorem \ref{introthm: main}).  We need the following  preliminary result.

\begin{lemma}\label{lem:elementlevelstr} Let $G$ be a connected reductive group over $\Q_p.$ 
If $g \in G(\Q_p)$ lies in some compact open subgroup of $G(\Q_p),$ then there exists a finite extension $F/\Q_p$  over which $G$ splits and such that 
$g$ lies in the parahoric subgroup of $G(F)$ associated to a
 special vertex in the building $\calB(G,F).$
\end{lemma}
\begin{remark}Note that if $G$ splits over $F$, the notion of special vertex, very special vertex, and hyperspecial vertex in $\calB(G,F)$ all coincide.
\end{remark}

\begin{proof} Write $g=g_sg_u$ for the Jordan decomposition of $g$ so that $g_s$ is semisimple and $g_u$ is unipotent. Since $g$ lies in a compact open subgroup of $G(\Q_p)$, $g$ is power bounded and hence $g_s$ and $g_u$ are power pounded. Let $T\subset G$ be a maximal torus defined over $\bbQ_p$ such that $g_s\in T(\Q_p)$. We will take $F$ to be the splitting field of $T$. 

Since $g_s\in T(F)$ is power bounded, it is contained in $\calT_{F,0}(\Ok_F)$ where $\calT_{F,0}$ is the connected N\'eron model  for the base change $T_F$. If we let $\calA(G,T,F)\subset \calB(G,F)$ be the apartment corresponding to $T_F$, then $g_s$ acts trivially on $\calA(G,T,F)$.

Now $g_u\in U(F)$ where $U$ is the unipotent radical of some Borel subgroup $B$ of $G_F$ containing $T$. 
Let $\fks\in \calA(G,T,F)$ be any special vertex and we use this vertex to identify $\calA(G,T,F)$ with $X_*(T)\otimes_{\bbZ}\bbR$. Since each affine root subgroup of $G_F$ fixes a half apartment in $ \calA(G,T,F)$, there exists a sufficiently dominant (with respect to the choice of Borel $B$)  special  vertex $\fks'$ which is fixed by $g_u$. It follows that $\fks'$ is fixed by $g$. We write $\widetilde{\calG}$ for the Bruhat--Tits stabilizer scheme over $\calO_F$ corresponding to $\fks'$; by the above discussion we have $g\in\widetilde{\calG}(\calO_F)$. Since $G$ is split over $F$, $\widetilde{\calG}$ is equal to the parahoric group scheme $\calG$ associated to $\fks'$.
\end{proof}

\subsubsection{} We now return to the assumptions and notation of \S \ref{subsec:Mumford-Tate groups}. Thus 
we have an abelian variety $A/\rmE,$ such that $\rho_\ell:\Gamma_{\rmE}\rightarrow \GL(T_\ell A^\vee)$ factors  through $\bfG(\bbQ_\ell)$ for all $\ell$.  Recall $E=\rmE_v$ and $\bbF_q$ is its residue field. The map $i_p:\overline{\bbQ}\rightarrow \overline{\bbQ}_p$ determines an inclusion \begin{equation}\label{eqn: decomposition group}\mathrm{Gal}(\overline{E}/E)\rightarrow \mathrm{Gal}(\overline{\rmE}/\rmE).\end{equation} We let $\widetilde{\sigma}_q\in \Gamma_{\rmE}$ be the image under (\ref{eqn: decomposition group}) of a lift of the geometric Frobenius in $\mathrm{Gal}(\overline{E}/E)$. 	

\begin{prop}\label{prop: reduction to acceptable case}Let $p>2$.
	There exists a totally real  field $\rmF$ such that if $(\bfH,X_{\bfH})$ denotes the Shimura datum of Hodge type coming from the construction in \S\ref{subsubsec: restriction of scalars}, there exists a  very special parahoric group scheme $\calH$ for $H=
	\bfH_{\bbQ_p}$ such that 
	\begin{enumerate}
		\item The image of $\rho_p^{\bfG}(\widetilde{\sigma}_q)$ in $H(\bbQ_p)$ lies in $\calH(\bbZ_p)$.
		
		\item $H'=\bfH'_{\bbQ_p}$ is a product of Weil restrictions of split groups.

	\end{enumerate}
\end{prop}

\begin{proof}Let $G=\bfG_{\bbQ_p}$. By Lemma \ref{lem:elementlevelstr} applied to the element $\rho_p^{\bfG}(\widetilde{\sigma}_q)\in G(\bbQ_p)$, there exists a finite extension $F/\bbQ_p$ such that $G_F$ is split and there exists a special parahoric $\calG$ of $G_F$ such that the image of $\rho_p^{\bfG}(\widetilde{\sigma}_q)$ in $G(F)$ lies in $\calG(\calO_F)$. We let $\rmF$ be a totally real field such that $\rmF_w\cong F$ for all places $w|p$ of $\rmF$. By construction $\bfH\subset\bfH'=\mathrm{Res}_{\rmF/\bbQ}\bfG$ and we have an isomorphism\begin{align*}
	H':=\bfH'_{\bbQ_p}\cong\prod_{w|p}\mathrm{Res}_{\rmF_w/\bbQ_p}\bfG_{\rmF_w}\cong \prod_{w|p}\mathrm{Res}_{F/\bbQ_p}G_F\end{align*}
so that (2) follows.

We let $\calH'$ denote the parahoric group scheme of $H'$ corresponding to $\prod_{w|p}\calG(\calO_F)$. Then $\calH'\cong \prod_{w|p}\mathrm{Res}_{\calO_F/\bbZ_p}\calG$, and since $G$ splits over $F$, $\calH'$ is a very special parahoric. It follows that  $\calH'(\bbZ_p)\cap H(\bbQ_p)$ arises as the $\bbZ_p$-points of a very special parahoric group scheme $\calH$ for $H$. Since $G(\bbQ_p)\subset H(\bbQ_p)$, the image of $\rho_p^{\bfG}(\widetilde{\sigma}_q)$ in $H(\bbQ_p)$ lies in $\calH(\bbZ_p)$ so that  (1) is satisfied. 
	\end{proof}

\subsubsection{}\label{sec: new H R-smooth}In order to apply the results of \S\ref{sec: Frob conj classes} we need to consider a modification of $\bfH$ with connected center. Thus let $T\subset H$ be the centralizer of a maximal $\brQ$-split torus in $H$. Then by an argument as in \cite[Proposition 2.2.4]{Ki2}, we may choose $\bfT$ a maximal torus in $\bfH$ such that $\bfT_{\bbQ_p}$ is $\bfH(\bbQ_p)$ conjugate to $T$ and there exists $h\in X$ such that $h$ factors through $\bfT_{\bbR}$. We let $\bfT^c$ denote the maximal compact subtorus of $\bfT$ which is defined over $\bbQ$. Then $T^c=\bfT_{\bbQ_p}^c$ is a product of induced tori. We set $\bfH_1:=\bfH\times^{\bfZ_{\bfH}}\bfT$ and let $\calH_1$ denote the very special parahoric of $H_1$ associated to $\calH$. We let $X_1$ denote the conjugacy class of Deligne homomorphisms for $\bfH_1$ determined by $h\times 1$ so that $(\bfH_1,X_1)$ is a Shimura datum.

\begin{lemma}\label{lem: main thm}
	\begin{enumerate}\item 	The triple $(\bfH_1,X_1,\calH_1)$ is strongly acceptable.
		\item The inclusion $\bfG\rightarrow \bfH_1$  induces a $\Gal(\overline{\bbQ}/\bbQ)$-equivariant injection $$\Conj_{\bfG}(\overline{\Q})\rightarrow \Conj_{\bfH_1}(\overline{\Q}).$$

	\end{enumerate}
\end{lemma}
\begin{proof}(1) Let $W'=\Hom_{\bfZ_{\bfH}}(W,W)$ ($\bbQ$-linear maps which are $\bfZ_{\bfH}$-equivariant) and we let $\bfH_1$ act on $W'$ via $(h,t)f(x)=hf(t^{-1}x)$. Then as in \cite[Lemma 4.6.22]{KP}, we may equip $W'$ with an alternating form such that we obtain a Hodge embedding $(\bfH_1,X_{1})\rightarrow (\mathbf{GSp}(W'),S'^\pm)$; thus $(\bfH_1,X_{1}) $ is of Hodge type.

Note that $\bfH_1^{\der}=\bfH^{\der}=\bfH'^{\der}$, and $\calH_1$ is a very special parahoric.  Moreover $Z_{H_1}^c\cong T^c$ is a product of induced tori and hence the result follows. 
	
	(2)	We first show that $\bfG\rightarrow \bfH$ induces a $\Gal(\overline{\bbQ}/\bbQ)$-equivariant injection \begin{equation}\label{eqn: injection conj class}\Conj_{\bfG}(\overline{\Q})\rightarrow \Conj_{\bfH}(\overline{\Q}).\end{equation} Let $g,g'\in \bfG(\overline{\Q})$ such that there exists $h\in \bfH(\overline{\Q})$ such that $h^{-1}gh=g$. We consider $\bfH$ as a subgroup of $\bfH'$. Then under the identification $$\bfH'_{\overline{\Q}}\cong \prod_{\iota:\rmF\rightarrow\overline{\Q}}\bfG_{\overline{\Q}},$$ $g,g'$ correspond to the elements $(g,\dotsc,g),(g',\dotsc,g')$ respectively and we write $h=(h_1,\dotsc,h_n)$. Then $h^{-1}gh=g'$ implies $h_1gh_1^{-1}=g'$. Thus $g$ and $g'$ have the same image in $\Conj_{\bfG}(\overline{\Q})$. The $\text{Gal}(\overline{\Q}/\Q)$-equivariance follows from the fact that $\bfG\rightarrow \bfH$ is defined over $\Q$.
	
	Now let $h\in \bfH(\overline{\bbQ})$ and $(h',t)\in \bfH_1(\overline{\bbQ})$ with $h'\in \bfH(\overline{\bbQ})$, $t\in\bfT(\overline{\bbQ})$. Then we have $$(h',t) (h,1) (h',t)^{-1}=(h'hh'^{-1},1),$$and hence $\bfH\rightarrow \bfH_1$ induces a $\Gal(\overline{\bbQ}/\bbQ)$-equivariant injection  $$\Conj_{\bfH}(\overline{\Q})\rightarrow \Conj_{\bfH_1}(\overline{\Q}),$$ and the result follows by composing with \eqref{eqn: injection conj class}.
\end{proof}
\begin{thm}\label{thm: l indep for abelian var}Let $p>2$ be  a prime and $v|p$ a place of $\rmE$ where $A$ has good reduction.
	 Then there exists an element $\gamma\in \Conj_{\bfG}(\Q)$ such that for all $\ell\neq p$, we have $\gamma=\gamma_\ell(v)$ in $\Conj_{\bfG}(\Q_l)$. 
\end{thm}

\begin{remark}
 As remarked above, the group $\bfG$ depends on the embedding $\rmE\hookrightarrow \bbC$ up to inner automorphism. If $\bfG'$ is the group associated to a different embedding  $\rmE\hookrightarrow \bbC$, the inner twisting $\bfG_{\overline{\bbQ}}\cong\bfG'_{\overline{\bbQ}}$ induces a canonical isomorphism $\Conj_{\bfG}\cong \Conj_{\bfG'}$ and it can be checked that the statement of the theorem is independent of the choice of embedding.
\end{remark}

\begin{proof}[Proof of \ref{thm: l indep for abelian var}] We may assume  that $\bfG$ is not a torus as in this case $A$ has complex multiplication and the result is a theorem of Shimura--Taniyama. 
We choose a totally real field $\rmF$ as in Proposition \ref{prop: reduction to acceptable case} and let $(\bfH,X_{\bfH},\calH)$ denote the associated  triple.  By construction,  the image of  $\rho_p^{\bfG}(\widetilde{\sigma}_q)$ inside $\bfH(\bbQ_p)$ lies in $\rmK_p:=\calH(\bbZ_p)$. 
Hence, there exists a finite extension $\rmE'$ of $\rmE$ such that $\rho^{\bfG}_p|_{\Gamma_{\rmE'}}$ factors through $\rmK_p,$ 
and such that there is a prime $v'|v$ of $\rmE'$ such that $\rmE'_{v'}$ has residue field $\bbF_q.$ We may thus replace $\rmE$ by $\rmE',$ without changing the statement of the theorem, and assume that the image of $\rho^{\bfG}_p$ in $\bfH(\bbQ_p)$ factors through $\rmK_p.$  

Now let $(s_{\alpha,\ell})_{\ell \neq p}\in \widehat{V}^p(A)^\otimes$ denote the $\ell$-adic realizations of the absolute Hodge cycles for $A$. By our assumption on $\rmE$, the representation $\rho^p:\Gamma_{\rmE} \rightarrow \GL(\widehat{V}^p(A))$ factors through 
$\bfG(\bbA_f^p) \subset \bfH(\bbA_f^p),$ and hence through a compact open subgroup  $\rmK^p\subset \bfH(\bbA_f^p)$.
Write $\rmK:=\rmK_p\rmK^p$. 

We now define a point of $\Sh_{\rmK}(\bfH,X_{\bfH})$ using the Hodge embedding $$\iota':(\bfH,X_{\bfH})\rightarrow (\mathbf{GSp}(W),S^\pm).$$ Consider the abelian variety up to isogeny 
$A^{\rmF} = A\otimes_{\Q} \rmF$ given by the Serre tensor construction \cite[\S7]{Conrad}, equipped with the isomorphism $\varepsilon: \widehat{V}(A^{\rmF}) \simeq V\otimes_{\bbQ} \bbA_f \otimes_{\bbQ}\rmF$ induced by the identity on $V$. Since $\rho_p^{\bfG}$ and $\rho^p$ act via $\rmK$, the $\rmK$-orbit of $\varepsilon$ is $\Gamma_\rmE$-invariant. Thus, the triple $(A^{\rmF},\lambda\otimes F, \varepsilon),$ 
defines a point $\widetilde{x}_A \in \Sh_{\rmK}(\bfH,X_{\bfH})(\rmE).$ (Note that, since $\psi$ is $\bfH$-invariant, up to scalars, $\lambda$ is defined over $\rmE$ as a weak polarization). 

Now let $(\bfH_1,X_1,\calH_1)$ be the modification of $(\bfH,X_{\bfH},\calH)$ given by the construction in \ref{sec: new H R-smooth}. We set $\rmK_{1,p}=\calH_1(\bbZ_p)$, $\rmK_1=\rmK_{1,p}\rmK_1^p$ where $\rmK_1^p\subset \bfH(\bbA_f^p)$ is a compact open subgroup containing the image of $\rmK^p$. We let  $\widetilde{y}_A\in \Sh_{\rmK_1}(\bfH_1,X_1)(\rmE)$ denote the image of  $\widetilde{x}_A$. By Lemma \ref{lem: main thm} (1), the triple $(\bfH_1,X_{1},\calH_1)$ satisfies the assumptions of Theorem  \ref{thm: l indep full}. Thus we may apply it  to the reduction $y_A \in \scrS_{\rmK}(\bfH_1,X_{1})(\bbF_q)$, where  $\scrS_{\rmK_1}(\bfH_1,X_{1})$ is the integral model constructed from a choice of auxiliary Hodge type Shimura datum. This implies that there exists $\gamma\in \Conj_{\bfH_1}(\Q)$ such that for all $\ell\neq p$, we have $\gamma=\gamma_\ell(v)$ in 
$\Conj_{\bfH_1}(\Q_\ell).$ By Lemma \ref{lem: main thm} (2), it follows that $\gamma\in \Conj_{\bfG}(\Q)$ and $\gamma=\gamma_\ell(v)$ in 
$\Conj_{\bfG}(\Q_\ell).$
\end{proof}

\subsection{Refinements}\label{sec: liftability}

\subsubsection{} We retain the notation introduced above. Write 
$\tilde \gamma_{\ell}(v) = \rho_\ell^\bfG(\mathrm{Fr}_v).$ 
There are a number of ways one might try to refine Theorem \ref{thm: l indep for abelian var}. 
For example one could ask if $\gamma$ lifts to a point $\tilde \gamma \in \bfG(\Q).$ When this is the case, 
one can also try to refine the relationship between $\tilde \gamma_{\ell}(v)$ and $\tilde \gamma.$ 
We will prove such a result when $\bfG$ has simply connected derived group, and is quasi-split at $p.$ 

\subsubsection{} Let $H, H'$ be connected reductive groups over a field $K$ with algebraic closure $\bar K.$ 
Suppose we are given an isomorphism $H \simeq H'$ over $\bar K.$ 
Recall that an {\em inner twisting} between $H$ and $H'$ is an isomorphism $\psi: H \simeq H'$ over $\bar K$ 
such that for $\sigma \in \Gal(\bar K/K),$ there exists $g_{\sigma} \in H'(\bar K)$ so that 
 $\sigma(\psi(h)) = g_{\sigma}\psi(\sigma(h)) g_{\sigma}^{-1}$  for $h \in H(\bar K).$ We say that a subgroup $M \subset H$ 
{\em transfers} to $H'$ via $\psi,$ if  $\psi(M) \subset H'$ is defined 
over $K,$ and $\psi$ induces an isomorphism $M \simeq \psi(M)$ over $K.$ 
We say that $M$ {\em transfers} to $H',$ if it transfers to $H'$ via some $\psi.$ 

An element $h \subset H_{\R}$ is called {\em elliptic} if it is contained in an elliptic maximal torus, that is a maximal torus 
which is anisotropic modulo the center of $H.$ 

\begin{lemma}\label{lem:ellipticrep} $\gamma \in \Conj_{\bfG}(\Q)$ lifts to an elliptic element $\tilde \gamma_{\R} \in \bfG(\R).$
\end{lemma}
\begin{proof} The composite $\GG_m \overset \mu \rightarrow \bfG \overset {c_{\bf G}}\rightarrow \GG_m$ 
is given by $x \mapsto x^i$ for some $i.$ Here, as above, $\mu$ and $c_{\bfG}$ are the Hodge cocharacter and multiplier homomorphism respectively. For any lift $\tilde\gamma \in \bfG(\C)$ of $\gamma,$ we have $c_{\bfG}(\tilde \gamma) = q^i$ 
\cite[2.2.3]{De2}. In particular, $c_{\bfG}(\tilde \gamma)  \in \R^{\times,+}.$ Hence there exists $z \in Z_{\bfG}(\R)$ 
with $c_{\bfG}(z) = c_{\bfG}(\tilde \gamma).$ Set $\gamma_1 = \gamma z^{-1} \in \Conj_{\bfG}(\R).$ 
It suffices to show $\gamma_1$ admits an elliptic lift in $\bfG(\R).$ 

Let $\tilde \gamma_1 \in \bfG(\C)$ be any lift. 
Under any representation of $\bfG$ (for example its canonical symplectic one), the eigenvalues of the image of $\tilde \gamma_1$ 
have absolute value $1.$ Hence $\tilde\gamma_1$ is contained in a maximal compact subgroup of $\bfG(\C).$  
Let $\bar\bfG = \bfG/w(\GG_m),$ and denote by $\tilde \gamma_2 \in \bar \bfG(\C)$ the image of $\tilde \gamma_1$
Then $\tilde \gamma_2$ is contained in a maximal compact subgroup of $\bar\bfG(\C).$ Such a subgroup has the form $\bar\bfG^{c}(\R),$ 
where $\bar\bfG^{c}$ is a real form of  $\bar\bfG.$ Consider the canonical isomorphism 
$\psi:\bar\bfG^{c}_{\C} \simeq \bar\bfG_{\C}.$ 
As the center of $\bar\bfG_{\R}$ is anisotropic, $\psi$ induces an isomorphism between the centers of 
$\bar\bfG$ and $\bar\bfG^c,$ over $\R.$
Moreover, $\bfG^{\der}$ is an inner form of its compact form, so this implies that $\psi$ is an inner twisting. 
Let $T \subset \bar\bfG^{c}_{\R}$ be a maximal torus containing $\psi^{-1}(\tilde \gamma_2).$ 
Then $T$ transfers to $\bar\bfG$ \cite[Lem.~5.6]{LR},
and $\psi^{-1}(\tilde \gamma_2) \in T(\R) \subset \bar\bfG(\R)$ is elliptic. Any lift of $\psi^{-1}(\tilde \gamma_2)$ to $\bfG(\R)$ 
yields the required lift of $\gamma.$
\end{proof}

\begin{cor}\label{cormainthm} With the assumptions of Theorem \ref{thm: l indep for abelian var}, suppose that 
$\bfG^{\der}$ is simply connected and that $\bfG_{\Q_p}$ is quasi-split. 
Then $\gamma$ lifts to an element $\gamma_0 \in \bfG(\Q)$ such that 
\begin{itemize}
\item $\gamma_0 \in \bfG(\R)$ is elliptic
\item $\gamma_0$ is conjugate to $\tilde \gamma_{\ell}(v)$ in $\bfG(\Q_{\ell})$ for all but at most one prime $\ell \neq p.$
\end{itemize}
\end{cor}
\begin{proof} Since $\gamma$ lifts to an elliptic element by Lemma \ref{lem:ellipticrep}, this follows from the argument of 
\cite[p188]{KottwitzAA}.
\end{proof}

\begin{remarks} \label{rem: refinements}
\hspace{.1mm}\\
\vspace{-.3cm}
\begin{enumerate} 
\item When $\bfG$ is {\em not} quasi-split at $p,$ there does not seem to be any reason to believe that $\gamma$ 
in the statement of Theorem \ref{thm: l indep for abelian var} should lift to an element of $\bfG(\Q).$ 
\item When $\bfG$ is quasi-split at $p,$ one expects the conclusion of Corollary \ref{cormainthm} to hold without 
assuming that $\bfG^{\der}$ is simply connected, and without excluding one prime $\ell \neq p.$ Indeed this follows 
when one can show that the isogeny class on the corresponding Shimura variety contains a point which lifts to a special point. 
This is conjectured to hold in general \cite[Conj.~2.3.8]{KPS}. One way to motivate this conjecture would be to prove 
the analogous statement for the admissible morphisms which appear in the Langlands--Rapoport conjecture \cite{LR}. 
This is done in {\em loc.~cit} when the level at $p$ is hyperspecial.

\item It follows from the argument of \cite[p188]{KottwitzAA} that the exceptional prime in the statement of the corollary can actually be chosen in a set of positive density. Of course the choice of this prime affects the choice of $\gamma_0.$
\item It is possible to prove a version of  Theorem \ref{thm: l indep for abelian var} and Corollary \ref{cormainthm} which includes 
$\ell = p,$ using the crystalline Frobenius. We aim to return to this in a future work.
\item In a paper in preparation \cite{KZ2}, 
we extend our methods to prove a version of Theorem \ref{thm: l indep for abelian var}  at a place $v$ of $E$ where $A$ has bad reduction.
This involves an independence of $\ell$ statement for representations of the Weil--Deligne group. 

\end{enumerate}
\end{remarks}

\bibliographystyle{amsalpha}
\bibliography{bibfile}

\end{document}